\documentclass[9pt, amsfonts]{amsart}
\usepackage {amssymb, adjustbox, enumerate, amsbsy, comment}
\usepackage {amsmath}
\usepackage[mathscr]{eucal}
\usepackage{amsthm}
\usepackage{graphicx}
\usepackage{hyperref}
\usepackage {amscd}
\usepackage {epic}
\usepackage[alphabetic]{amsrefs}
\usepackage{adjustbox}
\usepackage{color}
\usepackage{stmaryrd}

\newcommand{\spec}[0]{\operatorname{Spec}}
\newcommand{\red}[0]{\operatorname{red}}

\newcommand{\Aut}[0]{\operatorname{Aut}}

\newcommand{\cX}{{\mathcal{X}}}
\newcommand{\cY}{\mathcal{Y}}

\newcommand{\bB}{\mathbf{B}}
\newcommand{\bD}{\mathbf{D}}
\newcommand{\bchi}{\boldsymbol{\chi}}

\newcommand{\bT}{\mathbf{T}}

\newcommand{\mX}{{\mathcal{X}}}

\newcommand{\cD}{{\mathcal{D}}}

\newcommand{\Supp}{{\rm Supp}}
\newcommand{\sym}{\mathrm{Sym}}
\newcommand{\Ric}{\mathrm{Ric}}

\newtheorem{thm}{Theorem}[section]

\newtheorem{lem}[thm]{Lemma}
\newtheorem{cor}[thm]{Corollary}

\newtheorem{prop}[thm]{Proposition}

\theoremstyle{definition}
\newtheorem{defn}[thm]{Definition}

\newtheorem{exmp}[thm]{Example}

\newtheorem{rem}[thm]{Remark}

\newtheorem{summ}[thm]{Summary}         
\newtheorem*{ack}{Acknowledgments}      

\newtheorem{defn-thm}[thm]{Definition--Theorem}  
\newtheorem{defn-lem}[thm]{Definition--Lemma}  

\newtheorem{assumption}[thm]{Assumption}

\newcommand{\AAA}{\mathbb{A}}
\newcommand{\HH}{\mathbb{H}}
\newcommand{\QQ}{\mathbb{Q}}
\newcommand{\NN}{\mathbb{N}}
\newcommand{\GG}{\mathbb{G}}
\newcommand{\PP}{\mathbb{P}}
\newcommand{\RR}{\mathbb{R}}
\newcommand{\CC}{\mathbb{C}}
\newcommand{\ZZ}{\mathbb{Z}}
\newcommand{\cZ}{\mathcal{Z}}

\newcommand{\rU}{\mathrm{U}}
\newcommand{\sE}{\mathscr{E}}
\newcommand{\sL}{\mathscr{L}}
\newcommand{\sO}{\mathscr{O}}

\newcommand{\cE}{\mathcal{E}}
\newcommand{\cM}{\mathcal{M}}
\newcommand{\cW}{\mathcal{W}}

\newcommand{\fg}{\mathfrak{g}}
\newcommand{\fk}{\mathfrak{k}}
\newcommand{\fp}{\mathfrak{p}}
\newcommand{\fsl}{\mathfrak{sl}}
\newcommand{\reg}{\mathrm{reg}}

\newcommand{\ti}{\tilde}

\newcommand{\aut}{\mathrm{Aut}}

\newcommand{\faut}{\mathfrak{aut}}
\newcommand{\dist}{\mathrm{dist}}

\newcommand{\chow}{\mathrm{Chow}}

\newcommand{\hilb}{\mathrm{Hilb}}
\newcommand{\Hilb}{\mathrm{Hilb}}
\newcommand{\Hom}{\mathrm{Hom}}
\newcommand{\isom}{\mathrm{Isom}}

\newcommand{\proj}{\mathrm{Proj}}
\newcommand{\DF}{\mathrm{DF}}

\newcommand{\GH}{\mathrm{GH}}
\newcommand{\CM}{\mathrm{CM}}
\newcommand{\CH}{\mathrm{CH}}

\newcommand{\KE}{\mathrm{KE}}
\newcommand{\kst}{\mathrm{kst}}
\newcommand{\kps}{\mathrm{kps}}

\newcommand{\FS}{\mathrm{FS}}

\newcommand{\SL}{\mathrm{SL}}
\newcommand{\SP}{\mathrm{sp}}
\newcommand{\topo}{\mathrm{top}}
\newcommand{\tr}{\mathrm{Tr}}
\newcommand{\fuu}{\mathfrak{u}}
\newcommand{\fgg}{\mathfrak{g}}

\newcommand{\nr}{\mathrm{fd}}
\newcommand{\UU}{\mathrm{U}}
\newcommand{\SU}{\mathrm{SU}}
\newcommand{\Red}{}
\newcommand{\Blue}{}

\newcommand{\lam}{\lambda}

\newcommand{\al}{\alpha}
\newcommand{\be}{\beta}
\newcommand{\bbe}{\mathfrak{B}}
\newcommand{\ep}{\epsilon}
\newcommand{\si}{\sigma}
\newcommand{\vep}{\varepsilon}

\newcommand{\vpi}{\varpi}
\newcommand{\de}{\delta}
\newcommand{\De}{\Delta}
\newcommand{\ga}{\gamma}

\newcommand{\pp}{{\prime\prime}}
\newcommand{\ddbar}{\sqrt{-1}\partial\bar\partial}
\newcommand{\la}{\langle}
\newcommand{\ra}{\rangle}
\theoremstyle{remark}
\newtheorem{claim}[thm]{Claim}

\input{xy}
\xyoption{all}

\begin{document}

\title{On the proper moduli spaces of smoothable K\"ahler-Einstein Fano varieties }

\author{Chi Li}
\address{Mathematics Department, Stony Brook University\\ Stony Brook NY, 11794-3651\\ USA}
\address{Current Address: Department of Mathematics, Purdue University\\ West Lafayette, IN 47907-2067\\ USA}
\email{li2285@purdue.edu}
\vspace{.5cm}

\author{Xiaowei Wang}
\address{Department of Mathematics and Computer Science\\
           Rutgers University, Newark NJ 07102-1222\\ USA}
\email{xiaowwan@rutgers.edu}
\vspace{.5cm}

\author{Chenyang Xu}
\address{Beijing International Center of Mathematics Research\\ 5 Yiheyuan Road, Haidian District, Beijing, 100871, China}
\email{cyxu@math.pku.edu.cn}
\address{Current Address: Department of Mathematics\\Massachusetts Institute of Technology\\ Cambridge, MA 02139-4307\\ USA}
\email{cyxu@math.mit.edu}

\date{\today}
\maketitle
\begin{abstract}

In this paper, we investigate the geometry of the orbit space of the closure of the subscheme  parametrizing smooth  Fano K\"ahler-Einstein manifolds inside an appropriate Hilbert scheme.  In particular, we  prove that being K-semistable is a Zariski open condition and establish  the uniqueness for the Gromov-Hausdorff limit for a punctured flat family of Fano K\"ahler-Einstein manifolds. {Based on these,  we  construct a proper scheme parameterizing the S-equivalent classes of $\QQ$-Gorenstein smoothable, K-semistable Fano varieties, and verify various necessary properties  to guarantee that it is a {\em good} moduli space.}
\end{abstract}

\tableofcontents
\section{Introduction}
Constructing  moduli spaces for higher dimensional  algebraic varieties is a fundamental problem in algebraic geometry. For dimension one case, the moduli space parametrizing Deligne-Mumford stable curves was constructed via various kind of methods, e.g. geometric invariant theory (GIT), Teichm\"uler space  quotient by mapping class group, etc.  For higher dimensional case, one of the natural classes to consider is all canonically polarized manifolds, for which GIT machinery  is quite successful  (see \cite{Aubin78, Yau78, Vie83, Don01}). However,  to construct a geometrically natural  compactification for these moduli spaces, the GIT method in its classical form fails to produce that(cf. \cite{WX}), thus people have to develop substitutes. In fact,  it has been quite a while for people to realize what kind of varieties should be included in order to form a {\em proper} moduli (cf. \cite{KSB}). Thanks to the recent breakthrough coming from the theory of minimal model program (see \cite{BCHM} etc.),  one is able to obtain a rather satisfactory theory on proper projective moduli spaces parameterizing  KSBA-stable varieties, named after Koll\'ar-Shepher-Barron-Alexeev (see \cite{Kollar13} for a concise survey of this theory).  We also remark that it is realized later that this compactification should coincide with the compactification from K\"ahler-Einstein metric/K-stability (cf. \cite{Odaka13,WX,BG13}).

As for Fano varieties, the story is much subtler. Apart from some local properties, e.g. having only {\em Kawamata log terminal} (klt) singularities when a Fano variety is assumed to be K-semistable  (cf. \cite{Odaka13} ) and admitting klt Fano degenerations as long as a general fiber is a klt Fano variety in a one parameter family (cf. \cite{LX}), it is still not clear what kind of general Fano varieties we should parametrize in order for us to obtain a nicely behaved moduli space, especially if we aim to find a compact Hausdorff one, and how to construct it.  The recent breakthrough in K\"ahler-Einstein problem, namely the solution to the Yau-Tian-Donaldson Conjecture (\cite{CDS1, CDS2, CDS3} and \cite{Tian2014}) is a major step forward, especially for understanding those Fano manifolds with K\"ahler-Einstein metrics. Furthermore, it implies that the right limits of smooth  K\"ahler-Einstein manifolds form a bounded family.  In this paper, we aim to use the analytic results they established to investigate the geometry of the compact space of orbits which is the closure of the  space parametrizing smooth Fano varieties.

\subsection{Main results}

Our first main result of this paper is  the following:
\begin{thm}\label{main}Let $\cX\to C$ be a flat family of projective varieties over a pointed smooth curve
$(C,0)$ with $0\in C$. Suppose
\begin{enumerate}
\item $K_\cX$ is $\QQ$-Cartier and $ -K_{\cX/C}$ is relatively ample over $C$;
\item\label{sm}for any $t\in C^\circ:=C\setminus \{0\}$, $\cX_t$ is smooth and $\cX_0$ is klt; 
\item\label{Kst} $\cX_0$ is {K-polystable}.
\end{enumerate}
Then
\begin{enumerate}[(i)]
\item  there is a Zariski open neighborhood $U$ of $0\in C$ on which $\cX_t$ is K-semistable for all $t\in U$, and  K-stable if we assume further  $\cX_0$ has a discrete automorphism group;
\item  for any other flat projective family $\cX'\to C$ satisfying (1)-(3) as above and
$$
\cX'\times_C C^\circ\cong \cX\times_C C^\circ,
$$
we can conclude  $\cX'_0\cong \cX_0$;
\item  ${\cX_0}$ admits a weak K\"ahler-Einstein metric $\omega_\KE(\cX_0)$. Moreover, if we assume further that $\cX_t$ is K-polystable for all $t\in C^\circ$, then $(\cX_0,\omega_\KE(\cX_0))$ 
is the Gromov-Hausdorff limit of a family $\{(\cX_t,\omega_\KE(\cX_t)\}_{t\in C^\circ}$ as $t\to 0$ where $\omega_\KE(\cX_t)$ is a K\"ahler-Einstein metric on $\cX_t$ for each $t\in C^\circ$.
\end{enumerate}
\end{thm}

If both $\cX_0$ and $\cX'_0$ are assumed to be {\em smooth} K\"ahler-Einstein manifolds then part of Theorem \ref{main} is a consequence of the work \cite{Sz2010}, where the more general case for arbitrary polarization is established. When the fiber is of dimension 2, this is also implied by the work of \cite{Tian1990, OSS} as explicit compactifications of K\"ahler-Einstein Del Pezzo surfaces are constructed there.  We remark that the Zariski openness has already been established in \cite{Don2013,Odaka14} when the fibers are Fano K\"ahler-Einstein {\em manifolds} with {\em discrete} automorphism.  Finally, we remark that there is an independent work in \cite{SSY} obtaining similar results along this line, see Remark \ref{r-history}.

\bigskip

Now let us give a brief account of our approach to Theorem \ref{main}. First we note that although  part of our theorem is stated in algebro-geometric  terms, the proof indeed relies heavily on known analytic results, especially the recent work in \cite{CDS2, CDS3, Tian2014}. On the other hand, we remark that no  further analytic tools are developed beyond their work in our paper. So our argument is actually more of an algebro-geometric nature.

The first main tool for us is a continuity method very similar to the one proposed by Donaldson in \cite{Don2011}. Indeed, by throwing in an auxiliary divisor $\cD\in |-mK_\cX|$, we consider the following log extension of Theorem \ref{main}.

\begin{thm}\label{log-main}For a fixed $\be\in [0,1]$, let $\cX\to C$ be a flat family over a pointed smooth curve
$(C,0)$ with a relative codimension one cycle $\cD$ over $C$. Suppose
\begin{enumerate}
\item  $-K_{\cX/C}$ is ample and $\cD\sim_{C} -mK_{\cX/C}$ for some positive integer $m>1$;
\item\label{l-sm}for any $t\in C^\circ:=C\setminus \{0\}$, $\cX_t$ and $\cD_t$  are smooth, $(\cX_0,\frac{1}{m}\cD_0)$ is klt;
\item\label{l-Kst}  $(\cX_0, \cD_0)$ is $\be$-K-polystable. (cf. Definition \ref{be-K}).
\end{enumerate}
Then
\begin{enumerate}[(i)]
\item  there is a Zariski neighborhood  $U$ of  $0\in C$,  on which $(\cX_t, \cD_t)$ is $\be$-K-semistable (in fact $\be$-K-polystable if $\be<1$) for all $t\in U$;
\item  for any other flat projective family  $(\cX',\cD')\to C$  with a relative codimension one cycle $\cD'$  satisfying (1)-(3) as above and
$$
(\cX',\cD')\times_C C^\circ \cong (\cX,\cD)\times_C C^\circ,
$$
we can conclude $(\cX'_0,\cD'_0)\cong (\cX_0,\cD_0)$;
\item ${(\cX_0,\cD_0)}$ admits a conical weak K\"ahler-Einstein metric with cone angle $2\pi(1-(1-\be)/m)$ along $\cD_0$, which is the Gromov-Hausdorff limit of $(\cX_{t_i},\cD_{t_i})$ endowed with the conical K\"ahler-Einstein metric with cone angle $2\pi (1-(1-\be_i)/m)$ along $\cD_{t_i}\subset \cX_{t_i}$ for any sequence $t_i\to 0$ and ${\beta_i\nearrow \beta}$.
\end{enumerate}
 \end{thm}

To prove Theorem \ref{log-main},  one notices that the uniqueness is  well-understood  when the angle is small.  We give an account to this fact using a completely algebro-geometric means. To be precise, we use the result that the set of log canonical thresholds satisfies ascending chain condition (ACC) (see \cite{HMX2014}) to show that when the angle $\beta$ is smaller than a positive number $\beta_0>0$ there is only one extension with at worst  klt singularities. Fix $\epsilon$, such that $0<\epsilon<\beta_0$. We define a set $\bB\subset [\ep,1]$ (cf. Section \ref{ss-continuity} for the precise definition) for which the conclusions of Theorem \ref{log-main} hold for the angles belonging to the set $\bB$. The result on small angle case implies $\bB\supset [\epsilon,\beta_0]$.

Now to prove Theorem \ref{main}, let us  first {\it assume} that all  the nearby fibers $\cX_t$ are  K-semistable. Then it suffices to show that $\bB$ is open and closed in $[\ep,1)$. We establish them using two facts. First we prove a simple but very useful fact (see Lemma \ref{disj}), which says that for a point $p$ on the {\em limiting orbit} with {\em reductive stabilizer}, there is  a Zariski open neighborhood $p\in U$ such that the closure of the ${\rm SL}(N+1)$-orbit of any point in the limiting orbit near $p$  actually contains $g\cdot p$ for some $g\in {\rm SL}(N+1)$. In particular, it guarantees that there is no nearby non-equivalent K-polystable points on the limiting orbit. With this, using a crucial {\em Intermediate Value Theorem} type of results (cf. Lemma \ref{IVT}), we show that if there is a different limit, which a priori could be far away from the  given central fiber in the parametrizing Hilbert scheme, then we can indeed always find another limit which either specializes to $(X_0,D_0)$  in a test configuration or becomes the central fiber of a test configuration of $(X_0,D_0)$, violating the K-stability assumption. Similarly, this argument can also  be applied to study the case when $\beta\nearrow 1$.

To finish the proof, we need to verify the assumption that all the nearby fibers $\cX_t$ are {\em K-semistable}. For this, one needs  two observations. First, it follows from the work of \cite{CDS2, CDS3, Tian2014} that to check K-semistability of $\cX_t,\ t\ne 0$, it suffices to test for all  one-parameter-group (1-PS) degenerations in a fixed $\mathbb{P^N}$. Second, it follows from a straightforward GIT argument that {\em K-semistable threshold} (kst) (cf. Section \ref{ss-zopen}) is a constructible function. So what remains to show is that it is also lower semi-continuous (which is also observed in \cite{SSY}), but this is a consequence of the upper semi-continuity of the dimension of the automorphism groups and the continuity method deployed in the proof of Theorem \ref{log-main}.

With all this knowledge in hand,  we are able to achieve the main goal of this paper, i.e. constructing a proper {\em good} moduli space for all $\QQ$-Gorenstein smoothable K-semistable Fano varieties.
\begin{thm}\label{t-good}
For $N\gg 0$, let $Z^*$ be the semi-normalization of the locus inside $\hilb_\chi(\mathbb{P}^N)$ parametrizing all $\QQ$-Gorenstein smoothable K-semistable Fano varieties in $\PP^N$ with fixed Hilbert polynomial $\chi$ (see Section \ref{ss-luna} for the precise definition of $Z^*$). Then the algebraic stack
$[Z^*/{\rm SL}(N+1)]$ admits a proper semi-normal {\em scheme} $\mathcal{KF}_N$ as its \em good moduli space (in the sense of \cite{Alp13}). Furthermore, for sufficiently large $N$, $\mathcal{KF}_N$ does not depend on $N$.
\end{thm}
Recall  from \cite[Section 1.2]{Alp13} that a quasi-compact morphism $\phi: \cZ \longrightarrow M$ from an Artin stack  $\cZ$ to an algebraic space $M$ is a {\em good moduli space} if
\begin{enumerate}
\item  The push-forward functor on quasi-coherent sheaves is exact,
\item The induced morphism on sheaves $\sO_M \rightarrow \phi_\ast\sO_\cZ$ is an isomorphism.
\end{enumerate}
This concept is a generalization of {\em good quotient} in the classic GIT. In more concrete terms,  Theorem \ref{t-good} says that each  ${\rm SL}(N+1)$-orbit inside $Z^*$ corresponds to a $\QQ$-Gorenstein smoothable K-semistable $\mathbb{Q}$-Fano variety, and $Z^\ast$ admits a categorical quotient $\mathcal{ KF}_N$, whose points correspond to the $S$-equivalence (i.e., the equivalence relation generated by the orbital closure inclusion) classes of  ${\rm SL}(N+1)$-action on $Z^*$. In particular, the set of $\CC$-points in $\mathcal{KF}_N$ precisely corresponds to set of {\em closed minimal} ${\rm SL}(N+1)$-orbits in $Z^\ast$, i.e., the set of $\QQ$-Gorenstein smoothable K-polystable $\QQ$-Fano varieties over $\CC$.

The existence of a moduli space for K\"ahler-Einstein Fano manifolds is well expected  after the work of  \cite{Tian1990}.   A local quotient picture was suggested in \cite[Section 5.3]{Don2008} and \cite{Sz2010}, and was explicitly conjectured in \cite[Secion 1.3 and 1.4]{Sp2012} and \cite[Conjecture 6.2]{OSS}. Furthermore, the moduli space is speculated to be
projective by the existence of the descending of the {\rm CM}-line bundle (see e.g. \cite{PaTi06} and \cite{OSS}).  We also remark that for smooth K\" ahler-Einstein Fano manifolds with  discrete automorphism which  are known to be asymptotically Chow stable by  \cite{Don01}, they admit (possibly non-proper) algebraic moduli spaces thanks to the work of \cite{Don2013} and \cite{Odaka14}.

Now let us explain  our approach to Theorem \ref{t-good}. Due to the lack of a {\em global GIT } interpretation of the K-stability, { our  strategy is to replace GIT by the work of \cite{AFSV14}. So to obtain a good quotient, one needs to verify  all the assumptions of {\cite[Theorem 1.2]{AFSV14}}. In particular,
among other things one needs to establish the following two key properties:}
\begin{enumerate}
\item the stabilizer preserving condition for the local presentation of the moduli stack;
\item  the affineness of the quotient morphism.
\end{enumerate}
Intuitively,
the first property implies that the local Zarski open charts of the moduli space can be glued together; while the second property guarantees  that the  local charts constructed above are actually  {\em affine}. {Moreover, the second property guarantees the {\em goodness} of the quotient $[Z^*/{\rm SL}(N+1)]\to \mathcal{KF}_N$, and as a consequence the restriction of CM line bundle to $Z^*\subset \hilb$ can be descent to the good moduli space. This  will be crucial in our study of the projectivity of the moduli space $\mathcal{KF}$ in  \cite{LWX2018}}. We single out these two properties as they depend on the existence of a {\em global  proper}  topological (equipped with Gromov-Hausdorff topology) moduli space in a essential way.

We remark  that both properties follow from the famous Luna's \'etale slice theorem for a reductive group $G$-acting on an {\em affine} variety $Z$,  that is, if $z\in Z$ and  the  $G$-orbit $G\cdot z\subset Z$ is {\em closed} then there is a nice slice containing $z$ satisfying the above two properties. Unfortunately,  we are unable to verify the {\em affiness} assumption of  Luna's theorem since there is no global GIT interpretation of K-stability, but the closedness of $G\cdot x$ in an affine variety will be a  {\em consequence} of our proof  instead,  which is based on  the existence  of a nice {\em continuous proper slice} (although non-algebraic) lying over the stack. The slice is obtained via a family version of Tian's embeddings of K\"ahler-Einstein Fano varieties and its  properness follows from Theorem \ref{main}. The slice can be regarded as an alternative to the zero set of the moment map in the classical Kempf-Ness-Kirwan picture.

\bigskip

Finally we close the introduction by outlining  the plan of the paper. In Section \ref{s-pre}, we give the basic definitions.   In Section \ref{ggp}, we review some  facts on the linear action of a reductive group on a projective space. In Section \ref{GH-cont}, we list the main analytic results we need in this note. First we recall the recent results appeared in \cite{CDS2, CDS3, Tian2014}. Then we also state the Gromov-Hausdorff continuity for conical K\"ahler-Einstein metrics on a smooth family of Fano pairs (see \cite{CDS2, CDS3, Tian2014}).  In Section \ref{uni}, we prove that when the angle is small enough, the  filling is always unique.  In Section \ref{ss-continuity}, we establish the main technical tool of our argument, which is a continuity theorem. We remark, with it we can already show Theorem \ref{log-main} under the assumption that the nearby fibers are all $\be$-K-polystable.   In Section \ref{s-K-semistable}, we will  prove the K-semi-stability of the nearby points by applying the continuity method. First in Section \ref{K-semistable} we prove Theorem \ref{ss} which says that any orbit closure of a K-semistable Fano manifold contains only one isomorphic class of K-polystable $\QQ$-Fano variety. In particular, this is an extension of the result of \cite{CS2014} for the Fano case. In Section \ref{ss-zopen}, we show that a smoothing of a K-semistable $\mathbb{Q}$-Fano variety  is always K-semistable. In Section \ref{ss-proof}, by puting all the results together, we finish the proof of  Theorem \ref{main} and \ref{log-main}. In Section \ref{ss-luna}, we apply our results and prove a Luna slice  type theorem for K-stability, which is used to establish Theorem \ref{t-good}. In Section \ref{s-apendix}, we will discuss several technical results that are needed on the general theory of linear action of a reductive group  on projective space.

\bigskip

 \begin{rem}[{Remarks on the history}]\label{r-history}
This paper was original titled as {\it `Degeneration of Fano K\"ahler-Einstein manifolds'} (see  \cite{LWX2014}). In the first version, we {have} established the separateness of the moduli space and proved the uniqueness of K-polystable degeneration for K-semistable Fano manifolds.  After it was posted on the arXiv, we were informed by the authors of \cite{SSY} who independently investigated similar questions with a circle of parallel ideas but  in a more analytic fashion and obtained results which are closely related. In particular, in \cite{SSY}, the authors obtained {\em first} the existence of weak K\"ahler-Einstein metrics on $\QQ$-Gorenstein smoothable K-polystable  $\QQ$-Fano varieties; the {\em analytic} openness of K-stability under the assumption  of {\em finite} automorphism group; the lower semi-continuity of the cone angle for conical K\"ahler-Einstein metrics.  Those statements are not included in the first version of our preprint. As a consequence, the uniqueness of K-stable filling  with {\em finite} automorphism group was also independently obtained in \cite{SSY}. After the appearance of \cite{SSY} on the arXiv, we realize  that the approach in the first version of our paper  can be naturally  extended and give  a complete picture as in the current version. We would like to thank the authors of \cite{SSY} for communicating their work to us.
After we had posted the second version of our paper on the arXiv, we were contacted by Odaka, who claimed (see \cite{Odaka14a}) to have independently obtained part of Theorem \ref{t-good}  based on the work of \cite{LWX2014}  and \cite{SSY}.
\end{rem}


\begin{ack}
 The first author is partially supported by NSF: DMS-1405936. The second author is partially supported by a Collaboration Grants for Mathematicians from Simons Foundation:281299 and NSF:DMS-1609335, and
he also wants to thank Professor D.H. Phong, Jacob Sturm and Jian Song for their constant encouragement over the years.  The third author is partially supported by the grant `The Recruitment Program of Global Experts'. We are very grateful of Jacob Sturm for many valuable suggestions and comments. We also would like to thank Jarod Alper, Daniel Greb, Reyer Sjamaar and Chris Woodward for helpful comments. We are indebted to the anonymous referee{s} for numerous useful suggestions. A large part of this work was done when CX visits the Institute for Advanced Study, which is partially sponsored by Ky Fan and Yu-Fen Fan Membership Funds, S.S. Chern Foundation and NSF: DMS-1128155, 1252158.
\end{ack}
\section{Preliminaries}\label{s-pre}
In this section, we will fix our convention of the paper. The definition of K-stability (resp. $\beta$-K-stability) below are recalled  from  \cite{Tian1997,Don00} (resp. \cite{Don2011}). The readers may also consult the lecture notes \cite{PS2010} and \cite{Th2006} for both an analytic and an algebro-geometric  point of view.
\begin{defn}\label{log-tc}
Let $(X,D;L)$ be an $n$-dimensional projective variety polarized by an ample line bundle $L$ together with an effective 
 divisor $D\subset X$. A {\em log test configuration} of $(X,D;L)$ consists of
\begin{enumerate}
\item A projective flat morphism  $\pi:(\cX;\sL)\to \AAA^1$ and an effective divisor $\cD$ on $\cX$ such that $\Supp(\cD)$ does not contain any component of the central fiber $\cX_0$;
\item A $\GG_m$-action on $(\cX,\cD;\sL)$, such that $\pi$ is $\GG_m$-equivariant with respect to the standard $\GG_m$-action on $\AAA^1$ via multiplication;
\item $\sL$ is relative ample and we have $\GG_m$-equivariant isomorphism.
\begin{equation}\label{X0}
(\cX^\circ,\cD^\circ; \sL|_{\cX^\circ})\cong (X\times \GG_m, D\times\GG_m; \pi_X^\ast L)
\end{equation}
where $(\cX^\circ,\cD^\circ)=(\cX,\cD)\times_{\AAA^1}\GG_m$ and $\pi_X:X\times \GG_m\to X$.
\end{enumerate}
{A log test configuration is called a {\em product} test configuration if $(\cX,\cD;\sL)\cong (X\times \AAA^1,D\times \AAA^1; \pi_X^\ast L)$ where $\pi_X:X\times \AAA^1\to X$,  and a {\em trivial} test configuration if $\pi:(\cX,\cD;\sL)\to \AAA^1$ is a product test configuration with $\GG_m$ acting trivially on $X$.}
\end{defn}

Assume $X$ to be normal. Let  $\chi$ denote the Hilbert polynomial of $(X,L)$ and we introduce $a_i, \ti a_i, b_i, \ti b_i\in \QQ$ via  the following expansions.
\begin{itemize}
\item $\chi(X,L^{\otimes k}):=\dim H^0(X,L^k)=a_0 k^n+a_1 k^{n-1}+O(k^{n-2})$;
\item $\chi(D,(L|_D)^{\otimes k}):=\dim H^0(D,L^k|_D)=\ti a_0 k^{n-1}+O(k^{n-2})$;
\item $w(k):=$ weight of $\GG_m$-action on $\wedge^\topo H^0(\cX_0,\sL^{\otimes k}|_{\cX_0})=b_0k^{n+1}+b_1k^n+O(k^{n-1})$;
\item $\ti w(k):=$ weight of $\GG_m$-action on $\wedge^\topo H^0(\cD_0,\sL^{\otimes k}|_{\cD_0})=\ti b_0k^{n}+O(k^{n-1})$.
\end{itemize}

In this article, we will focus on the projective pairs $(X,D)$ introduced in Definition \ref{log-tc} satisfying  that the divisor $D$ is prime and $(X,\frac1m D)$ is a projective pair with  {\em Kawamata log terminal} (klt) singularities (see \cite[2.34]{KM98}) for a given positive integer $m$.
\begin{defn} We call a projective klt pair $(X,D)$ to be {\em a log Fano pair} if $-(K_X+D)$ is an ample  $\QQ$-Cartier divisor  and {\em a $\QQ$-Fano variety} if $D=0$.
\end{defn}

Now we are ready to  state the algebro-geometric criterion for the existence of  conical K\"ahler-Einstein metric  on a Fano manifold $X$  with cone angle $2\pi (1-(1-\be)/m)$ for  $\be\in(0,1]$ along a divisor $D\in |-mK_X|$.
\begin{defn}\label{be-K}
For a  $\QQ$-Fano variety $X$ with $D\in |-mK_X|$  and a real number $\beta\in [0,1]$, we define the {\em log generalized Futaki invariant with the angle $\beta$} as following:
\begin{eqnarray*}
\DF_{1-\be}(\cX,\cD;\sL)
&=&\DF(\cX;\sL)+(1-\beta)\cdot \CH(\cX,\cD;\sL)\\
\end{eqnarray*}
with
$$ \DF(\cX;\sL):=\frac{a_1b_0-a_0b_1}{a_0^2}\text{ and } \CH(\cX,\cD;\sL):=\frac{1}{m}\cdot \frac{a_0\ti b_0-b_0\ti a_0}{2a_0^2}\text{ (cf. \cite[ Definition 3.3]{LS2014} )}\ .$$
Then $$\DF_{1-\be}(\cX,\cD;\sL^{\otimes r})=\DF_{1-\be}(\cX,\cD;\sL)\ .$$
We say $(X,D;L)$ is called {\em $\be$-K-semistable} if $\DF_{1-\be}(\cX,\cD;\sL)\geq 0$ for any  normal test configuration $(\cX,\cD;\sL)$, and {\em $\be$-K-polystable} (resp. $\beta$-K-stable) if it is $\be$-K-semistable with $\DF_{1-\be}(\cX,\cD;\sL)=0$ if and only if $(\cX,\cD;\sL)$ is a {product test configuration (resp. trivial test configuration)}.
\end{defn}

Thanks to the linear dependence of
$\DF_{1-\beta}(\cX,\cD; \sL)$  on $\beta$, we immediately obtain the following interpolation property:
\begin{lem}\label{interpolate}
If $(X,D; L)$ is both $\beta_1$-K-semistable and $\beta_2$-K-polystable with $\beta_1<\beta_2$ (resp. $\beta_2<\beta_1$), then $(X, D; L)$ is $\beta$-K-polystable for any
$\beta\in (\beta_1, \beta_2]$ (resp. $\beta\in [\beta_2, \beta_1)$).
\end{lem}

\begin{rem}Notice that if for $(X,D;K_X^{\otimes(-r)})$ where $X$ is a $\QQ$-Fano variety with $D\in|-mK_X|$,
$$\lam:\GG_m\to {\rm SL}(N_r+1) {\text{ with } N_r+1:=\dim H^0(X,K_X^{\otimes(-r)})}$$ induces a test configuration $(\cX,\cD;\sL)$, then
\begin{equation}\label{CH}
\CH(\cX,\cD;\sL)=\frac{1}{2mr^n(-K_X)^n}\cdot\left(\CH(\cD_0)-\frac{{n}m}{(n+1)}r\CH(\cX_0)\right)
\end{equation}
with $\CH(\cD_0)$ and $\CH(\cX_0)$ being precisely the $\lam$-weight for the Chow points of
$\cD_0,\cX_0\subset \PP^{N_r}.$
\end{rem}

\section{Linear action of reductive groups on projective spaces}\label{ggp}
In this section, we prove a basic fact on a reductive group acting on $\PP^M$, which will be crucial  for the later argument.
Let $G$ be a reductive  algebraic group acting on $\PP^M$ via a  rational representation $\rho:G\to {\rm SL}(M+1)$  and  $z:C\to \PP^M$ be an algebraic morphism satisfying $z(0)=z_0\in\PP^M$ where $(0\in C)$ is a smooth pointed curve germ. Let
$$ \overline{ BO}:=\lim_{t\to 0} \overline{O_{z(t)}}$$
with  $O_{z(t)}:=G\cdot z(t)$  and $ \overline{O_{z(t)}}\subset \PP^M $ be its closure, {that is, $\overline{BO}$ is a union of (broken) orbits that $\overline{O_{z(t)}}$ specialized to.}

\begin{lem}\label{disj}
Suppose  $G_{z_0}< G$, the  stabilizer  of $z_0\in \PP^M$ for the $G$-action on $\PP^M$, is reductive.  Then there is a $G$-invariant Zariski open neighbourhood of  $z_0\in U\subset \PP^M$ satisfying:
\begin{equation}\label{Oz}
\overline{BO}\cap U=\bigcup_{ \stackrel{O_p\subset \overline{BO}\   }{O_{z_0}\cap \overline{O_p}\ne \varnothing}} O_p\cap U\text{ where } O_p :=G\cdot p\subset \overline{BO},
\end{equation}
i.e. the closure of the $G$-orbit of any point in $\overline{BO}$ near $z_0$  contains $g\cdot z_0$ for some (hence for all)  $g\in G$.
We will call $O_{z_0}$ a {\em minimal} orbit.
\end{lem}
\begin{proof}
We divide the proof into two steps:

{\em Step 1: $G=G_{z_0}$.}
The representation  $\rho:G\to {\rm SL}(M+1)$  induces a  $G$-linearization  of  $\sO_{\PP^M}(1)\to\PP^M$. Let $\rho_0:G\to \GG_m$ be the character of the resulting $G$-action on $\sO_{\PP^M}(1)|_{z_0}$, since $z_0$ is fixed by $G$. Then $z_0$ is GIT polystable with respect the linearization of $\sO_{\PP^M}(1)$ induced by the representation $\rho\otimes \rho_0^{-1}:G\to {\rm SL}(M+1)$.
It follows from the  construction in classic GIT that  the semistable locus $z_0\in U:=(\PP^M)^{\rm ss}\subset \PP^M$ is  $G$-invariant and Zariski open.
To see that $U$ serves our purpose, it suffices to notice that $G\cdot z_0$ is the {\em unique} polystable orbit in $(\PP^M)^{\rm ss}\cap\overline{BO}$ and for any $z\in \overline{BO}\cap U$, $O_{z_0}\subset \overline{G\cdot z}$, which follows from the classical result of Kempf-Ness \cite[proof of Theorem 8.3 ]{MFK}.

{\em Step 2: $G>G_{z_0}$.} Since $G_{z_0}$ is reductive, we have a decomposition of its Lie algebra
$$\mathrm{Lie}(G)=\fg=\fg_{z_0}\oplus \fp$$
 as representations of $G_{z_0}$. The infinitesimal action of $G$ at $0\ne\hat z_0\in \CC^{M+1}$, a lifting of $z_0\in \PP^M$, induces a $G_{z_0}$-invariant decomposition  $\CC^{M+1}=\CC\cdot \hat z_0\oplus W'\oplus \fp$. By the proof of \cite[Proposition 1]{Don2012},
$$\PP W=\PP(W'\oplus \CC\hat z_0)\subset \PP^M$$ satisfies the following properties:
\begin{enumerate}
\item $z_0\in \PP W$ and is preserved by $G_{z_0}$;
\item $\PP W$ is transversal to the $G$-orbit of $z_0$ at $z_0$;
\item \label{tv}for $w\in \PP W$ near $z_0$ and $\xi\in \fg:={\rm Lie}(G)$,  if we let $\sigma_w: \fg\to T_w\PP^M$ denote the infinitesimal action of $G$ then
$$\sigma_w(\xi)\in T_w\PP W \iff \xi\in \fg_{z_0}:={\rm Lie}(G_{z_0})\ .$$
\end{enumerate}
 In particular, part \eqref{tv} implies that there exists a  Zariski open neighborhood $U_0\subset \PP W$ of $z_0$ such that the infinitesimal action induced by  $\fp^\perp$ on $\PP W$ is {\em transversal} for all points in $U_0$  (cf. Lemma \ref{G0-inv}).

\begin{claim}\label{V-i}
 Let $S:=G\cdot \mathrm{Im}z\subset \PP^N$ and $H$ be the identity component of $G_{z_0}$. Then
there is a Zariski open subset $U_W\subset U_0\subset \PP W$ and a {\em finite} collection of pointed arcs $\{z^i: (C_i,0)\to (U_0,z_0)\}_{i=0}^d$ with $z^0=z:C\to \PP W$  such that
$$\overline S\cap U_W=\bigcup_{i=0}^d \overline{ O(H,z^i)}\cap U_W \text{ with } O(H,z^i):=H\cdot \mathrm{Im}z^i\subset \PP W.$$
\end{claim}

Assume Claim \ref{V-i} for the moment, let us define

$$\overline{BO^W_i}:= \lim_{t\to 0}\overline{O_{z^i(t)}^W}\  \text{ with } O_{z^i(t)}^W:=H\cdot z^i(t)\subset \overline{O(H,z^i)}\subset \PP W.
$$
Next for each $0\leq i\leq d$,  applying {\em Step 1} to the $H$-action on $\PP W$  and $\overline{{BO^W_i}}\subset\PP W$,
 we obtain an ${H}$-invariant Zariski open $z_0\in U'_i\subset\PP W$ such that
   $$\forall p\in U'_i\cap\overline{{BO_{i}^W}}\Longrightarrow z_0\in \overline{G\cdot p}\ .$$
Then
  $U={G\cdot}\left(\displaystyle\bigcap_{i=0}^d  U'_i\right)$ is the $G$-invariant Zariski open set we want. {In fact, to see $U$ is Zariski open, one  first notice that $\displaystyle\bigcap_{i=0}^d  U'_i$ is Zariski open as each of $U'_i\subset \PP W$ is so for all $i$, hence  $U$  is constructible by  Chevalley's Lemma \cite[Chapter II, Exercise 3.19]{Har77}.  On the other hand, $U$ is also open in $\PP^M$ with respect to the analytic topology. This follows from the fact that the $\fg_{z_0}^\perp$-action on $\PP W$ is  transversal  (cf. Lemma \ref{G0-inv}) and $\forall g\in G<\SL(M+1)$ is an automorphism of $\PP^M$.  Being constructible and analytically open implies $U$ is  Zariski open in $\PP^M$.}

\bigskip

Now let us proceed to the proof of Claim \ref{V-i}. To better illustrate the picture, let us treat the case $\dim G_{z_0}=0$ first.

{\em Case 1: $\dim G_{z_0}=0$}. Let us consider the variety $S:=G\cdot \mathrm{Im} z\subset \PP^M$ and let $\partial S:=\overline{S}\setminus S$. Then there is an open neighborhood $U_W\subset U_0$ such that $\overline{S}\cap U_W$ has only finitely many
 irreducible components. Let us write
$$\overline S\cap U_W=\bigcup_{i=0}^d C_i$$
with $C_0=\mathrm{Im}z(C)$ and $C_i$ are irreducible components passing through $z_0$.

Since $\partial \overline S\cap C_i$ is constructible, after a possible shrinking of $C_i$ we have  two possibilities:
\begin{enumerate}
\item $\partial\overline S\cap C_i=C_i$
\item  $\partial\overline S\cap C_i =\varnothing$ or $z_0$.
\end{enumerate}
We claim that the first case does not happen and then by choosing the arc  $z^i: (C_i,0)\to (U_0,z_0)$ we establish Claim \ref{V-i}.  To prove our claim, one notices there are  two kinds of points on the boundary $\partial\overline S$:
\begin{itemize}
\item {\em first kind}:  a boundary point of $\overline{G\cdot z(t)}$ for a fixed $t$;
\item {\em second kind}: all the remaining points on $\partial\overline S$.
 \end{itemize}
 Notice that the set of both kinds of points  form constructible sets. Any boundary point of the first kind  can be indeed written as a limit of points in $G\cdot z(t)\cap U_W$ for a fixed $t$, but this is absurd as $G$ acts on $U_0$  transversally. So we may assume all the points on $C_i$ are of the {\em second kind}, this implies that
 $$\mathrm{Im}z\not\subset \overline{G\cdot z(t)}\text{ for a fixed }t\in C.$$
  In particular, we have $\dim G+1= \dim \overline S$ as $\dim G_{z_0}=0$. Since $\partial \overline S$ is $G$-invariant, we have $G\cdot C_i\subset \partial \overline S$. Now let us consider the $G$-action on $z\in C_i$, which implies that the
 $$\dim \partial\overline S\geq \dim G+\dim C_i=\dim G+1= \dim \overline S,
  $$
  a contradiction. Thus our claim is verified.


\bigskip

{\em Case 2: the general case.} Let us consider the variety $S:=G\cdot \mathrm{Im} z\subset \PP^M$ and let $\partial S:=\overline{S}\setminus S$.  Then there is an $H$-invariant open neighborhood $U_W\subset U_0$ such that $\overline{S}\cap U_W$ has only finitely many
{\em irreducible} components, which are denoted by
$$\overline S\cap U_W=\bigcup_{i=0}^d V_i$$
with $V_0=\overline{O(H,z)}$ and $z_0\in V_i, 0\leq i\leq d$. Moreover, $V_i$ is $H$-invariant for each $i$ since $\overline S$ is.

Then Claim \ref{V-i} amounts to saying that for each $i$, there is an arc $z^i: C_i\to  U_0$ such that
$$V_i=\overline{ O(H,z^i)}\cap U_W.$$
To find such an arc, all we need is  a {\em  general } $v\in V_i$
satisfying

\begin{equation}\label{Hv}
\dim H\cdot v+1\geq \dim V_i,
\end{equation}
since that implies  two situations: either $\dim H\cdot v<\dim V_i$ for which we choose $z^i:C_i\to V_i$  be an arc joining $z_0$ and $v$ so that $\mathrm{Im}z^i\not \subset \overline{H\cdot v}$; or $\dim H\cdot v=\dim V_i$ for which
we choose any nonconstant arc $z^i:C_i\to V_i$ satisfying $z^i(0) =z_0$.
Then $\dim V_i=\dim O(H,z^i)$ and our Claim is justified.

To find such $v\in V_i$,  we only need it to satisfy
$$\dim H\cdot v\geq \dim H\cdot z(t)\text{ for all }t\in C, $$
which again follows from the transversality. Indeed, there is a Zariski open set $U_C$ of $C$, such that for any $t_0\in U_C$,
$$\dim H\cdot z(t_0)=\max_{t\in C}\dim H\cdot z(t).$$
By definition of $V_i$, for a fixed general $v\in V_i$ there is a $g_i\in G$ and $t_0\in U_C$ such that $g_i \cdot z(t_0)\in B(v,\ep)\in \PP^M$, by the transversality of $\fp$-action on $U_0$, for $\ep\ll 1$ there is an $h\in G$ close to identity such that  $h\cdot g_i\cdot z(t_0)\in V_i$.
By the genericity of $v$, we obtain
$$\dim H\cdot v\geq \dim H\cdot h\cdot g_i\cdot z(t_0)=\dim H\cdot z(t_0)\geq \dim H\cdot z(t) \text{ for all }t\in C\ .$$
and hence $\dim O(H, z^i)\geq \dim O(H,z)$ by our choice of $z^i:C_i\to V_i$.

Now we prove \eqref{Hv}. Suppose \eqref{Hv} does not hold  which is equivalent to $\dim V_i>\dim O(H,z^i)$, then we have
\begin{eqnarray*}
 \dim \overline S
 &\geq& \dim G\cdot V_i\\
&\geq& \dim G/H+\dim V_i  \ \ \ \ (\fp\text{-acting }  \text{transversely on } U_0)\\
 &>& \dim G/H+\dim O(H,z^i)\\
 &\geq & \dim G/H+\dim O(H,z)=\dim \overline S,
 \end{eqnarray*}
a contradiction. So the proof  of the Claim \ref{V-i} and hence the Lemma are  completed.
\end{proof}


The necessity of  the assumption that $G_{z_0}$ is reductive  can be illustrated by the following example.
\begin{exmp}\label{M2}
Let $M_2(\CC)=\{[v,w]\mid v,w\in \CC^2\}$
be the linear space of $2\times 2$ matrices, on which $G:={\rm GL}(2)$ is acting
 via multiplication on the left.  Let $V:=M_2(\CC)\oplus \CC\oplus \CC$, $G$ acts on $\PP V$ via the representation
\begin{equation*}
\begin{array}{cccc}
 \rho:& {\rm GL}(2) & \longrightarrow & {\rm SL}(V)\\
& g&\longmapsto &\rho(g)
\end{array}
\text{ with } \rho(g)\cdot \begin{bmatrix}A \\ x_5 \\ x_6 \end{bmatrix}:=\begin{bmatrix} g \cdot A  \\ \det (g^{-1})x_5  \\ \det (g^{-1}) x_6 \end{bmatrix}\ .
\end{equation*}
Let
$$z_0=\begin{bmatrix}0_{2\times 2} \\ 1 \\ 0\end{bmatrix}\text{  and }z_0^\prime=\begin{bmatrix}\begin{bmatrix} 1 & 0 \\ 0& 0\end{bmatrix}\\ 0 \\ 0\end{bmatrix}\in \PP V,$$
then their stabilizers are  $G_{z_0}=G$ and $G_{z_0^\prime}=\begin{bmatrix} \ast & \ast\\ 0& \ast \end{bmatrix}<{\rm GL}(2)$. In particular, $G_{z_0}$ is reductive while $G_{z_0^\prime}$ is not.
Now let
$$z(t)=\begin{bmatrix}\begin{bmatrix} t & 0 \\ 0& t^2\end{bmatrix}\\ 1 \\t\end{bmatrix}\text{  and }z^\prime(t)=\begin{bmatrix}\begin{bmatrix} 1 & 0 \\ 0& t\end{bmatrix}\\t \\ t^2\end{bmatrix}\in \PP V$$
be two curves in $\PP V$, then we have
$$ \lim_{t\to 0} \overline{O_{z(t)}}=\lim_{t\to 0}\PP V_{[1,t]}= \lim_{t\to 0} \overline{O_{z^\prime (t)}}=\PP V_{[1,0]}\ , $$
where $V_{[1,t]}:=\{ t x_5=x_6\}\subset V$.
Clearly, $z_0:=z(0)$ satisfies \eqref{Oz} while $z_0^\prime:=z^\prime(0)$ does not, since
$$
z^\prime_0\not\in\PP^1\cong\overline{G\cdot z^\pp_\ep}\subset \PP V_{[1,0]} \text{ for }0<|\ep|\ll 1\ \text{ where }z^\pp_\ep:=\begin{bmatrix}\begin{bmatrix} 1 & \ep \\ 0& 0\end{bmatrix}\\0 \\ 0\end{bmatrix}.
$$
\end{exmp}


\section{Gromov-Hausdorff continuity of conical K\"ahler-Einstein metric  on smooth Fano pair}\label{GH-cont}
In this section, we list  the important analytic results that will be needed in our main argument.
\subsection{Gromov-Hausdorff limit of  K\"ahler-Einstein Fano manifolds }
In this subsection, let us recall the main technical results obtained in the solution of Yau-Tian-Donaldson conjecture (see \cite{CDS2, CDS3}, \cite{Tian2014} and \cite{Ber12}).

\begin{thm}\label{t-YTD}
Let $X_i$ be a sequence of $n$-dimensional Fano manifolds with a fixed Hilbert polynomial $\chi$ and  $D_i\subset X_i$ be  smooth divisors in $|-mK_{X_i}|$ for a fixed $m>0$.  Let $\be_i\in (0,1)$ be a sequence converging to $\be_\infty$ with $0<\ep_0\leq \be_\infty\le1$. Suppose that each $X_i$ admits   K\"ahler  metric $\omega_i(\be_i)$ solving:
\begin{equation}\label{beKE}
\Ric(\omega(\be_i))=\be_i \omega(\be_i)+\frac{1-\be_i}{m}[D] \text{ on } X_i\ .
\end{equation}
that is, $\omega_i(\be_i)$ is a conical K\"ahler-Einstein metric will cone angle $2\pi(1-(1-\be_i)/m)$ along the divisor $D_i\subset X_i$. Then  the  Gromov-Hausdorff limit of any subsequence of $\{(X_i,\omega_i(\be_i))\}_i$ is  {\em homeomorphic} to a $\QQ$-Fano variety $Y$. Furthermore, there is a unique Weil divisor $E\subset Y$ such that
\begin{enumerate}
\item $(Y,\frac{1-\be_\infty}{m} E)$ is klt;
\item $Y$ admits a weak conical K\"ahler-Einstein metric solving
$$
\Ric(\omega(\be_\infty))=\be_\infty \omega(\be_\infty)+\frac{1-\be_\infty}{m}[E] \text{ on } Y\ .
$$
 In particular, $\aut(Y,E)$ is reductive and the pair $(Y,E)$ is $\be_\infty$-K-polystable;
\item possibly after passing to a subsequence, there are embeddings $T_i:X_i\to \PP^N$ and $T_\infty:Y\to \PP^N$, defined by the complete linear system $|-rK_{X_i}|$ and $|-rK_Y|$ respectively for $r=r(m,\ep_0, \chi)$ and $N+1=\chi(X_i,K^{-\otimes r}_{X_i})$, such that $T_i(X_i)$ converge to $T_\infty(Y)$ as projective varieties and $T_i(D_i)$ converge to $T_\infty (E)$ as algebraic cycles.
\end{enumerate}
\end{thm}

In the following corollary, we denote by $C^{, \al, \be}$ the space of conical K\"{a}hler metrics defined in \cite{Don2011}.
\begin{cor}\label{l-YTD}
 Let $(X,D)$ be smooth Fano pair with $D\in |-mK_X|$.  Then
 \begin{enumerate}
 \item
 $(X,D)$ is $\be$-K-stable if and only if it admits a conical K\"ahler-Einstein metric $\omega(\be)\in C^{,\al,\be}$ solving \eqref{beKE}.
\item
Let $\gamma\in (0,1]$. Then
$(X, D)$ is $\gamma$-K-semistable if and only if it admits a conical K\"{a}hler-Einstein metric $\omega(\beta)\in C^{,\al,\beta}$ solving \eqref{beKE} for any $\beta \in (0, \gamma)$.

\end{enumerate}

\end{cor}

\begin{rem}\label{Y-E}
Notice  that the limiting divisor  $E\subset Y$ is actually $\QQ$-Cartier. To see that, one notice that 
on the smooth locus of $Y$
\begin{equation}\label{sim}
E|_{Y^{\reg}}\sim -mK_{Y^{\reg}},
\end{equation}
which implies $E|_{Y}\sim -mK_{Y}$ as $Y$ is normal. On the other hand, $Y$ being $\QQ$-Fano implies that $K_Y$ is $\QQ$-Cartier. This together with \eqref{sim} implies that  $E$ is $\QQ$-Cartier. Also it was pointed out in \cite[Section 4.3]{DS2012} and \cite[Section 5]{CDS3} that if the sequence $\{(X_i,D_i)\}=\{(\cX_{t_i},\cD_{t_i})\}$ is a subsequence of a projective flat family $(\cX^\circ,\cD^\circ)\to C^\circ$ of smooth log Fano pairs over a smooth punctured (not necessarily complete) curve $C^\circ=C\setminus\{0\}$, i.e. $\{t_i\}\subset C^\circ$ and $t_i\stackrel{i\to \infty}{\longrightarrow}0$, then the Gromov-Hausdorff limit $(Y,E)$ can be realized as the central fiber of a {\em flat degeneration
$$
\xymatrix{
(\cX^\circ,\cD^\circ)\ \   \ar@{^{(}->}[r] \ar@{>}[d]  &(\cX,\cD) \ar@{>}[d]\\
          C^\circ \ar@{^{(}->}[r]  &\ \  C
            }\
$$
that is, $(Y,E)=(\cX_0,\cD_0)$. }
This important consequence is used in \cite{CDS3, Tian2014} to construct the destabilizing test configurations. In particular, the flatness of $\cX\to C$ is established in \cite[Section 4.3 ]{DS2012} and the flatness of $\cD\to C$ can be  deduced (see \cite[Chapter III, Exercise 10.9]{Har77}) from the fact that $\cD$ is Cohen-Macaulay since we have already shown it is $\QQ$-Cartier (see \cite[Corollary 5.25]{KM98}), and the morphism $\cD \to C$ is equi-dimensional.
\end{rem}

\subsection{Gromov-Hausdorff continuity of conical K\"ahler-Einstein metric  on smooth Fano family}\label{GH-KE}
\begin{defn}
Let
\begin{equation}\label{P-d-n}
 \HH^{\chi;N}:= \Hilb_\chi(\PP^N)\ .
\end{equation}
denote the Hilbert scheme of closed subschemes of $\PP^N$ with Hilbert polynomial $\chi$. For a closed subscheme $X\subset \PP^N$ with Hilbert polynomial $\chi\left(X,\left.\sO_{\PP^N}(k)\right |_X\right)=\chi(k)$, let $\hilb(X)\in \HH^{\chi;N}$ denote its {\em Hilbert} point.
\end{defn}

To set the scene, let
 $$
\begin{CD}
(\cX, \cD)@>i>> \PP^N\times \PP^N \times \Delta\\
@VV\pi V @VVV\\
\Delta@=\Delta
\end{CD}
$$
be projective flat family of Fano varieties  over the disc $\Delta=\{|t|<1\}\subset \CC$ such that:
\begin{enumerate}
\item  $\cX$ is {\em smooth}  and  $\cD\in |-mK_{\cX/\De}|$ is a {\em smooth} divisor defined by a {\em smooth} section $s_\cD\in \Gamma(\Delta, \omega_{\cX/\De}^{\otimes -m})$;
\item { Both $\pi$ and $\pi|_\cD$ are holomorphic {\em submersions} over $\Delta$.}
\end{enumerate}
To get rid of the ${\rm U}(N+1)$-ambiguity  for the later argument,  let us  assume that $\omega_\cX^{\otimes -r}$ is relatively very ample and $i$ be the embedding induced by a {\em prefixed basis}
$$\{s_i(t)\}_{i=0}^N\subset\Gamma(\Delta,\pi_\ast \sO_\cX(-rK_{\cX/\De}))$$
then $i^\ast \sO_{\PP^N}(1)\cong \sO_\cX(-rK_{\cX/\De})$. Now let ${(r \omega_\FS(t), h^{\otimes r}_\FS(t))}$ denote the metric on $(\cX_t,\sO_\cX(-rK_{\cX/C})|_{\cX_t})$ induced from the embedding $i$ via the basis $\{s_i\}$.  Suppose that for each $t\in \Delta$, $\cX_t$ is K-semistable. {Then by Lemma \ref{interpolate}, $(\cX_t, \cD_t)$ is $\beta$-K-polystable for any $\beta\in (0,1)$.  So by Corollary \ref{l-YTD},
for any $\beta\in (0,1)$ there exists conical
K\"{a}hler-Einstein metric $\omega(t,\beta)$ on the pair $(\mX_t, \frac{1-\beta}{m}\cD_t)$ which satisfies
$$
\Ric(\omega(t,\beta))=\beta \omega(t,\beta)+\frac{1-\beta}{m}[\cD_t]\ .
$$
In the following, by abusing of name, sometime we will abbreviate  $\omega(t,\beta)$ as a conical K\"{a}hler-Einstein metric with {\em cone angle} $\beta$ ({\em instead of $2\pi (1-(1-\beta)/m))$} {\em along }$D$, since the integer $m$ is fixed once for all for the whole paper .
Now assume
$\omega(t,\beta)=\omega_{\KE}(t,\beta)=\omega_\FS(t)+\ddbar \varphi(t,\be) $ where $r\cdot \omega_\FS(t)$ is equal to the Fubini-Study metric  induced from the embedding of $\cX_t\to \PP^N$ using the basis $\{s_i(t)\}_{i=0}^N$. Then $\varphi(t,\be)$ is
the unique solution (c.f.\cite[Theorem 7.3]{BBer2015}) to the equation
{
\begin{equation}\label{MA-h}
(\omega_\FS(t)+\ddbar \varphi(t,\be))^n=e^{f(t)-\beta \varphi(t,\be)}\frac{\omega^n_\FS(t)}{\left(|s_{\cD_t}|^2_{h_{\FS}^{\otimes m}(t)}\right)^{\frac{1-\beta}{m}}},
\end{equation}
}
where $f(t)$ satisfies
\begin{equation}\label{fJ}
\Ric(\omega_\FS(t))=\omega_\FS(t)+\ddbar f(t)\text{ and }\int_{\cX_t} e^{f(t)}\cdot\omega_\FS^n(t)=\int_{\cX_t}\omega^n_{\FS}(t)\ .
\end{equation}
}

\begin{rem}
It's easy to check that an equivalent form of equation \eqref{MA-h} is:
\begin{equation}\label{MA-h2}
(\omega_{\FS}(t)+\ddbar \varphi(t, \be))^n \cdot |s_{\cD_t}|_{h_{\FS}e^{-\varphi}}^{-\frac{2\beta}{m}} (s_{\cD_t}\otimes \overline{s_{\cD_t}})^{\frac{1}{m}}=1.
\end{equation}
\end{rem}

We define a positive definite Hermitian matrix
$$A_\KE(t,\be)=[( s_i,s_j)_{\KE,\be}(t)]$$ with
$$
(s_i,s_j )_{\KE,\be}(t)=\int_{\cX_t} \langle s_i(t),s_j (t)\rangle_{{h_\KE^{\otimes r}}(t,\be)}\omega^n(t,\be)\ ,
$$
where  $h_\KE(t,\beta):=h_\FS(t)\cdot e^{-\varphi(t,\beta)}$.
Now we introduce {\it {r-th} Tian's embedding}
\begin{equation}\label{Ts}
T:(\cX_t,\cD_t;\omega(t,\be))\longrightarrow \PP^N
\end{equation}
to be the one given by the basis $\{g(t,\be)\circ s_j(t)\}_{j=0}^N$ with $g(t,\be)=A^{-1/2}_\KE(t,\be)$.
\begin{defn}\label{def-chow}
We denote by

\begin{equation}\label{chow-be}
 \hilb(\cX_t,(1-\be)\cD_t) \in \HH^{\bchi;N}:=\HH^{\chi;N}\times \HH^{\ti\chi;N}
 \end{equation}
the {\em Hilbert} point of the pair $(\cX_t,\cD_t\subset \cX_t)\subset \PP^N$ using Tian's embedding for the basis $\{s_i\}$ with respect to K\"ahler form $\omega(t,\be)$, where $(\chi,\ti\chi)$ are the Hilbert polynomials of $X\subset \PP^N$ and $D\subset\PP^N$ respectively. We note that when $\beta=1$, the second factor $\HH^{\ti\chi;N}$ is not trivial as we still remember $\cD_t$, i.e., $\hilb(X, 0\cdot D)$ is not the same as $\hilb(X)$. See Remark \ref{remTE}.1 below.

\end{defn}

\begin{rem}\label{remTE}
We make some remarks:
\begin{enumerate}
\item
It is by definition that
$$\hilb(\cX_t,(1-\be)\cD_t)=(\hilb(\cX_t),\hilb(\cD_t); \omega(t,\beta)).$$ In the following we will always use the coefficient $(1-\be)$ to stress that the cycle is obtained via Tian's embedding with respect to the metric $\omega(t,\be)$.
\item
Tian's embedding is well defined for any klt $\mathbb{Q}$-Fano log pair with weak conical K\"{a}hler-Einstein metric $(X, (1-\beta)D; \omega_{\KE}(\beta))$. Note that for any weak conical
K\"{a}hler-Einstein metric $\omega_{\KE}(\beta)$, we always assume that the local potential is bounded (see \cite{BBEGZ}).
\item
The advantage of fixing a basis $\{s_i(t)\}_{i=0}^N\subset \Gamma(\Delta,\pi_\ast\sO_\cX(-rK_{\cX/\De}))$ lies in the fact that, the image of Tian's embedding and hence the Hilbert point,
$\hilb(\cX_t,(1-\be)\cD_t)$ is completely determined by the isometric class of $\omega(t,\be)$. See Lemma \ref{l-GH=CH}. \end{enumerate}
\end{rem}

\begin{prop}\label{cont-path}
 $\hilb(\cX_t, (1-\be)\cD_t)$ varies  continuously in $\HH^{\chi,\ti\chi;N}$ with respect to the pair $(\be,t)\in (0,1)\times \De$.
\end{prop}
\begin{proof}

Using the above notations, we claim that $\varphi_{\KE}(t,\beta)$ is continuous with respect to $t$ for any $\beta<1$. Assuming the claim, $A_{\KE}(t,\beta)$ is then continuous with respect $t$, and hence the images of Tian's embedding given by orthonormal basis change continuously.

Now we verify the claim by applying implicit function theorem. First we notice that
the complex manifold $(\cX_t,\cD_t)$ is diffeomorphic to a fixed pair $(X, D)$ endowed with the integrable complex structure $J_t$ thanks to the assumption that $\pi$ is a submersion.  Let $C^{2,\al;\be}(\cX_t,\cD_t; J_t)$ and $C^{,\al;\be}(\cX_t,\cD_t; J_t)$ denote the function spaces on $(\cX_t, \cD_t; J_t)$ defined in \cite{Don2011}.
For each fixed $t\in \Delta$, we consider the map:

\begin{equation}\label{F}
\begin{array}{cccc}
F(t,\be,\cdot):& C^{2,\al;\be}(\cX_t,\cD_t; J_t) & \longrightarrow & C^{,\al;\be}(\cX_t,\cD_t; J_t)\ \\
& \varphi &\longmapsto & \log\frac{(\omega_t+\sqrt{-1}\partial_{J_t} \bar{\partial}_{J_t}\varphi)^n |s_{\cD_t}|_{h_t}^{2(1-\be)/m}}{\omega^n }- f_{t}+\be\varphi
\end{array}
\end{equation}
where for simplicity we write $f_{t}=f(t)$, $\omega_t=\omega_{\FS}(t)$ and $h_t=h_{\FS}^{\otimes m}(t)$,
and $s_{\cD_t}$ is the defining section for $\cD_t$ as before. Note that $\varphi_{\KE}(t,\beta)$ is exactly the solution to the equation $F(t,\beta, \varphi)=0$. We would like to apply implicit function theorem to obtain the continuity of $\varphi_{\KE}(t,\beta)$ with respect to $t$.  In order to do that, we need to work with  a {\em fixed} function space,  whereas the spaces $C^{2,\alpha; \beta}(\cX_t,\cD_t; J_t)$
depends on the parameter $t$. To get around this, we notice that the metrics $\{\omega_t(\cdot, J_t\cdot)\}_t$ vary smoothly and hence $C^{,\al;\be}(X,D; J_t)=C^{,\al;\be}(X,D; J_0)$. This key observation allows us to identify the space $C^{2,\alpha;\beta}(X,D; J_0)$ and $C^{2,\alpha;\beta}(X,D; J_t)$ via the following simple way. Let us fix a family of background conical K\"{a}hler metrics:
\[
\hat{\omega}_t=\omega_t+\epsilon \sqrt{-1}\partial_{J_t}\bar{\partial}_{J_t}|s_{\cD_t}|_{h_t}^{2\gamma},
\]
with $\gamma=1-\frac{1-\beta}{m}\in (0,1)$ being fixed and $0<\epsilon\ll 1$.  Then we define a linear map:
\begin{equation}\label{Q}
\begin{array}{cccc}
Q_{t,\be}:=Q(t,\be,\cdot) :& C^{2,\al;\be}(X,D; J_0) & \longrightarrow & C^{2,\al;\be}(X,D; J_t)\ \\
& \tilde{\varphi} &\longmapsto & (-\triangle_{\hat{\omega}_t}+1)^{-1}\circ(-\triangle_{\hat{\omega}_0}+1)\tilde{\varphi}.
\end{array}
\end{equation}
Since $\ker(-\triangle_{\hat{\omega}_t}+1)=\{0\}$  by the proof of \cite[Proposition 8]{Don2011}, it follows from Donaldson's Schauder estimate in  \cite[Section 4.3]{Don2011} that $Q_{t,\be'}$ is an isomorphism for $ |t|\ll 1$ and $\be'\in (\be-\ep,\be+\ep)$ with $0<\ep\ll 1$. Also using the explicit parametrix constructed in \cite[Section 3]{Don2011}, $Q_t$ gives rise to a {\em continuous} local linear trivialization of the family of subspaces $C^{2,\al;\be}(X,D; J_t)\subset C^{,\al;\be}(X,D; J_t)=C^{,\al;\be}(X,D; J_0)$. Denoting $\tilde{\varphi}(t,\beta)=Q_{t,\be}^{-1}(\varphi(t,\beta))$, we can calculate:
\begin{eqnarray*}
\left.\frac{\partial F(t,\be,Q(t, \be, \tilde{\varphi}))}{\partial \tilde{\varphi}}\right |_{(0,\be,\tilde{\varphi}_{\KE})}(\phi)
&=&(\triangle_{\omega_\KE}+\be)\circ Q_0\phi=(\triangle_{\omega_\KE}+\be)\phi
\end{eqnarray*}
which is invertible by \cite[Theorem 2]{Don2011} since there exists no holomorphic vector field on the pair $(\cX_0,\cD_0)$ (see  \cite[Theorem 2.1]{SW2012} or Lemma \ref{l-aut}). Now we can apply effective implicit function theorem as in \cite[Section 4.4]{Don2011} to the map $F(t,\be,Q(t,\be,\cdot)):C^{2,\al;\be}(X,D; J_0)\to C^{,\al;\be}(X,D; J_0)$ to get a continuous family of solutions $\tilde{\varphi}_{\KE}(t,\be')$ to the equation $F(t,\be',Q(t,\be',\tilde{\varphi}))=0$ for all $|t|\ll 1$ and $\be'\in (\be-\ep,\be+\ep)$ with $0<\ep\ll 1$.
Since the argument for this last statement is standard, we will only sketch its proof. For a fixed $\beta$ by the usual implicit function theorem we first get a family of solutions $\tilde{\varphi}^{(1)}_{\KE}(t, \beta)$ to the equation $F(t, \be, Q(t, \beta, \tilde{\varphi}^{(1)}_{\KE}))=0$ for $|t|\ll 1$. Then we can apply Donaldson's argument of deforming cone angles
in \cite[Section 4.4]{Don2011} in a family version to further get $\tilde{\varphi}_{\KE}(t, \beta')$ for any $|t|\ll 1$ and $\beta'\in (\beta-\epsilon, \beta+\epsilon)$.

More precisely, let $\omega_{\KE}(t, \beta)=\omega_{\FS}+\sqrt{-1}\partial\bar{\partial}\tilde{\varphi}_{\KE}^{(1)}(t, \beta)$ be the continuous family of $C^{, \alpha; \beta}$ conical K\"{a}hler-Einstein metric obtained earlier. For each $\beta'\in (\beta-\epsilon, \beta+\epsilon)$ and $t$ near $0$, we define the new reference metric
$\omega(t, \beta'):=\omega_{\KE}(t, \beta)+\sqrt{-1}\partial\bar{\partial}(\|s_{\cD_t}\|^{2\beta'/m}-\|s_{\cD_t}\|^{2\beta/m})$ where $\|\cdot\|^2$ is a smooth extension of Hermitian metric determined by $h_{\FS} \exp(-\tilde{\varphi}^{(1)}_{\KE}(t, \beta))$ on $\left. (K_{\cX_t}^{-1})^{\otimes m}\right|_{\cD_t}$ (using the fact that $\tilde{\varphi}^{(1)}_{\KE}(t, \beta)$ is  {\em smooth in  tangential directions} by \cite[Section 4.3]{Don2011}). Then as in the proof of \cite[Proposition 7]{Don2011}, one can show that
\begin{enumerate}
\item $k_{\beta'}:=|s_{\cD}|_{\beta'}^{-2\beta'/m}(s_{\cD}\otimes \overline{s_{\cD}})^{1/m} \; \omega(t, \be')^n$ (see \eqref{MA-h2}) satisfies $\|k_{\beta'}-1\|_{C^{,\alpha; \be'}}\rightarrow 0$ as $\beta'\rightarrow \beta$. Here we used $|\cdot |^2_{\beta'}$ to denote the Hermitian metric on $\sO_{\cX_t}(-mK_{\cX_t})$ whose curvature is equal to $m\cdot \omega(t, \beta')$.
\item Let $\Delta_{\beta'}$ denote the Laplace operator associated to $\omega(t, \beta')$, then $\Delta_{\beta'}+\beta'$ is invertible and the operator norm of its inverse is bounded by a fixed constant independent of $\beta'$ and $t$ near $0$.
\end{enumerate}
So the effective version of implicit function theorem allows us to get a continuous family of solutions $\tilde{\varphi}_\KE(t, \beta')$. In fact, one notice that $\omega(t, \beta')=\omega_{\FS}+\ddbar \psi(t, \beta')$ with
$\psi(t, \beta')=\tilde{\varphi}^{(1)}_\KE(t, \beta)+\|s_{\cD_t}\|^{2\beta'/m}-\|s_{\cD_t}\|^{2\beta/m}$ is an {\em approximate solution } to the conical K\"{a}hler-Einstein equation by the item (1) above and is continuous with respect to both $t$ and $\beta'$. By item (2) and the effective implicit function theorem, we then know that the difference between the actual solution $\tilde{\varphi}_{\KE}(t, \beta')$ and $\psi(t, \beta')$ approaches $0$ in $C^{2, \alpha; \beta'}$-norm as $\beta'\rightarrow \beta$. As a consequence, $\tilde{\varphi}_{\KE}(t, \beta')$ is continuous at $\beta'=\beta$ in $C^0$-norm with respect to both $\beta'$ and  $t$.  Noticing  that the argument above does not depend on the origin $0$ and $ \be$ we choose, hence $\varphi_{\KE}(t,\be')=Q(t,\be',\tilde{\varphi}_{\KE}(t,\be'))$ is  continuous with respect to all $t\in \De$ and $\be'\in (\be-\ep,\be+\ep)$.

By using the complex Monge-Amp\`{e}re equation in \eqref{MA-h} or \eqref{MA-h2}, we see that the family of volume forms $\omega_{\KE}(t, \beta')^n$ on the fixed {\it smooth} manifold $X$ is continuous in $L^{p}(X),\forall p\in [1,1/(1-\be'))$  with respect to $\beta'$ and $t$, which implies that the family of matrices of $L^2$-inner products $A_{\KE}(t, \beta')=[(s_i, s_j)_{\KE, \beta'}(t)]$ is also continuous with respect to $t$ and $\beta'$. So  the Tian's embeddings $T(\cX_t, \cD_t; \omega(t, \beta'))$ determined by $\{A^{-1/2}_{\KE}(t, \beta') \circ s_j(t)\}_{j=0}^N$ indeed produce a continuous family of Hilbert points inside $\HH^{\bchi;N}$.

\end{proof}
Let $\{(X_i,D_i)\}$ be a sequence of smooth Fano pairs with  {\em a fixed} Hilbert polynomial $\chi$ and $D_i\in |-mK_{X_i}|$. Suppose  each $X_i$'s admit a unique  conical K\"ahler-Einstein form $\omega(i,\be_i)$ solving
$$\Ric(\omega(i,\beta_i))=\beta_i \omega(i,\beta_i)+\frac{1-\beta_i}{m}[D_i] \text{ on } X_i
$$
with $\inf\beta_i\geq\epsilon>0$, we define
$$T_i: (X_i,D_i;\omega(i,\beta_i))\longrightarrow \PP^N$$
to be the Tian's embedding with respect to  $\omega(i,\beta_i)$ for sufficiently large $N$ depending only on $\epsilon,\ m$ and the fixed Hilbert polynomial $\chi$, and let $\hilb(X_i,(1-\be_i)D_i)\in \HH^{\chi;N}\times \HH^{\ti\chi;N}$ denote the Hilbert point corresponding to the Tian's embedding of $X_i$ with respect to $\omega(i,\beta_i)$.  Then we have


\begin{lem}\label{l-GH=CH}
 Let  $(X,D)\subset \PP^N$ be  a log $\QQ$-Fano pair with  the same Hilbert polynomial $\chi$ and $D\in |-mK_{X}|$. Suppose $(X,D)$ admits a
weak conical K\"ahler-Einstein form $\omega(\be)$ with $\be=\lim_{i\to \infty}\be_i$ solving
$$\Ric(\omega(\beta))=\beta \omega(\beta)+\frac{1-\beta}{m}[D] \text{ on } X\ .
$$
Then
$$(X_i,D_i;\omega(i,\beta_i))\stackrel{\GH}{\longrightarrow} (X, D;\omega'(\be)) \mbox{ as } i\rightarrow \infty $$ for a conical K\"{a}hler-Einstein metric $\omega'(\beta)$
is equivalent to the following statement: there is a sequence of  $\{g_i\}\subset {\rm U}(N+1)$ such that
$$g_i\cdot \hilb(X_i, (1-\be_i)D_i)\longrightarrow \hilb (X,(1-\be) D)\in \HH^{\bchi;N} \mbox{ as } i\to \infty, $$
where $\hilb(X,(1-\be)D)$ denote the Hilbert point of Tian's embedding $T: (X,D;\omega(\beta))\to \PP^N$ for a fixed  basis $\{s_i\}$.
\end{lem}
\begin{proof}
This follows directly from Theorem \ref{t-YTD} which is in turn from the works of \cite{CDS2, CDS3} and \cite{Tian2012, Tian2014}.
Indeed let us assume that $(X_i, D_i; \omega(i,\beta_i))\stackrel{\GH}{\longrightarrow} (X,D; \omega'(\beta))$ where we can assume the limit exists by Theorem \ref{t-YTD}. Then by Theorem \ref{t-YTD}.(3)
$(T_i(X_i), T_i(D_i))\rightarrow (T'_\infty(X), T'_\infty(D))$ where $T_i$ (resp. $T'_\infty$) is given by Tian's embedding determined by an orthonormal basis of $H^0(X_i, -mK_{X_i})$ (resp. $H^0(X, -mK_X)$) with respect to $\omega(i, \beta_i)$ (resp. $\omega'(\beta)$).
Assume $\omega(\beta)$ is also a conical K\"{a}hler-Einstein metric on $(X, D)$. By the uniqueness of conical K\"{a}hler-Einstein metrics proved in \cite{BBEGZ}, there exists a holomorphic automorphism $\sigma \in {\rm Aut}(X, D)$ such that $\sigma^*\omega(\beta)=\omega'(\beta)$. Moreover because $\sigma$ lifts to ${\rm Aut}(X, D, -mK_X)$, there is a unitary isomorphism between $(H^0(X, -mK_X), \|\cdot \|^2_{\omega'(\beta)})$ and $(H^0(X, -mK_X), \|\cdot\|^2_{\omega(\beta)})$ where $\|\cdot\|^2_{\omega(\beta)}$ ($\|\cdot\|^2_{\omega'(\beta)}$) is the $L^2$ inner product induced by $\omega(\beta)$ (resp. $\omega'(\beta)$). Via this isomorphism, we have $(T_i(X_i), T_i(D_i))\rightarrow (T_\infty(X), T_\infty(D))$ where $T_\infty$ is given by Tian's embedding determined by an orthonormal basis of
$H^0(X, -mK_X)$ with respect to $\omega(\beta)$. Now the statement of the lemma holds because the orthonormal basis of a unitary vector space is defined only up to a ${\rm U}(N+1)$-ambiguity.
\end{proof}

\section{Strong uniqueness for $0<\beta\ll 1$}\label{uni}
In this section, we will give a purely algebro-geometric proof of the fact that when the angle $\beta>0$ is sufficiently small, then there is a unique filling.
\begin{prop}\label{p-acc}
For a fixed a finite set $I\subset [0,1]$, there exists a number $\beta_I>0$ such that if $(X,(1-\beta_I)D)$ is a klt pair, $D$ is $\mathbb{R}$-Cartier and the coefficients of $\RR$-divisor $D$ are contained in $I$, then  $(X,D)$ is log canonical.
\end{prop}
\begin{proof}By \cite[Theorem 1.1]{HMX2014}, we know that for the set of all $n$-dimensional log pairs $(X,D)$ satisfying the property that $D$ is a $\RR$-divisor and its coefficients  are contained in $I$, the set of log canonical thresholds
$$\{{\rm lct}(X,D)|\mbox{ $X$ is $n$ dimensional, the coefficients of $D$ are in $I$}\} $$
satisfies the {\em ascending chain condition} (ACC). In particular, there exists a maximum $\beta_I$ among all log canonical thresholds which are strictly less than 1.

Then we know that if $(X, (1-\beta_I)D)$ is klt and $D$ is $\QQ$-Cartier, $(X,D)$ is log canonical, since otherwise, we will have a pair whose log canonical threshold is in $(1-\beta_I,1)$, which is a contradiction.

\end{proof}

Let $\cX\to C$ be a flat family of $\mathbb{Q}$-Fano varieties over a smooth pointed curve $0\in C$, and $\cX$ is assumed to be $\mathbb{Q}$-Gorenstein.
Fix $m> 1$ and $\cD\sim_{C}-mK_{\cX}$ is a divisor  such that {after a possibly shrinking of the pointed curve $C$}, for {\em every} $t\in C$ the fiber $(\cX_t,\frac{1}{m}\cD_t)$ is {\em klt}. For instance, we can choose $m$ sufficiently divisible such that $|-mK_{\cX}|$ is relatively base point free over $C$ and $\cD\sim_{C} -mK_{\cX} $ to be a  general divisor in $|-mK_\cX|$. In particular,  after a possible shrinking of $C$,
if $\mathcal{X}\to C$ is smooth over $C^{\circ}$, one might choose $\cD$ so that  $\cD_t$ is {\em smooth} for $t\in C^\circ$ provided $\cX_t$ is so for $t\in C^\circ$.

\begin{thm}\label{unique}
Let $(\cX,\cD)\to C$ be a flat family introduced as above. Let
$$\beta_0:=\min\left\{\beta_I, \frac{1}{m+1}\right\}$$ with $\beta_I$ being given in Proposition \ref{p-acc} for the set $I=\{\frac{q}{m}| q=1,2,...,m\}$. For any fixed $\beta\in (0, \beta_0]$, suppose $(\cX^\prime,\cD^\prime)\to C$ is another flat  family  with $\displaystyle K_{\cX'}+\frac{1}{m}\cD'$ being $\QQ$-Cartier and satisfies
\begin{equation}\label{C0}
(\cX^\prime,\cD^\prime)\times_{C} C^\circ\cong (\cX,\cD)\times_{C} C^\circ
\end{equation}
  and  $(Y_\beta, \frac{1-\beta}mE_\beta):=(\cX^\prime_0,\frac{1-\beta}m\cD^\prime_0)$ being $\QQ$-Fano. 
 Then the above isomorphism can be extended to an isomorphism
$$ (\cX^\prime,\cD^\prime)\cong  (\cX,\cD).$$
\end{thm}
\begin{proof}
 Since 
 $\cD\times_{C}C^{\circ}$ is integral by our choice, the coefficients of ${\frac{1}{m}E_\beta}$ lies in the set $\{\frac{q}m\mid q\in \NN\}$. By our assumption that $(Y_\beta, \frac{1-\beta}mE_\beta)$ is klt and $\be\leq\be_0\leq \frac{1}{m+1}$ , we have
$$
\frac{1}{m+1}\leq\frac{1-\be}m \text{ and }\frac{1-\beta}m c_i<1 \text{ hence } c_i<m+1
$$
where $E_\beta=\sum_ic_i E_{\be,i}$ with $E_{\be,i}$ being a prime divisor for each $i$. Hence the coefficients of ${\frac{1}{m}E_\beta}$ must lie in $I=\{\frac{q}{m}| q=1,2,...,m\}$. By our assumption of $\beta\in(0,\be_0]\subset(0,\be_I] $, we know that $(Y_\beta, \frac{1}{m}E_\beta)$ is {\em log canonical} by Proposition \ref{p-acc}. Furthermore, since $Y_\beta$ is irreducible, we know that
$$K_{\cX'}+\frac{1}{m}\cD'\sim_{\QQ, C}0$$ as this holds over $C^\circ$. 

Let $W$ be a common resolution
$$
\xymatrix{ & W  \ar@{>}[dl]_{p} \ar@{>}[dr]^{q} &\\ \cX  \ar@{-->}[rr]^\phi & & \cX'}
$$
that is an isomorphism over $C^\circ$.
If the birational map $\phi$  extends to a birational map
$
\xymatrix{\cX_0  \ar@{-->}[r]^{\phi|_{\cX_0}} & Y_\be},
$
then
$$q^\ast K_{\cX'}\sim_{C,\QQ} p^\ast K_{\cX}$$
as $\phi$ is an isomorphism in codimension one, which implies
$$
\cX=\proj \bigoplus_{r=0}^\infty\sO_W(-rp^\ast K_{\cX/C})=\proj \bigoplus_{r=0}^\infty\sO_W(-rq^\ast K_{\cX'/C})=\cX'
$$
as both $\cX_0$ and $Y_\be$ are $\QQ$-Fano and  we are done already. So from now on we assume $\cX_0\neq Y_\be${ on $W$}.

Now let us write
\begin{equation}\label{p}
p^*(K_{\cX}+\frac{1}{m}\cD)+a_0Y_\beta+\sum a_iE_i=K_W+\frac{1}{m}p^{-1}_*\cD.
\end{equation}
Since $(\cX_0, \frac{1}{m}\cD_0)$ is klt, this implies that $(\cX,\frac{1}{m}\cD+\cX_0)$ is plt near $\cX_0$ by inversion of adjunction \cite[Theorem 5.50]{KM98}. Hence for any divisor $F$ whose center is contained in $\cX_0$ we have
$$ -1< a(F, \cX,\frac{1}{m}\cD+\cX_0)=a(F, \cX,\frac{1}{m}\cD)-{\rm ord}_F(\cX_0)\leq a(F, \cX,\frac{1}{m}\cD)-1
$$
{where ${\rm ord}_F$ denotes the vanishing order along the divisor $F$}.  Therefore,  $(\cX,\frac{1}{m}\cD)$ is terminal along $\cX_0$ and $a_0>0,\ a_i>0$.
Similarly, by writing
\begin{equation}\label{q}
q^*(K_{\cX'}+\frac{1}{m}\cD')+b_0\cX_0+\sum b_iE_i=K_W+\frac{1}{m} q^{-1}_*\cD',
\end{equation}
 we obtain $b_0,b_i\ge 0$ because $(Y,\frac{1}{m}E)$ is log canonical thanks to our choice of $\beta$ and Proposition \ref{p-acc}. Since the right hand sides of \eqref{p} and \eqref{q} are equal to each other by \eqref{C0}, $\cX_0\neq Y_\be$,
and both $K_{\cX}+\frac{1}{m}\cD$ and $K_{\cX'}+\frac{1}{m}\cD'$ are $\QQ$-linearly equivalent to a relatively {\em trivial} divisor over $C$,  these imply there is a constant $c\le 0$ such that
$$a_0Y_\beta+\sum a_iE_i=b_0\cX_0+\sum b_iE_i+c\cdot W_0\ .$$
By comparing the coefficients of $Y_\beta$ on both sides, we see $c>0$; but by comparing the coefficients of $\cX_0$ on both sides, we see $c\le 0$. This contradiction implies that $\cX^{\prime}=\cX$.
\end{proof}

\begin{rem}
{If $m=1$, the pair we get is plt instead of klt. The above argument indeed also applies to this case. }

A similar uniqueness statement is observed in  \cite[4.3]{Odaka14} and the above argument indeed gives a straightforward proof of it.
\end{rem}

We also notice that the automorphism group ${\rm Aut}(X,D)$ is always finite by the following well known fact.
\begin{lem}\label{l-aut}
Let $(X,D)$ be a klt pair such that $-K_X$ is ample and $D\sim_{\mathbb{Q}}-K_X$. Then ${\rm Aut}(X,D)$ is finite.
\end{lem}
\begin{proof}We can choose a sufficiently small $\epsilon>0$ such that $(X,(1+\epsilon)D)$ is klt and we know  $K_X+(1+\epsilon)D$ is ample. As ${\rm Aut}(X,D)$ preserves $K_X+(1+\epsilon)D$, so it gives polarized automorphisms. Therefore, to prove it is finite, we only need to show that it does not contain $\mathbb{G}_m$ or $\mathbb{G}_a$ as a subgroup.
For $\mathbb{G}_m$ this follows from \cite[Lemma 3.4]{HX11}. As mentioned there, the same argument also works for $\mathbb{G}_a$ verbatimly.
\end{proof}


\section{Continuity method}\label{ss-continuity}

In this section, we will develop our continuity method which serves as the main technique of the proof of the main result. Let $(C,0)$ be a smooth pointed curve,  we define  $C^\circ:=C\setminus \{0\}$ as before. To begin with, let us fix $\mathfrak{B}\in (0,1]$ and  we will assume the nearby smooth fibers are all  {\em $\mathfrak{B}$-K-polystable} for the rest of this section. We fix an $\ep\in (0,\be_0)$, with $\beta_0$ being given  as in Theorem \ref{unique}. By Lemma \ref{interpolate}, for any $\beta\in [\ep, \mathfrak{B}]$, $(\cX_t,\cD_t)$ is $\beta$-K-polystable. Applying \cite{CDS1,CDS2,CDS3, Tian2012} (cf. Corollary \ref{l-YTD}), we  conclude that $(\cX_t,\cD_t)$ admits  a ( unique when $\be<1$ thanks to  \cite[Theorem 7.3]{BBer2015}) conical K\"{a}hler-Einstein metric with cone angle $2\pi(1-(1-\be)/m)$ along $\cD_t$  for all $t\in C^{\circ}$ near 0.
 This leads us to introduce the following notion.

\begin{defn}\label{fano-r}
 We say
 $$
\begin{CD}
(\cX, \cD;\sL)@>>> (\PP\sE;\sO_{\PP\sE}(1))\\
@VV\pi V @VVV\\
C @= C
\end{CD}
$$
is a {\em K\"ahler-Einstein degeneration of index $(r,\mathfrak{B})$ } if for any $\be\in [\ep,\mathfrak{B}]$
 \begin{enumerate}
 \item $\cD\in |-mK_{\cX}|$;
 \item  $\sL=K_\cX^{\otimes -r}$ is  relatively {\em very ample} and $\sE=\pi_\ast \sL $ is locally free of rank $N+1$;
 \item   $\forall t\in C$, $(\cX_t,\frac{1}{m}\cD_t)$ is klt and $(\cX_t,\cD_t)$ is a {\em smooth}  Fano pair for $\forall t\in C^\circ$;
 \item  For $\beta<1$ and $\forall t\in C^\circ$, $(\cX_t,\cD_t)$ admits a unique K\"ahler form $\omega(t,\be)\in C^{,\al,\be}$ in the sense of \cite{Don2011}  solving
 \begin{equation}\label{omega-t-be}
 \Ric(\omega(t,\beta))=\beta \omega(t,\beta)+\frac{1-\beta}{m}[\cD_t]\text{ on }\cX_t.
\end{equation}
Moreover, $\omega(t,\be)$ gives rise to $r$-th Tian's embedding
 $$ T:(\cX_t,\cD_t;\omega(t,\be))\longrightarrow \PP^N\ .$$
 \end{enumerate}
 \end{defn}

By Theorem \ref{t-YTD}, there is a uniform $r=r(\cX,\cD)$ being  independent of $\be\in[\ep,\mathfrak{B}]$ such that all  Gromov-Hausdorff limits of  subsequences of the family $\{(\cX_t,\cD_t;\omega(t,\be))\}_{t\in C,\ \be\in[\ep,\bbe]}$ can  be embedded in to $\PP^N$. 
\begin{defn}\label{B}
Let us continue with the notation as above and define


\begin{equation*}
\bB_r(\cX,\cD):=\left \{\be\in[\ep, \bbe]\left | \mbox{\stackbox[v][m]{{\footnotesize
$(X,D)$ admits a  conical K\"ahler-Einstein metric $\omega(\be)$  solving
$$
\Ric(\omega(\beta))=\beta \omega(\beta)+\frac{1-\beta}{m}[D]\text{ on }X.
$$
Moreover,   $(\cX_t,\cD_t;\omega(t,\be))\stackrel{\GH}{\longrightarrow} (X,D;\omega(\be))$ as $t\to 0$ .
}}}\right .\right \}
\end{equation*}
and we fix $\bT$ such that $\epsilon \le \bT \le \sup\{\sigma\in [\ep,\bbe] \left | [\ep,\sigma] \subset \bB_r(\cX,\cD)\right.\} \ .$
\end{defn}

By Theorem \ref{t-YTD}, the Gromov-Hausdorff limit of any subsequence of $(\cX_{t_i}, \cD_{t_i},\omega({t_i,\beta}))$ is a $\mathbb{Q}$-Fano $Y$ together with a $\QQ$-Cartier divisor $E$ such that $(Y,\frac{1-\beta}{m}E)$ is log Fano.

\begin{lem}\label{be-sm}
$\bB_r(\cX,\cD)\supset [\ep,\beta_0]$.
\end{lem}

\begin{proof} After shrinking $C$ if necessary, we may choose a holomorphic basis
$$\{s_i(t)\}_{i=0}^N\subset \Gamma(\Delta,\pi_\ast \sO_\cX(-rK_{\cX/\De}))$$
 for the family $\cX\to C$ as in Section \ref{GH-KE}, which gives rise to  an  \em algebraic arc
\begin{equation}
\begin{array}{ccc}
 z: C & \overset{} {\xrightarrow{\hspace*{1.3cm}}}  & \HH^{\bchi;N} \times C  \ \\
 t & \longmapsto & (\hilb(\cX_t,\cD_t),t) \ .
\end{array}%
\end{equation}
  For this arc, we know that $\hilb(Y,E)$ for  the Gromov-Hausdorff limit $(Y,E;\omega_Y)$ of any subsequence $\{(\cX_{t_i},\cD_{t_i};\omega_\KE(t_i,\be))\}_{t_i\to 0}$  lies in the fiber over $0\in C$ of the morphism
 $$
\begin{CD}
\overline{\SL(N+1)\cdot \mathrm{Im}z} @>>> \HH^{\bchi;N}\times C\\
@VV\pi_C V @VVV\\
C @= C\ .
\end{CD}
 $$
By choosing an arc $\ti z: \ti C\to \overline{\SL(N+1)\cdot \mathrm{Im}z}$ that passes through $\hilb(Y,E)$ and dominates $C$, and comparing the universal family over $\mathrm{Im}\ti z\subset \HH^{\bchi;N}\times C$  with the pull-back family induced by the map $\pi_C\circ \ti z:\ti C\to C$, we conclude that $(Y,E)=(X,D)$ as long as $\beta\le \beta_0$ thanks to  Theorem  \ref{unique}. Our proof is thus completed.
\end{proof}

\begin{rem}\label{r-small}
Notice that Lemma \ref{be-sm} implies that for $\beta\in [0,\beta_0]$, $(X,D)$ is actually {\em $\beta$-K-stable} (see Lemma \ref{l-aut}), which can  also  be proved by using Theorem \ref{unique} and a verbatim extension of the theory of special test configuration  developed in \cite{LX} to the log setting. In fact, using the latter approach, we can indeed conclude a pair  $(X_0, D_0)$ is $\beta$-K-stable if $D_0\sim -mK_{X_0}$, $(X_0,\frac{1}{m}D_0)$ is klt  and $\beta\in [0,\beta_0]$, {\em without} assuming $X_0$ is smoothable. However, this stronger fact is not needed for the rest of the paper.
\end{rem}

From now on, let us assume $(\cX_0, \cD_0)$ is {\em $\bbe$-K-polystable}, we are going to show that $\bB_r(\cX,\cD)$ is both open and closed  in the set $[\ep, \mathfrak{B}]$, or equivalently we can choose
\[
\bT=\mathfrak{B}=\max_{[\ep,\sigma]\subset \bB_r(\cX,\cD)}\{\sigma\}.
\]
To do this, we first define a map
\begin{equation}\label{tau}
\begin{array}{cccc}
 \tau:& [\epsilon, \mathfrak{B}]\times C^\circ & \longrightarrow &\HH^{\bchi;N}\ .\\
& (\beta, t)&\longmapsto & \hilb(\cX_t,(1-\be)\cD_t)
\end{array}
\text{ (cf. see Definition \ref{def-chow} )}
\end{equation}
Then we have
\begin{lem} \label{tau-cont}$\tau|_{[\epsilon,\mathfrak{B}]\times C^\circ}$ is continuous.
\end{lem}
\begin{proof}
By Proposition \ref{cont-path},   $\tau(\cdot,\cdot)$ is continuous with respect to $(\be, t)$ on $[\ep, \bbe)\times C^\circ$.
By Theorem \ref{t-YTD}, the Gromov-Hausdorff limit of $(\cX_t,\cD_t;\omega(t,\be_i))$ for any sequence $\be_i\nearrow \mathfrak{B}$ is $\mathfrak{B}$-K-polystable and lies in $\overline{ {\rm SL}(N+1)\cdot\cX_t}$. On the other hand, since $(\cX_t,\cD_t)$ is $\mathfrak{B}$-K-polystable, this implies the limit must lie in ${\rm U}(N+1)\cdot \hilb(\cX_t,(1-\mathfrak{B})\cD_t)$, hence the metrics $\{h_\KE(t,\beta)\}_{(t,\be)\in [\ep,\bbe]\times C^\circ}$( cf. Section \ref{GH-KE}) vary continuously for $(\be,t)\in [\ep,\bbe]\times C^\circ$. So $\tau(\cdot, t)$ is also continuous at $\be=\bbe$ with respect to the  basis $\{s_i\}$ in Definition \ref{def-chow}.
Thus the proof is completed.
\end{proof}

By Lemma \ref{be-sm}, we know that the continuity of $q\circ\tau$  can be  extended to $[\epsilon,\beta_0]\times \{0\}$, where $q\colon \HH^{\bchi;N} \to  \HH^{\bchi;N} / {\rm U}(N+1)$ is the natural quotient morphism, which is continuous with respect to the quotient topology on $ \HH^{\bchi;N} /\rU(N+1)$. Next we will show indeed $\beta$-continuity of $q\circ\tau$ can be extended  to $[\epsilon, \bT]\times\{0\}$ (i.e.  including the central fiber)  as long as $q\circ \tau$ can be continuously  extended to $[\epsilon, \bT)\times {C}$ based on the fact that $(X,D)$ is a degeneration of smooth pairs $(\cX_t,\cD_t)$ admitting conical K\"ahler-Einstein metrics $\omega(t,\beta)$ for any $\beta\in [\epsilon, \bT)$.
To do that,  let  us  prefix a {\em continuous} distance function on  $\HH^{\bchi;N}$
 \begin{equation}\label{dist}
\dist_{\HH^{\bchi;N}}:\HH^{\bchi;N}\times \HH^{\bchi;N}\longrightarrow \RR_{\geq 0}\ .
\end{equation}

\begin{lem}\label{X-be}
 Let us continue with the above setting.  In particular, $(X,D)=(\cX_0,\cD_0)$ is $\bbe$-K-polystable. Then $(X,D)$ admits a conical K\"ahler-Einstein metric $\omega_{X}(\bT)$ with angle  $2\pi(1-(1-\bT)/m)$ along the divisor $D$.

Furthermore, for any sequence $\{\be_i\}\subset (\ep,\bT)$ satisfying $\be_i\nearrow \bT$,  we have
$$\dist_{\HH^{\bchi;N}} (\hilb(X,(1-\be_i)D),  \rU(N+1)\cdot\hilb(X,(1-\bT)D)\longrightarrow 0, 
$$
where $\hilb(X,(1-\bT)D)$ is the Hilber point corresponding to the cycle obtained via Tian's embedding of  $(X,D;\omega_X(\bT))$.
\end{lem}
 \begin{proof}
 By Theorem \ref{t-YTD} and the definition of $ \bT$, for any $\beta<\bT$, the Gromov-Hausdorff limit as $t\to 0$ of $(\cX_t,\cD_t;\omega(t,\beta))$ converges to a weak conical  K\"ahler-Einstein metric on $(X,D;\omega(\be))=(\cX_0,\cD_0;\omega(0,\beta))$.
This implies  that for each fixed $\be_i<\bT$, there is a  $C^\circ\ni t_i\to 0$ so that
\begin{equation}\label{1/i}
\dist_{\HH^{\bchi;N}} (\hilb(\cX_{t_i},(1-\be_i)\cD_{t_i}),\rU(N+1)\cdot \hilb(X,(1-\be_i)D))<1/i\ .
\end{equation}

It follows from Theorem \ref{t-YTD} that any subsequence of $\{ (\cX_{t_i},\cD_{t_i}; \omega(t_i,\be_{i}))\}$, there is a Gromov-Hausdorff convergent  subsequence. Now suppose there is a subsequence
$$
(\cX_{t_{i_k}},\cD_{t_{i_k}};\omega(t_{i_k},\be_{i_k}))\stackrel{\GH}{\longrightarrow }(Y,E; \omega_Y(\bT)) \text{ as }k\to \infty,
$$
from which we obtain there are $g_{i_k} \in {\rm U}(N+1)$ such that
$$
g_{i_k}\cdot \hilb(\cX_{t_{i_k}},(1-\be_{i_k})\cD_{t_{i_k}})\longrightarrow  \hilb(Y,(1-\bT)E),
$$
where $\hilb(Y,(1-\bT)E)$ is the Hilbert point corresponding to the Tian's embedding of $(Y,E)$ using the limiting conical K\"ahler-Einstein metric $\omega_Y(\bT)$ of angle $2\pi(1-(1-\bT)/m)$ along a $\QQ$-Cartier divisor $E$.  In particular, $(Y,E)$ is $\bT$-K-polystable by \cite[Theorem 4.2]{Be12}.
On the other hand, by \eqref{1/i} we have
\begin{equation}\label{Y-X}
\hilb(Y,(1-\bT)E)\in \overline{{\rm SL}(N+1)\cdot \hilb(X, D)}\subset \HH^{\bchi;N},
\end{equation}
Suppose $(Y,E)\not\cong(X,D)$, then by \cite[Proposition 1]{Don2012} there is a test configuration of $(X,D)$ with central fiber $(Y,E)$ and vanishing generalized Futaki invariant since $(Y,E)$ is $\bT$-K-polystable. This contradicts our assumption that  $(X,D)$ is $\bT$-K-polystable. Hence we must have $(Y,E)\cong(X,D)$. In particular, $X$ admits a weak conical K\"ahler-Einstein metric with angle $2\pi(1-(1-\bT))$ along $D$.

In conclusion, we have
 $$
(\cX_{t_{i_k}},\cD_{t_{i_k}};\omega(t_{i_k},\be_{i_k}))\stackrel{\GH}{\longrightarrow }(X,D; \omega_X(\bT)),
$$ which   implies
$$\dist_{\HH^{\bchi;N}} (\hilb(X,(1-\be_i)D),  \rU(N+1)\cdot\hilb(X,(1-\bT)D)\longrightarrow 0\ .$$
 Combining with \eqref{1/i}, the proof is completed.
 \end{proof}

\begin{rem}\label{r-semistable}
Notice that in the argument above,  the existence of the conical K\"ahler-Einstein  metric on $\cX_{t_i}$ is needed only for an angle $\beta_i<\bT$ instead of $\bT$. So the proof  remains {\em valid} by only assuming that  $\cX_{t}$ is {\em $\bT$-K-semistable} for any $t\in C^{\circ}$ instead of being  $\bT$-K-polystable.
\end{rem}

An immediate consequence is the following.
\begin{cor}\label{c-aut} ${\rm Aut}(X,D)$ is finite. If $\bT=1$, ${\rm Aut}(X)$ is reductive.
\end{cor}
\begin{proof}The first part is just Lemma \ref{l-aut}. The second part follows from \cite[Theorem 6]{CDS3}  thanks to the existence of  weak K\"ahler-Einstein metric on $X$.
\end{proof}

Let
\begin{equation}\label{O}
 \overline{BO}:=\lim_{t\to 0}\overline{{\rm SL}(N+1)\cdot\hilb(\cX_t,\cD_t)}\subset \HH^{\bchi;N}\ .
 \end{equation}
 denote  the limiting orbit and
 $$O_{\hilb(X,(1-\bT)D)}={\rm SL}(N+1)\cdot \hilb(X,(1-\bT)D)\text{ and } \overline{O_{\hilb(X,(1-\bT)D)}} \subset\HH^{\bchi;N}\
 $$
 be the ${\rm SL}(N+1)$-orbit of $\hilb(X,(1-\bT)D)$  and its closure.
 By Corollary \ref{c-aut},  this allows us to construct an {\em $\SL(N+1)$-invariant Zariski open  } neighborhood
 \begin{equation}\label{U}
\hilb(X,(1-\bT)D)\in U\subset  \HH^{\bchi;N}
 \end{equation}
satisfying the condition \eqref{Oz} in  Lemma \ref{disj}.  We  want to remark  that 
the open neighborhood $U$ is {\em independent} of $\bT$ (cf.  part (1) of Remark \ref{remTE}).

Then we have the following

\begin{lem}\label{IVT}
Let   $\{t_i\}\subset C$ be a sequence of points approaching  $0\in C$ and
$$\{\be_i\}, \   \{\be^*_i\} ,\ \{\be^\prime_i\}\subset [\ep,1]$$
be three sequences satisfying $\be_i^\ast <\be_i$ for all $i$.
\begin{enumerate}
\item  Assume $\be_i\to \bT$, $\be^*_i\to \bT$ and that there is a sequence   $\{(\cX_{t_i},\cD_{t_i})\mid (\cX_{t_i},\cD_{t_i})\text{ being } \beta_i\text{-K-polystable}\}$ with $t_i\to 0$ such that
\begin{equation}\label{T-open}
\hilb(\cX_{t_i}, (1-\be_i^*)\cD_{t_i})\stackrel{i\to \infty}{\longrightarrow} {\rm U}(N+1)\cdot \hilb(X, (1-\bT)D)
\end{equation}
and for $g_i\in \UU(N+1)$
\begin{equation}\label{YE}
g_i \cdot \hilb(\cX_{t_i}, (1-\be_i)\cD_{t_i})\stackrel{i\to \infty}{\longrightarrow} \hilb (Y, (1-\bT)E) \ .
 \end{equation}
Then $\hilb(Y,(1-\bT)E)=g\cdot\hilb(X,(1-\bT)D)$ for some $g\in \UU(N+1)$.

\item Assume $\be^\prime_i\nearrow \bT$ and that for any fixed $i$, there is a $g_i\in \rU(N+1)$ such that
\begin{equation}\label{T-close}
\hilb(\cX_{t}, (1-\be^\prime_i)\cD_t)\stackrel{t\to 0}{\longrightarrow} g_i\cdot \hilb(X, (1-\be^\prime_i)D)\ \
\end{equation}
and
\begin{equation}\label{YE-1}
 \hilb(\cX_{t_i}, (1-\beta^\prime_i)\cD_{t_i})\stackrel{i\to \infty}{\longrightarrow} \hilb (Y, (1-\bT)E)\in \overline{BO}\setminus O_{\hilb(X,(1-\bT)D)}\  .
\end{equation}
If $(X,D)\not\cong (Y,E)$, then there exists a sequence $\{t_i^\prime\}$ satisfying $0<\dist_C(t^\prime_i,0)<\dist_C(t_i,0)$ such that
\begin{eqnarray}\label{Y'-1}
&&\hilb(Y',(1-\bT)E')=\lim_{i\to \infty} \hilb(\cX_{t'_i},(1-\be^\prime_i)\cD_{t'_i})\\
 &\in &\left( \overline{O_{\hilb(X,(1-\bT)D)}} \bigcup ( U\cap\overline{BO})\right)\setminus O_{ \hilb(X,(1-\bT)D)} \subset \HH^{\bchi;N}\nonumber \ .
\end{eqnarray}
where $\dist_C:C\times C\to \RR$ is a fixed {\em continuous distance} function  on $C$.
\end{enumerate}
\end{lem}

\bigskip


\begin{proof}[Proof of Lemma \ref{IVT}]
To prove {\em part 1)}, one first notices that \eqref{T-open} together with Lemma \ref{l-GH=CH} imply that $(X,D)$ is $\bT$-K-polystable.  We will show that under the above assumption   and
$$\hilb(Y,(1-\bT)E)\not \in{\rm U}(N+1)\cdot \hilb(X,(1-\bT)D),$$
then one can  construct a {\em new} sequence $\{\be_i^{\prime\prime}\}$ satisfying $\be_i^{\prime\prime} \in [\be_i^\ast, \be_i]$ such that
\begin{eqnarray*}
&&\hilb(Y',(1-\bT)E')=\lim_{i\to \infty} \hilb(\cX_{t_i},(1-\be^{\prime\prime}_i)\cD_{t_i})\\
&\in& \left( \overline{O_{\hilb(X,(1-\bT)D)} }\bigcup ( U\cap \overline{BO})\right)\setminus O_{ \hilb(X,(1-\bT)D)} \subset \HH^{\bchi;N}\nonumber \ .
\end{eqnarray*}
On the other hand, Lemma \ref{l-GH=CH} implies
$$(\cX_{t_i},\cD_{t_i};\omega(t_i,\be^\pp_i))\stackrel{\GH}{\longrightarrow}(Y',E';\omega_{Y'}(\bT)),$$ thus $(Y',E')$ admits weak K\"ahler-Einstein metric with angle $2\pi(1-(1-\bT)/m)$ along $E'$ and hence $\bT$-K-polystable.  These allow one to construct either a test configuration of  $(X,D)$ with central fiber $(Y',E')$  and vanishing generalized Futaki invariant or a test configuration  of $(Y',E')$ with central fiber $(X,D)$ and vanishing generalized Futaki invariant,  contradicting to  the fact that both $(X,D)$ and $(Y',E')$ are $\bT$-K-polystable.  So we must have
$$\hilb(Y,(1-\bT)E)=g\cdot\hilb(X,(1-\bT)D)$$ for some $g\in {\rm U}(N+1)$.

Now we proceed to the construction of $\{\be_i^\pp\}$.
Let
\begin{equation*}\label{ball}
B(\hilb(X,(1-\bT)D),\ep_1)\Subset U
\end{equation*}
 be the radius $\ep_1$  {\em open} balls   with respect to the distance function \eqref{dist} and $U$ be given as in \eqref{U}.

By shrinking the pointed curve  $( 0\in C)$ if necessary,   we may assume that
\begin{equation}\label{T-open'}
\hilb(\cX_{t_i},(1-\be_i^\ast)\cD_{t_i})\in  {\rm U}(N+1)\cdot  B( \hilb(X,(1-\bT)D),\ep_1)
\end{equation}
for all $i$ thanks to our assumption \eqref{T-open}.  On the other hand,  by our assumption that $(X,D)\not\cong (Y,E)$, and we may assume $(Y,E)$ is not in the closure of the orbit of $(X,D)$ (otherwise, we can just let $\be_i''=\be_i$), then there  is an $\ep_1>0$ such that
$$\dist_{\HH^{\bchi;N}}( \hilb(\cX_{t_i},(1-\be_i)\cD_{t_i}), O_{\hilb(X,(1-\bT)D)})>\ep_1 \text{ for } i\gg1 \  .$$
 By the $\beta$-continuity of $\tau(\cdot,t_i)$ for each fixed $i\gg 1$, for any $0<\vep<\ep_1$  there is a
 \begin{equation}\label{bpp}
 \be^{\prime\prime}_{i,k}=\sup \left\{\be \in(\be_i^\ast,\be_i)
\left | \tau(\cdot ,t_i)|_{(\be_i^\ast, \be)}\subset B(O_{\hilb(X,(1-\bT)D)}, \vep/2^k)\cup \UU(N+1)\cdot B (\hilb(X,(1-\bT)D),\ep_1)\right.\right\}\\
 \end{equation}
where $B(O_{\hilb(X,(1-\bT)D)}, \vep/2^k)$ is the $\vep/2^k$-tubular neighbourhood of $O_{\hilb(X,(1-\bT)D)}$,
that is, $\be^{\prime\prime}_{i,k}$ is the smallest $\be$ such that $\tau(\cdot, t_i)$ escapes    $B(O_{\hilb(X,(1-\bT)D)}, \vep/2^k)\cup\UU(N+1)\cdot  B (\hilb(X,(1-\bT)D),\ep_1)$. Clearly, we have  $\be^\pp_{i,k+1}\le\be^\pp_{i,k}$. Now if
$$\tau(\be^{\pp}_{i,0},t_i)\in {\rm SL}(N+1)\cdot B(\hilb(X,(1-\bT)D),\ep_1)$$
we let $\be^\pp_i=\be^\pp_{i,0}$, otherwise, we let $\be^\pp_i=\be^\pp_{i,k}$ where   $\be^\pp_{i,k}$ is the {\em first} number satisfying
$$\tau(\be^{\pp}_{i,k},t_i)\in {\rm SL}(N+1)\cdot B(\hilb(X,(1-\bT)D),\ep_1).$$
    Such $k$ exists because of \eqref{T-open'}. Now by our construction, there is a $g_i\in {\rm SL}(N+1)$ such that
\begin{equation}\label{gi-mi}
\tau(\be^\pp_i ,t_i)\in g_i\cdot B(\hilb(X,(1-\bT)D),\ep_1).
\end{equation}
We let
$$M_i=\inf\{\tr(g^\ast g) \left | g\in \SL(N+1) \text{ such that  \eqref{gi-mi} is satisfied}\right. \}+1 $$
and by passing through a subsequence we may assume $\tr(g_i^\ast g_i)\leq M_i$.
Then we have the following dichotomy:

\smallskip

{\em Case 1.} there is a subsequence $\{M_{i_l}\}$ such that $|M_{i_l}|<M$ for some constant $M$ independent of $i$. Then we claim that
$$\{\tau(\be^\pp_{i_l},t_{i_l})=\hilb(\cX_{t_{i_l}},(1-\be^\pp_{i_l})\cD_{t_{i_l}})\}$$ is the subsequence we want, and  its limit $\hilb(Y',(1-\bT)E')$ lies in
$$(U\cap \overline{BO})\setminus O_{\hilb(X,(1-\bT)D)}.$$ To see this, one only needs to notice that it follows from  our construction of $\be^\pp_{i_l}$ that
$$\dist_{\HH^{\bchi;N}}(\tau(\be^\pp_{i_l},t_i), O_{\hilb(X,(1-\bT)D)})$$
is {\em uniformly} bounded from below by some $\vep/2^{k}$, since there is a $k=k(M)$ such that
 $$
 \left\{\left. z\in \HH^{\bchi;N}\right | \dist_{\HH^{\bchi;N}}(z,g\cdot\hilb(X,(1-\bT)D))\leq \vep/ 2^{k(M)} \text{ and } |g|<M\right \}\subset {\rm SL}(N+1)\cdot U\ .
 $$
\bigskip

 {\em Case 2.} $|M_i|\to \infty$. If that happens, let us replace $\vep$ by $\vep/2$ in \eqref{bpp} and repeat the above process, if for the new sequence $\{M^{[1]}_i\}\subset \RR $ there is a bounded subsequence $\{M_{i_l}^{[1]}\}$ then we reduces to the {\em Case 1}, otherwise, we keep on repeating this process. Then either we stop at  a finite stage or this becomes  an infinite process. If we stop at a finite stage, then we obtain our subsequence as before, if the process never terminates, we claim that we are able to extract a subsequence whose limit $\hilb(Y',(1-\bT)E')$ lands in the {\em boundary}
 $$\partial\overline{ O_{\hilb(X,(1-\bT)D)}}= \overline{O_{\hilb(X,(1-\bT)D)}} \setminus O_{\hilb(X,(1-\bT)D)}.$$ This is because by choosing a diagonal sequence we will have
 $$\dist_{\HH^{\bchi;N}}(\tau(\be^{\pp, [k]}_{i_k},t_{i_k}),O_{\hilb(X,(1-\bT)D)})< \vep/2^k\to 0,
 $$
 so we know $$z:=\lim_{k\to \infty} \tau(\be^{\pp, [k]}_{i_k},t_{i_k})\in \overline{O_{\hilb(X,(1-\bT)D)}}.$$
On the other hand, if  $z\in  O_{\hilb(X,(1-\bT)D}$, then
   $$z=g\cdot \hilb(X,(1-\bT)D)$$
 for some $g\in \SL(N+1)$.
In particular,  $g\cdot B(\hilb(X,(1-\bT)D),\ep_1) $ contains a neighborhood of $z$.
However, this violates the assumption that $|M^{[k]}_{i_k}|\to \infty$ as $k\to \infty$. Hence our proof is completed.

\bigskip

The proof of  {\em part 2)} is similar.  Contrast to the {\em part 1)}, we will vary $t$  instead of $\beta$ in $\tau(\beta,t)$.
First by our assumption \eqref{YE-1} together with Lemma \ref{l-GH=CH},  $(Y,E)$ is $\bT$-K-polystable hence
$$\hilb(Y,(1-\bT)E)\not\in \partial\overline{ O_{\hilb(X,(1-\bT)D)}}\ .$$

So there  is an $\ep_1>0$ such that
$$\dist_{\HH^{\bchi;N}}( \hilb(\cX_{t_i},(1-\be'_i)\cD_{t_i}), O_{\hilb(X,(1-\bT)D)})>\ep_1 \text{ for } i\gg1 \  .$$
On the other hand, by our assumption \eqref{T-close} and Lemma \ref{X-be} we have for any {\em fixed} $\be'_i$ with $i\gg1$,  there is a $0<s_i\in \RR$ such that
$$
\hilb(\cX_{t'_i},(1-\be'_i)\cD_{t'_i})\in \rU(N+1)\cdot B( \hilb(X,(1-\bT)D),\ep_1)
$$
for any $t$ satisfying $0<\dist_C(t,0)<s_i$, since
$$\rU(N+1)\cdot \hilb(X,(1-\be'_i)D)\stackrel{i\to \infty}{\longrightarrow} \rU(N+1)\cdot \hilb(X,(1-\bT)D)$$ inside $\HH^{\bchi;N}/{\rm U}(N+1).$

By the $t$-continuity of $\tau(\be'_i,\cdot)$ for each fixed $i\gg 1$, for any $\vep<\ep_1/2$ there is

 \begin{equation}\label{sp}
 s_{i,k}:=\sup \left\{ s\in [0, |t_i|)
\left | \tau(\be'_i ,\cdot )|_{B_C(0,s)}\subset B(O_{\hilb(X,(1-\bT)D)}, \vep/2^k)\cup\UU(N+1)\cdot B (\hilb(X,(1-\bT)D),\ep_1) \right.\right\}
 \end{equation}
where $|t_i|:=\dist_C(t_i,0)$ and $B_C(0,s):=\{t\in C\mid \dist_C(t,0)\leq s\}$.  Then $s_{i,k}=|t_{i,k}|$ is the smallest distance needed for $t$ so that $\tau(\be'_i, t)$ escapes  $B(O_{\hilb(X,(1-\bT)D)}, \vep/2^k)\cup\UU(N+1)\cdot B (\hilb(X,(1-\bT)D),\ep_1)$.
 Clearly, we have  $s_{i,k+1}<s_{i,k}$. Now if
 $$\tau(\be'_i,t_{i,0})\in {\rm SL}(N+1)\cdot B(\hilb(X,(1-\bT)D),\ep_1)$$ we let $t'_i=t_{i,0}$, otherwise, we let $t'_i=t'_{i,k}$ where   $t'_{i,k}$ is the {\em first} point in $C$ satisfying
 $$\tau(\beta'_{i},t^\prime_{i,k})\in {\rm SL}(N+1)\cdot B(\hilb(X,(1-\bT)D),\ep_1).$$
Such a process must terminate in {\em finite} steps by \eqref{T-close}. Now we define $M_i\in \RR$ to be
$$M_i:=\inf_{g_i}\{\tr(g_i^\ast g_i)+1 |\  \tau(\be'_i ,t'_i)\in g_i\cdot B(\hilb(X,(1-\bT)D),\ep_1)  \}.$$
Then again we have two situations exactly the same as in the proof of part one depending on $\{M_i\}$ being bounded or not.  Replacing $\be^\pp_i$ by $t'_i$ in the argument for Part 1), one see that the rest of  the proof is a verbatim, which we will skip. Thus  the proof of the Lemma is completed.
\end{proof}

\begin{rem}\label{v-IVT}
Notice that when $\bT=1$ and both $\be_i,\be_i^\ast\leq 1,\ \forall i$ then Lemma \ref{IVT} and its proof imply a slight variation of the following form.

Let
\begin{equation}\label{pi1}
\begin{array}{cccc}
 \pi_1:&  \HH^{\bchi;N}=\HH^{\chi;N}\times \PP^{\ti\chi;N} & \longrightarrow &\PP^{\chi,n;N}\ \\
& (\hilb(X),\hilb(D))&\longmapsto & \hilb(X)
\end{array}
\end{equation}
be the projection to the first factor.
\begin{enumerate}
\item  Assume $\be_i\to 1$, $\be^*_i\to 1$ and that there is a sequence $\{(\cX_{t_i},\cD_{t_i})\mid (\cX_{t_i},\cD_{t_i})\text{ being } \beta_i\text{-K-polystable}\}$ with $t_i\to 0$ such that
\begin{equation}\label{T1-open}
\pi_1 (\hilb(\cX_{t_i}, (1-\be_i^*)\cD_{t_i}))\stackrel{i\to \infty}{\longrightarrow} {\rm U}(N+1)\cdot \hilb(X)\subset \HH^{\chi;N}
\end{equation}
and for $g_i\in U(N+1)$
\begin{equation}\label{YE1}
\pi_1( g_i \cdot \hilb(\cX_{t_i}, (1-\be_i)\cD_{t_i}))\stackrel{i\to \infty}{\longrightarrow} \hilb (Y)\in\HH^{\chi;N} \ .
 \end{equation}
Then $\hilb(Y)=g\cdot\hilb(X)$ for some $g\in U(N+1)$.

\item Assume $\be^\prime_i\nearrow 1$ and that for any fixed $i$, there is a $g_i\in \rU(N+1)$ such that
\begin{equation}\label{T1-close}
\hilb(\cX_{t}, (1-\be^\prime_i)\cD_t)\stackrel{t\to 0}{\longrightarrow} g_i\cdot \hilb(X, (1-\be^\prime_i)D)\in \HH^{\bchi;N} \
\end{equation}
and
\begin{equation}\label{YE1-1}
\pi_1( \hilb(\cX_{t_i}, (1-\beta^\prime_i)\cD_{t_i}))\stackrel{i\to \infty}{\longrightarrow} \hilb (Y)\in \overline{BO}\setminus O_{\hilb(X)}\subset\HH^{\chi;N}\  .
\end{equation}
If $X\not\cong Y$, then there exists a sequence $\{t_i^\prime\}$ satisfying $0<\dist_C(t^\prime_i,0)<\dist_C(t_i,0)$ such that
\begin{eqnarray}\label{Y'1-1}
&&\hilb(Y')=\lim_{i\to \infty} \pi_1(\hilb(\cX_{t'_i},(1-\be^\prime_i)\cD_{t'_i}))\\
 &\in &\left( \overline{O_{\hilb(X)}} \bigcup ( U\cap\overline{BO})\right)\setminus O_{ \hilb(X)} \subset \HH^{\chi;N}\nonumber \ .
\end{eqnarray}
where $\dist_C:C\times C\to \RR$ is a fixed {\em continuous distance} function  on $C$.
\end{enumerate}
\end{rem}

Now we are ready to  prove the openness.

\begin{prop}\label{open}
Let  $(\cX,\cD;\sL)\to C$ be  K\"ahler-Einstein degeneration  of  index $(r,\mathfrak{B})$ as in Definition \ref{fano-r} with $r=r(\cX,\cD)$ being the uniform index as in Theorem \ref{t-YTD}(3). Then $\bB_r(\cX,\cD)\subset [\ep,\mathfrak{B}]$ is an open set.
\end{prop}
 \begin{proof}
Let us assume $\bT\in \bB_r(\cX,\cD)$, then by fixing a local basis $\{s_i\}$ for $\pi_\ast \omega_{\cX/C}^{-\otimes r}$  we have
\begin{equation}
\dist_{\HH^{\bchi;N}}(\hilb(\cX_{t},(1-\bT)\cD_{t}),\rU(N+1)\cdot \hilb(X,(1-\bT)D))\longrightarrow 0 \text{ as }t\to 0\ .
\end{equation}
Now we claim that  there is a $\delta>0$ such that   $[\ep,\bT+\delta)\subset \bB_r(\cX,\cD)$. Suppose not, for any $k$, there is a $\bT<\be_{k}<\bT+1/k$ and a sequence  $\{t_{i,k}\}_{k=1}^\infty$:
 $$\hilb(\cX_{t_{i,k}},(1-\be_{k})\cD_{t_{i,k}})\stackrel{i\to \infty}{\longrightarrow} \hilb(Y_k,(1-\be_{k})E_k)\not\in U\subset \HH^{\bchi;N}$$
 with $U\subset \HH^{\bchi;N}$ being the $\SL(N+1)$-invariant Zariski open neighborhood of $\hilb(X,(1-\bT) D)$ constructed in Lemma \ref{disj}, since $(X,D)$ is also $\be_{k}$-K-polystable because of $\be_{k}\in [\ep,\bbe]$ and Lemma \ref{interpolate}. For any fixed  $i$,  we can pick up $k_i\gg 0$ such that
  $$\hilb(\cX_{t_{i,k_i}},(1-\be_{k_i})\cD_{t_{i,k_i}} )\not\in \rU(N+1)\cdot B(\hilb(X,(1-T)D),\ep_1) .$$
Now let us introduce the {\em diagonal} sequence
 $$\{\hilb(\cX_{t_i},(1-\be_i)\cD_{t_i}):=\hilb(\cX_{t_{i,k_i}},(1-\be_{k_i})\cD_{t_{i,k_i}})\}_{i=0}^{\infty}.$$
Then by Theorem \ref{t-YTD}, after passing to a subsequence if necessary,  we obtain a new sequence, which by abuse of notation will still be denoted by $\be_i\searrow \bT$ and $t_i\to 0$, such that
\begin{equation}\label{y-ne-x}
\hilb(\cX_{t_i},(1-\be_i)\cD_{t_i})\longrightarrow \hilb(Y,(1-\bT)E)\not\in O_{\hilb(X,(1-\bT)D)}\ .
\end{equation}
But this violates the first part of Lemma \ref{IVT}(1) with  $\be_i^\ast=\bT\ \forall i$.
\end{proof}

Next we prove the closedness.


\begin{prop}\label{close}
 Let  $(\cX,\cD)\to C$ be a family satisfying the condition of Proposition \ref{open}. Suppose further that $\cX\to C$ is a family of {\em $\mathfrak{B}$-K-polystable} varieties. Then $\bB(\cX,\cD)\subset [\ep,\mathfrak{B}]$ is also closed with respect  to the induced topology, hence $\bB(\cX,\cD)=[\ep,\mathfrak{B}]$.
\end{prop}

\begin{proof}
By our assumption, for every $t\in C^\circ$,  $(\cX_t,\cD_t)$ is a smooth Fano pair with $\cD_t\in |-mK_{\cX_t}|$.  Since $\cX_t$ is {$\mathfrak{B}$-K-polystable, hence it is $\be$-K-{polystable} for $\be\in[\ep,\mathfrak{B}]$ by Lemma \ref{interpolate}}. As $(\cX_t,\cD_t)$ are smooth, by Theorem \ref{t-YTD} and \cite[Proposition 2.2]{SW2012} \cite[Proposition 1.7]{LS2014} it  admits a unique conical K\"ahler-Einstein metric $\omega_t$ solving
$$\Ric(\omega(t,\beta))=\beta \omega(t,\beta)+\frac{1-\beta}{m}[\cD_t]
$$
 with angle $2\pi(1-(1- \beta)/m)$ along $\cD_t$ for any $\beta\in {[\ep,\mathfrak{B}]}$. By Theorem \ref{t-YTD} and definition of $ \bT$, for any {\em fixed} $\beta<\bT$,
we have
$$(\cX_t,\cD_t;\omega(t,\beta))\stackrel{\GH}{\longrightarrow}(\cX_0,\cD_0;\omega(0,\beta)) \text{ as } t\to 0\ .$$
By Lemma \ref{X-be}, for any sequence $\be_i\nearrow \bT$ we have
$$
\dist_{\HH^{\bchi;N}}(\hilb(X,(1-\be_i)D), \rU(N+1)\cdot \hilb(X,(1-\bT)D))\longrightarrow 0\ .
$$

Our goal is  to prove that
$$
\hilb(\cX_{t},(1-\bT)\cD_{t})\longrightarrow\rU(N+1)\cdot \hilb(X,(1-\bT)D)) \text{ as } t\to 0\ .$$
We will argue by contradiction.

Suppose  this is not the case, then  there is a subsequence $\{t_i\}_{i=1}^\infty \subset C,\ t_i\to 0$ as $i\to \infty$ such that
$$\hilb(\cX_{t_i},(1-\bT)\cD_{t_i})\to \hilb(Y,(1-\bT)E)\not\in  {\rm U}(N+1)\cdot\hilb(X,(1-\bT)D)\ .$$
By the continuity of $\tau(\cdot,t_i)$ at $\bT$ for each fixed $i$ (cf. Lemma \ref{tau-cont}), there is a consequence $\{\be'_i\}_{i=1}^\infty\subset (\ep_0, \bT)$ such that $\be'_i\nearrow\bT$ and
\begin{equation}\label{y-ne-x-1}
\hilb(\cX_{t_i},(1-\be'_i)\cD_{t_i})\to \hilb(Y,(1-\bT)E)\not\in \rU(N+1)\cdot  \hilb(X,(1-\bT)D)\text{ as }i\to \infty .
\end{equation}


We claim that $\hilb(Y,(1-\bT)E)\in \overline{BO}\setminus {\rm SL}(N+1)\cdot U$.  Otherwise, $\hilb(Y,(1-\bT)E)\in  U$ then
$$\hilb(X,(1-\bT)D)\in \overline{{\rm SL}(N+1)\cdot \hilb(Y,(1-\bT)E)}\ .$$
But this violates the fact that $(Y, E)$ is $\bT$-K-polystable by  \cite[Theorem 4.2]{Be12}, since
 we can construct is a test configuration of $(Y,E)$ with central fiber $(X,D)$ and vanishing generalized Futaki invariant. Hence our claim is proved.

Now we can apply the second part of Lemma \ref{IVT} to obtain a {\em new} sequence $\{t'_i\}\subset C^\circ$  satisfying $t'_i\to 0\in C$ and
\begin{eqnarray}\label{Y''-1}
&&\hilb(Y',(1-\bT)E')=\lim_{i\to \infty} \hilb(\cX_{t'_i},(1-\be^\prime_i)\cD_{t'_i})\\
 &\in &\left(\overline{O_{\hilb(X,(1-\bT)D)}}\bigcup ( U\cap\overline{BO})\right)\setminus O_{ \hilb(X,(1-\bT)D)} \subset \HH^{\bchi;N}\nonumber \ ,
\end{eqnarray}
which contradicts to the fact that both $(Y',E')$  and $(X,D)$ are $\bT$-K-polystable by the same reason as above. Thus the proof is completed.
\end{proof}

\begin{rem}We remark an interesting point of the proof is that in the proof of Proposition \ref{open}, we have only used the continuity of $\tau(\cdot, t)$ for each fixed $t$. In particular, its continuity of $\tau$ with respect to the variable $t$ is not used.  Contrast to this, the continuity of $\tau(\be,\cdot)$ with respect to $t$ is what we use in the proof of Proposition \ref{close}.
 \end{rem}

 We note that by this point, we have already established the following.
 \begin{cor}\label{r-weak}
Theorem \ref{log-main} holds under  {\em an additional} assumption that $\cX_t$ is $\beta$-K-polystable for all $t\in C^\circ$.
 \end{cor}

\section{K-semistability of  the nearby fibers}\label{s-K-semistable}

\subsection{Orbit of K-semistable points}\label{K-semistable}
In this subsection, we extend our continuity method to study the {\em uniqueness} of K-polystable Fano varieties that a K-semistable Fano manifold can specialize to,  which will also be needed in the proof of our main theorem.

Let $X$ be a smooth Fano manifold, and $D\in |-mK_X|$ be a smooth divisor for $m\ge 2.$ Assume $X$ is $\bT$-K-semistable with respect to $D$.  By Theorem \ref{t-YTD}, we know that for any sequence $\beta_i\nearrow \bT$, after possibly passing to a subsequence (, which by abusing of notation will  still be denoted by $\be_i\nearrow \bT$), there exists a log $\QQ$-Fano pair $(X_0,D_0)$ which is the Gromov-Haussdorf limit of the conical K\"ahler-Einstein metric $(X, D;\omega(\beta_i))$, that is,
$$ \hilb(X,(1-\be_i)D)\longrightarrow \rU(N+1)\cdot\hilb(X_0,(1-\bT)D_0)\in \overline{O_{\hilb(X,(1-\bT) D)}}\text{ as }i\to \infty
$$
with $X_0$ being $\bT$-K-polystable, where
$$\overline{O_{\hilb(X,(1-\bT)D)}}=\mbox {the closure of } {\rm SL}(N+1)\cdot \hilb(X,(1-\bT)D)\subset \HH^{\bchi;N}\ .$$
In particular, $(X_0,D_0)$ admits a weak conical K\"ahler-Einstein metric $\omega(\bT)$ with cone angle $2\pi(1-(1-\bT)/m)$ along the divisor $D_0\subset X_0$.

\begin{lem}\label{l-ksemi}
The limit is independent of the choice of the sequence $\{\be_i\}$ in the sense that for every sequence $\be_i \nearrow \bT$,
$$(X, D;\omega(\be_i))\stackrel{\GH}{\longrightarrow} (X_0,D_0; \omega(\bT)).$$
\end{lem}
\begin{proof}

The existence of a weak conical K\"ahler-Einstein metric $\omega(\bT)$ on $(X_0,D_0)$ allows us to construct a test configuration
$(\cX,\cD;\sL)$ of $(X,D)$ with central fiber $(X_0,D_0)$ since $\aut(X_0,D_0)$ is reductive by Theorem \ref{t-YTD}. Now our claim follows by applying Lemma \ref{IVT} (1) to the family $(\cX,\cD;\sL)$.

\end{proof}

\begin{thm}\label{ss}
Suppose $X$ is a smooth K-semistable Fano manifold and $D_0 \in |-m_0K_X|$ and  $D_1 \in |-m_1K_X|$ are two smooth divisors. Let $X_0$ and $X_1$ be the limits defined as in Lemma \ref{l-ksemi} with $\bT=1$,  then
$X_0\cong X_1$.
\end{thm}
\begin{proof}
By introducing a third divisor in $|-mK_X|$ with $m={\rm lcm}(m_0,m_1)$, we may assume $rm_0=m_1$ for a positive integer $r$.
By Bertini's Theorem, we may choose  $\{D_t\}_{t\in [0,1]}\subset |-mK_X|$ to be a continuous path joining $rD_0$ and $D_1$  such that
\begin{itemize}
\item  the path $\{D_t\}$ lies in an algebraic arc $C\subset |-mK_X|$ with corresponding family $\cD\to C$;
\item  $D_t$ is smooth for all $t\neq 0$.
\end{itemize}
By assumption, $X$ is K-semistable, hence $(X,D_t)$ are $\be$-K-stable for all $(\be,t)\in (0,1)\times (0,1]$. In particular, $\{(X,D_t)\}$ admit conical K\"ahler-Einstein metric $\omega(t,\beta),\ \forall (\be,t)\in (0,1)\times [0,1]$ by Corollary \ref{l-YTD}, using Tian's embedding we can similarly define a map
\begin{equation}\label{sigma}
\begin{array}{cccc}
 \sigma:& (0, 1)\times (0,1] & \longrightarrow &\HH^{\bchi;N}\\
& (\beta, t)&\longmapsto & \hilb(X,(1-\be)D_t) \
\end{array}
\end{equation}
 using a prefixed basis of $H^0(X,\sO_X(-rK_X))$.
By Proposition \ref{cont-path} and \cite[Theorem 2]{Don2011}, $\sigma$ is continuous on $(0, 1)\times (0,1]$. We claim that $q\circ \sigma$ is continuous on $(0,1)\times[0,1]$ with $q\colon \HH^{\bchi;N} \to \HH^{\bchi;N}/ {\rm U}(N+1)$.  For fixed $\be\in (0,1)$, we can deduce  the continuity of  $\sigma(\be,\cdot)$ at $0$ by applying Corollary \ref{r-weak} to the product family $(\cX=X\times C,\cD)\to C$ with  $(\cX_t,\cD_t)=(X,D_t)$.

Thus all we need to show is
\begin{equation}\label{X1}
\lim_{\be\to 1}\dist_{\HH^{\chi;N}}(\ti\si(\be,t), \rU(N+1)\cdot \hilb(X_0))=0,\ \  \forall t\in [0,1]
\end{equation}
where $\ti \sigma:=\pi_1\circ \sigma$ with $\pi_1$ being given in \eqref{pi1}.
To achieve that, let $\ti q:\HH^{\chi;N}\to \HH^{\chi;N}/\rU(N+1)$ then  Lemma \ref{l-ksemi} allows us to introduce
$$
\lim_{\be\to 1}\ti q\circ \ti \si(\be,t)=\rU(N+1)\cdot \hilb(X_t)\in \HH^{\chi;N}/\rU(N+1), \text{ for }t\in [0,1]
$$
with $X_t$ being a $\QQ$-Fano variety admitting weakly K\"ahler-Einstein metric for each $t\in [0,1]$. Let $\cX_1\to \AAA^1$ be a test configuration with  central fiber  $X_1$ and $\hilb(X_1)\in U\subset\HH^{\chi;N}$ be the open neighborhood constructed for the family $\cX_1\to \AAA^1$ via Lemma \ref{disj}.

Now suppose \eqref{X1} does not hold,  i.e. there is a $t_0\in [0,1]$ such that
 $$\displaystyle\lim_{\be\to 1}\ti\si(\be,t_0)=\hilb(X_{t_0})\not\in U\cdot \hilb(X_1)\ .$$
Then by applying the continuity of $\ti q\circ\ti\si(\be,\cdot)$ with respect to $t\in[0,1]$ for fixed $\be$ the same way as in the proof of Lemma \ref{IVT}(2),
we can construct a new sequence $\{(\be_i,t_i)\}_{i=1}^\infty\subset (0,1]\times[t_0,1]$ such that $\be_i\nearrow 1$ as $i\to \infty$ and
 \begin{eqnarray}\label{t''}
&&\hilb(Y)= \lim_{i\to \infty} \ti\si(\be_i,t_i)
 \in \left(\overline{O_{\hilb(X_1)}}\bigcup ( U\cap  \partial \overline{O_{\hilb(X)}})\right)\setminus O_{ \hilb(X_1)} \subset \HH^{\chi;N}\nonumber \ ,
\end{eqnarray}
with  both $X_1$ and $Y$($\not\cong X_1$) being K-polystable, which is impossible. Hence our proof is completed.

\end{proof}

\subsection{Zariski Openness of K-semistable varieties}\label{ss-zopen}
In this section, we will study the Zariski openness of  the locus of  the  {\em $\QQ$-Gorenstein smoothable} $K$-semistable varieties inside Hilbert schemes. This needs a combination of  the continuity method with the algebraic result in Appendix \ref{s-constru}.

Let
 $$
\begin{CD}
(\cX, \cD)@>\iota>> \PP^N\times \PP^N\times S\\
@VV\pi V @VVV\\
S@=S
\end{CD}
$$
be a flat family of {\em $\QQ$-Fano varieties} over a smooth base $S$ (not necessarily complete)and  $\cD\in |-mK_\cX|$ be an irreducible  divisor defined by a section $s_\cD\in \Gamma(S,  \sO_\cX(-mK_\cX))$. Let us assume  further that $\sO_\cX(-rK_\cX)$ is relatively very ample and $\iota$ is the embedding induced by a prefixed basis  $\{s_i(t)\}_{i=0}^N\subset\Gamma(S,\pi_\ast \sO_\cX(-rK_{\cX/S}))$,
in particular $\iota^\ast \sO_{\PP^N}(1)\cong  \sO_\cX(-rK_{\cX/S})$. Then we have the following

\begin{thm}\label{t-open}
Let $(\cX,\cD)\to C$ be the family over a {\em smooth curve} such that  $(\cX_t,\cD_t)$ is smooth for  $t\in C^\circ$ and $(\cX_t,\frac{1}{m}\cD_t)$ is a klt for all $t\in C$. Assume  $(\cX_0, \cD_0)$ is $\mathfrak{B}$-K-semistable. Then there is a Zariski open neighborhood $0\in C^*\subset C$ such that $(\cX_t,\cD_t)$ is $\mathfrak{B}$-K-semistable for $t\in C^*$. Furthermore, if $(\cX_0, \cD_0)$ is $\bbe$-K-polystable  and  has only finitely many automorphisms, then $(\cX_t, \cD_t)$ is $\mathfrak{B}$-K-polystable after a possibly further shrinking of $C^*$.
\end{thm}
\begin{defn} For every $t\in S$, we  define  the {\it K-semistable threshold} as follows
$$\kst(\cX_t, \cD_t):=\sup \left\{\beta \in [0,\mathfrak{B}] \left | \ \ (\cX_t, \cD_t) \mbox{ is $\beta$-K-semistable} \right. \right\}. $$
By Theorem \ref{t-YTD},   testing $\be$-K-semistability for $\cX_t, \ \forall t\in S$ is reduced  to test  for all  1-PS  inside ${\rm SL}(N+1)$ for a fixed sufficiently large $\mathbb{P}^N$. This implies that $\kst(\cX_t,\cD_t)$ is a  constructible function of $t $ (cf. Proposition \ref{p-con} below).
By Remark  \ref{r-small}, we know $(\cX_t, \cD_t)$ is $\beta$-K-stable for all $\beta\in (0,\beta_0]$. This together with Lemma \ref{interpolate} in particular imply that $\kst(\cX_t,\cD_t)$ is actually a  {\em maximum} for every $t\in S$.
\end{defn}
Then we have the following Proposition which is essentially follows from Paul's work, especially his theory on stability of pairs (see \cite[Theorem 1.3]{Pa2012}).  For reader's convenience, a proof will be included in the Section \ref{s-constru}, Proposition \ref{be-constr}.
\begin{prop}\label{p-con}
${\rm kst}(\cX_t,\cD_t)$ defines a constructible function on $S$, i.e. $S=\sqcup_i S_i$ is a union of {\em finite} constructible sets $\{S_i\}$, on which   ${\rm kst}(\cX_t,\cD_t)$ are  constant.
\end{prop}

\begin{proof}[Proof of Theorem \ref{t-open}]
By Proposition \ref{p-con}, $\kst(\cX_t,\cD_t)$ is constant when restricted to each strata $S_i$.  So all we need  is that if $t_i\to 0$ and $(\cX_{t_i},\cD_{t_i})$ {\em strictly} $\bT$-K-semistable then
$$\bT= {\rm kst}(\cX_{t_i}, \cD_{t_i})\ge  {\rm kst}(X, D)=\mathfrak{B}.$$

Suppose this is not the case,  we have $\mathfrak{B}>\bT$ and we seek for a contradiction. First, we claim for any sequence $t_i\to 0$, after passing to a subsequence which by abusing of notation still denoted by $\{t_i\}$, we can find a sequence
$\{\be_i^\ast\}\nearrow \bT$ such that
 \begin{equation}\label{b-ast}
 \dist_{\HH^{\bchi;N}}(\hilb(\cX_{t_{i}}, (1-\be_i^\ast)\cD_{t_{i}}), \rU(N+1)\cdot \hilb(X,(1-\bT)D)) \longrightarrow 0\ .
 \end{equation}

In fact, since we have already established Theorem \ref{log-main} under  the extra assumption that the nearby points are all $\beta$-K-polystable (see Corollary \ref{r-weak}), for any fixed $\be<\bT$ we have
$$
\dist_{\HH^{\bchi;N}}(\hilb(\cX_{t}, (1-\be)\cD_{t}),\rU(N+1)\cdot \hilb(X, (1-\be)D))\longrightarrow 0 \text{ as } t\to 0,$$
thus Lemma \ref{X-be} implies that
$$
\dist_{\HH^{\bchi;N}}(\hilb(X,(1-\be'_i)D),\rU(N+1)\cdot \hilb(X,(1-\bT)D))\longrightarrow 0
$$
for any sequence $\be'_i\nearrow \bT<\mathfrak{B}$.  Since $t_i\to 0$, for any fixed $\be'_i$ there is a $k_i\ge i$ such that
$$\dist_{\HH^{\bchi;N}  }(\hilb(\cX_{t_{k_i}},(1-\be_i)\cD_{t_{k_i}}), \rU(N+1)\cdot \hilb(X,(1-\be_i)D))<1/i .$$
Now we pick the subsequence $ \{t_{k_i} \}$ and define $\be_{k_i}^*:=\be'_i $, then the sequence $\{\be_{k_i}^*\}_{i\to \infty}\nearrow \bT$ is a sequence satisfying \eqref{b-ast}, hence our claim is  justified.

On the other hand, for each fixed  $t_i$, let $\be\nearrow \bT$. By Theorem \ref{t-YTD},  we have
\begin{equation}\label{be-T}
\dist_{\HH^{\bchi;N}}(\hilb(\cX _{t_i},(1-\be)\cD_{t_i}), \rU(N+1)\cdot \hilb(\ti\cX_{t_i},(1-\bT)\ti \cD_{t_i}))\longrightarrow 0
\end{equation}
with $\hilb(\ti\cX_{t_i},(1-\bT)\ti\cD_{t_i})\in  \partial\overline{ O_{\hilb(\cX_{t_i},\cD_{t_i})}}$ and $(\ti\cX_{t_i},(1-\bT)\ti \cD_{t_i})$ being a  $\bT$-K-polystable variety.
Now we claim that
\begin{equation}\label{be-K-poly}
\hilb(\ti\cX_{t_{i}}, (1-\bT)\ti\cD_{t_{i}})\longrightarrow g\cdot \hilb(X,(1-\bT)D) \text{ for some }g\in {\rm U}(N+1)\ .
\end{equation}
To see this, one notices that by Theorem \ref{t-YTD} and Lemma \ref{l-GH=CH} after passing to a subsequence
there is a sequence $\be_i\nearrow \bT$ such that
$$\hilb(\ti\cX_{t_{i}}, (1-\be_i)\ti\cD_{t_{i}})\longrightarrow  \hilb(Y,(1-\bT)E), $$
such that $(Y,E)$ is $\bT$-K-polystable.
Moreover, we may assume $\be_i^\ast<\be_i,\forall i$ after rearranging.
Combining \eqref{be-T} and Lemma \ref{l-GH=CH},
 we have
$$(\cX_{t_{i}}, \cD_{t_{i}};\omega(t_i,\be_i))\stackrel{\GH}{\longrightarrow}(Y,E;\omega_Y(\bT)),$$
where $(Y,E)$ is a log $\QQ$-Fano pair admitting a weak conical K\"ahler-Einstein metric $\omega_Y(\bT)$ with angle $2\pi(1-(1-\bT)/m)$ along $E$. In particular, $(Y,E)$ is $\bT$-K-polystable.
By Lemma \ref{IVT}(1),  we conclude that
$$\hilb(Y,(1-\bT)E)=g\cdot \hilb(X,(1-\bT)D) \text{ for some }g\in {\rm U}(N+1)\ .$$ Hence our claim is proved.

%

To conclude the proof, we notice that the stabilizer group of $\hilb(\cX'_{t_i},(1-\bT)\cD'_{t_i})$ is of positive dimension for each $i$. Let $\fg=\fsl(N+1)$ be the Lie algebra.  By the upper semicontinuity of the dimension of the stabilizer $\fg_{\hilb(\cX'_{t_i},(1-\bT)\cD'_{t_i})}$, we must have $\dim \fg_{\hilb(X,(1-\bT)D)}>0$ contradicting to the fact that the automorphism group of $(X,D)$ is finite for $\bT<\mathfrak{B}\leq 1$ (see Corollary \ref{c-aut}). To prove the last part of  the statement, we just notice that  under our assumption $(\cX'_t,\cD'_t)$ has to have finite automorphism groups, which implies
$$(\cX'_t,\cD'_t)\cong(\cX_t,\cD_t).$$
Hence our proof is completed for this case.

\end{proof}


\subsection{Proof of Theorem \ref{main} and \ref{log-main}}\label{ss-proof}

Before we start the proof, let us fix a divisor $\cD\sim_C-mK_{\cX}$ in {\em general position} for the flat family $\cX\to C$ satisfying the assumption of Theorem \ref{unique} and $(\cX_t,\cD_t)$  being smooth for all $t\in C^\circ$.
\begin{proof}[Proof of Theorem \ref{main}]
First, we notice that (i) is proved in Section \ref{ss-zopen}.

To prove (ii), one notices that  Theorem \ref{t-YTD} implies that there exists an $r$, such that the Gromov-Hausdorff limit of the family $(\cX_t,\cD_t; \omega(t,\beta_t))$ for any $t\in C$ and $\beta<1$ can all be embedded into $\mathbb{P}^N$ for $N=N(r,d)$.
By putting  Proposition \ref{open} and \ref{close} together, we obtain that  for every $\bbe <1$,
$$\bB_r(\cX,\cD)= [\ep,\mathfrak{B}]$$
for $(\cX,\cD)$ (See Corollary \ref{r-weak}). Therefore, their union will contain $[\ep, 1)$. In particular,  it follows from Lemma \ref{X-be} and Remark \ref{r-semistable} for $\bbe=1$ that $X=\cX_0$ admits a K\"ahler-Einstein metric. This in particular verifies the first part of (iii).

Now we finish the proof of  part (ii).  By part (i), after a possible shrinking of $C$, we may assume that $\cX_t$ is K-semistable for every $t\in C^\circ$.  For any $t\ne 0$, there is a {\em unique} K-polystable $\QQ$-Fano $\tilde{\cX}_t$ such that
$\hilb(\tilde{\cX}_t)\in \overline{O_{\hilb(\cX_t)}}$ by Theorem \ref{ss}, which is  the Gromov-Hausdorff limit of $(\cX_t,\cD_t; \omega(\be))$ as $\be\to 1$ and hence admits a weak K\"ahler-Einstein metric $\tilde{\omega}(t)$ by Theorem \ref{log-main}.

We claim that
 \begin{equation}\label{Y=X}
 \dist_{\HH^{\chi;N}}(\rU(N+1)\cdot \hilb(\ti\cX_t), \rU(N+1)\cdot \hilb(X))\longrightarrow 0, \text{ as } i\to \infty,
 \end{equation}
 and hence part (ii) follows.
To prove that, let $t_i\to 0$ be {\em any } sequence. It follows from the {\em compactness} of Hilbert scheme of $\PP^N$  that  after passing to a subsequence if necessary we may assume
$$\hilb(\tilde{\cX}_{t_i})\longrightarrow \hilb(Y) \text{ as } t_i\to 0\ .$$
 Since
$$(\cX_{t_i},\cD_{t_i}; \omega(t_i,\be))\stackrel{\GH}{\longrightarrow} (\tilde{\cX}_{t_i};\tilde{\omega}(t_i)) \text{ as }\be\nearrow 1,$$
 by Theorem \ref{ss}, there is a sequence $\be_i\nearrow 1$ such that
$$\dist_{\HH^{\chi;N}}(\pi_1\circ \hilb(\cX_{t_i},(1-\be_i)\cD_{t_i}), \rU(N+1)\cdot  \hilb(\tilde{\cX}_{t_i}))<1/i, $$
where $\pi_1$ is given in \eqref{pi1}.
In particular, by passing to another subsequence if necessary, we may assume
$$(\cX_{t_i},(1-\be_i)\cD_{t_i}; \omega(t_i,\be_i))\stackrel{\GH}{\longrightarrow} (Y,\omega_Y)$$
 by Lemma \ref{l-GH=CH}, where $Y$ is  a $\QQ$-Fano variety admitting a weak K\"ahler-Einstein metric $\omega_Y$.
This implies that
\begin{equation}\label{be--}
\dist_{\HH^{\chi;N}}(\pi_1\circ \hilb(\cX_{t_i}, (1-\be_i)\cD_{t_i}), \rU(N+1)\cdot \hilb (Y))\longrightarrow 0 \ ,\text{ as }i\to \infty.
\end{equation}

On the other hand, by Lemma \ref{interpolate} we know  $(\cX_{t_i},\cD_{t_i})$ is $\be$-K-polystable for any $\be<1$. This together with  Corollary \ref{r-weak}  imply that for every fixed $\be<1$
$$
\dist_{\HH^{\bchi;N}}(\hilb(\cX_{t_i}, (1-\be)\cD_{t_i}),\rU(N+1)\cdot \hilb(X, (1-\be)D))\longrightarrow 0 \text{ as } i\to \infty\ .
$$
Therefore, for any fixed $\be_i$ there is a $k_i>i$ such that
$$\dist_{\HH^{\bchi;N}  }(\hilb(\cX_{t_{k_i}},(1-\be_i)\cD_{t_{k_i}}),\rU(N+1)\cdot \hilb(X,(1-\be_i)D))<1/i\ .$$
On the other hand, Lemma \ref{X-be} implies that
$$\dist_{\HH^{\chi;N}}(\pi_1\circ\hilb(X,(1-\be)D),\rU(N+1)\cdot \hilb(X))\longrightarrow 0 \text{ as  }\be \to 1\ .
$$
 These imply that if
we define $\be_{k_i}^*:=\be_i<\be_{k_i}$ then $\be_{k_i}^\ast \to 1$ and
\begin{equation}\label{be-ast}
\dist_{\HH^{\chi;N}}(\pi_1\circ\hilb(\cX_{t_{k_i}}, (1-\be_{k_i}^*)\cD_{t_{k_i}}),\rU(N+1)\cdot \hilb(X))\longrightarrow 0 \text{ as } i\to \infty \ .
\end{equation}
By putting together \eqref{be-ast} and \eqref{be--}, and applying Remark \ref{v-IVT} (1),
we  conclude  that $\hilb(Y)\in {\rm U}(N+1)\cdot \hilb(X)$, and \eqref{Y=X} is established. Thus the proof of part (ii) is completed.

Finally, to finish the proof of part (iii), we can assume  $\cX_t$ is K-polystable for all $t\in C$ by Theorem \ref{t-open}, then by taking $\mathfrak{B}=1$ we can  conclude that $\bB_r(\cX,\cD)= [\ep,1]$. In particular,
$(\cX_{t_{i}};\omega(t_{i}))\stackrel{\GH}{\longrightarrow} (\cX_0; \omega_{\cX_0}).$
Hence our proof is completed.

\end{proof}
{
\begin{proof}[Proof of Theorem \ref{log-main}] Choose a sequence $\be\nearrow \mathfrak{B}$. Applying Proposition \ref{open} and \ref{close}, we obtain that $\bB_r(\cX,\cD)= [\ep,\mathfrak{B}]$. Then by repeating the argument completely parallel to the one given above, we obtain the conclusion.
\end{proof}
}

\begin{rem}\label{r-smoothable}
We call a $\mathbb{Q}$-Fano variety to be {\em $\QQ$-Gorenstein smoothable} if there is a projective flat family $\cX$ over a smooth curve $C$ such that $K_{\cX}$ is $\mathbb{Q}$-Cartier, anti-ample over $C$, a general fiber $\cX_t$ is smooth and $X\cong \cX_0$ for some $0\in C$.  We note that by a standard argument, we can generalize Theorem \ref{main}, \ref{ss} and \ref{t-open} to the case that the base is of higher dimension.  As a consequence, we can just assume in these theorems that the general fibers are $\QQ$-Gorenstein smoothable instead of smooth. These extensions  will be frequently used in Section \ref{ss-luna}.
\end{rem}
\section{Local geometry near a $\QQ$-Gorenstein smoothable K-polystable $\QQ$-Fano variety}\label{ss-luna}

In this section, we will devote to the proof of Theorem \ref{t-good}
 based on Theorem \ref{main}, the results in Section \ref{s-K-semistable}, and the following criterion.

\begin{thm}[{\cite[Theorem 1.2 ]{AFSV14}}]\label{t-AFScriterion}
Let $\cX$ be an algebraic stack of finite type over $\CC$. Suppose that:
\begin{enumerate}
\item for every closed point $x\in \cX$ , there exists a {\em local quotient presentation} $f:\cW\longrightarrow \cX$ (see \cite[Definition 2.1]{AFSV14}) around
$x$ such that:
\begin{enumerate}
\item  the morphism $f$ is {\em stabilizer preserving} (see \cite[Definition 2.5]{AFSV14}) at closed points of $\cW$, and
\item the morphism $f$ sends closed points to closed points; and
\end{enumerate}
\item for any point $x \in X (\mathbb{C})$, the closed substack $\overline{\{x\}}$ admits a good moduli space.
\end{enumerate}
Then $\cX$ admits a good moduli space {as an algebraic space}.
\end{thm}

Let us fix our notation.

\begin{defn}\label{Z0}
We define
\begin{equation}\label{Z}
Z:=\left \{\hilb(Y)\left | \mbox{\stackbox[v][m]{{\footnotesize
$Y\subset \PP^N$ be a smooth Fano manifold with
$N=\dim H^0(Y,\sO_Y(-rK_Y))$, \\ $\left. \sO_{\PP^N}(1)\right|_Y\cong \sO_Y(-rK_Y)$
and $\chi\left(Y,\left.\sO_{\PP^N}(k)\right|_Y\right)=\chi(k)$.
}}}\right .\right \}\subset\HH^{\chi;N} {\subset\PP^M}\ ,
\end{equation}
{where the last inclusion is the Pl\"uker embedding.}
By the boundedness of smooth Fano manifolds with fixed dimension (see \cite{KMM92}), we may choose $r\gg1$ such that $Z$ includes all such Fano manifolds. Now let $\overline{Z}\subset \HH^{\chi;N}$ be the closure of $Z\subset \HH^{\chi;N}$ and $Z^\circ$ be the open  set of $\overline{Z}$ that parametrizes the K-semistable $\mathbb{Q}$-Fano subvariety $Y$ (see Theorem \ref{t-open}) such that $\sO_Y(-rK_Y)\sim \sO_{\PP^N}(1)|_{Y}$ (cf. Lemma 1.19 in \cite{Vie83}).
Let $Z^*$ be the semi-normalization of $Z_{\rm red}^\circ$ which is the reduction of $Z^\circ$.
\end{defn}
Then we have a commutative diagram
 \begin{equation}\label{Z*}
\begin{CD}
{\cX}^\ast@>i>>\PP^N\times Z^\ast @>>> \PP^N\times  {Z_{\rm red}^\circ}\\
@V\pi VV @VVV@VVV\\
{Z^\ast}@>>>Z^\ast@>>>{Z_{\rm red}^\circ}
\end{CD}
\end{equation}
where $\cX^\ast $ is the universal family over ${Z}^*$ (see \cite[Section I.3]{Kol96}).
\begin{rem}
{If we choose $r$ sufficiently divisible},
by Theorem \ref{t-YTD}, the Gromov-Hausdroff limit of Fano K\"ahler-Einstein manifolds is automatically in $Z^\circ$ and hence so are the $\QQ$-Gorenstein smoothable K-polystable $\QQ$-Fano varieties.
\end{rem}

\Blue{ Our goal  is to prove the quotient stack
$[ Z^\ast/\SL(N+1)]$ admits a {\em good moduli space} in the sense of \cite[Section 1.2]{Alp13}, which is a proper scheme when $r$ is sufficiently divisible (We indeed get a slightly stronger statement that the quotient is a {\em scheme} rather than algebraic space.).
By Theorem \ref{t-AFScriterion}, to achieve this what we need to show is that
for any closed point $[z_0=\hilb(X)]\in [Z^\ast/\SL(N+1)]$, there is a  quotient presentation
$$\mathcal{W}:=[\spec A_{z_0}/G_{z_0}=\aut(X)]\longrightarrow [Z^\ast/\SL(N+1)]$$ for some finite type $\CC$-algebra $A_{z_0}$, which satisfies Condition (1) Theorem \ref{t-AFScriterion}.  This is given by Theorem \ref{K-luna}.
Then we spend the main body of the remaining section to prove that for any $\mathbb{C}$-point $z\in Z^\ast$ specializing to $z_0$ under $\SL(N+1)$-action, the closure $\overline{\{[z]\}}\subset [Z^\ast/\SL(N+1)]$ of substack $[z]$ inside $[Z^\ast/\SL(N+1)]$  admits a  good moduli space. To get this,  we will  establish a stronger statement saying that the local presentation can be chosen to be finite \'etale over its image,  which also yields that the quotient is a {scheme}.  This completes the proof of Theorem \ref{t-good}. .}

Let us  first  state the  following boundedness result which is a  consequence of our Theorem \ref{main}.
\begin{lem}\label{K-bdd}
The K-semistable $\QQ$-Fano varieties admitting a $\QQ$-Gorenstein smoothing with a fixed dimension form a bounded family.
\end{lem}
\begin{proof}We first prove for the statement for  K-polystable $\QQ$-Fano varieties.
Let $X$ be an $n$-dimensional $\QQ$-Gorenstein smoothable K-polystable $\QQ$-Fano variety and $\cX\to C$ be a smoothing of $X$ with $\cX_0=X$. It follows from Theorem \ref{main} that nearby fibers $\cX_t$ are all K-semistable, and we can take a $\mathcal{D}\sim_C -mK_{\cX/C}$, such that $\cX_0$ is the Gromov-Hausdorff limit of  $(\cX_{t_i}, (1-\beta_i)\cD_t)$ for any sequences $t_i\to 0$ and $\beta_i\to 1$.

On the other hand, by the boundedness of smooth Fano varieties, we know that there exists $m_0$ depending only on $n$, and a divisor
$$\cD^*\sim_{C^\circ}-m_0K_{\cX^{\circ}/C^{\circ}},$$
 such that $\cD^*_t$ is smooth for any $t\in C^{\circ}$ {after a possible shrinking of the base}.
 Since all $\cX_{t}$ are K-semistable, they admit conical K\"ahler-Einstein metrics $\omega(t,\beta_i)$ with cone angle $2\pi(1-(1-\be_i)/m)$ along $\cD_{t}^*$.
 By applying Theorem \ref{log-main}(iii) for $(\cX_t, (1-\beta_i)\cD^*_t)$, we know that  the Gromov-Hausdorff limit for this family as $t\rightarrow 0$ is also $\cX_0$. Thus it is a subvariety of a fixed $\mathbb{P}^N$ for some $N\gg 0$ by Theorem \ref{t-YTD}.

In general, if $X$ is $\QQ$-Gorenstein smoothable K-semistable $\QQ$-Fano variety, then we know that the closure of its orbit contains a unique K-polystable $\QQ$-Fano variety $X_0$(cf. Theorem \ref{ss} and Remark \ref{r-smoothable}). And as a consequence of volume convergence for Gromov-Hausdorff limit, we obtain that
$$(-K_{X_0})^n=(-K_{X})^n$$
 are bounded from above; on the other hand the Cartier index of $K_{X}$ divides the Cartier index of $K_{X_0}$,
which is also bounded from above thanks to work of \cite[Theorem 1.2]{DS2012}. Therefore $X$ is contained in a bounded family (see e.g. \cite[Corollary 1.8]{HMX2014}).
\end{proof}

Let $X$ be a K-polystable $\QQ$-Fano variety  parametrized by a point in $Z^*$, so it is $\QQ$-Gorenstein smoothable by the definition of $Z^*$,
so it admits a weak K\"ahler-Einstein metric by Theorem \ref{main}, from which we deduce that $\aut(X)\subset \SL(N+1)$
is reductive.  Let $\hilb(X)$ be the Hilbert point for the Tian's embedding of $X\subset\PP^N$ after we fix a basis of $H^0(\sO_X(-rK_X))$. {Let $\HH^{\chi;N}\subset \PP^M$ be the Pl\"ucker's embedding which is clearly $\SL(N+1)$-equivariant.} Then by \cite[Proposition 1]{Don2012} or the proof of Lemma \ref{disj}, there is an $\aut(X)$-invariant linear subspace
$
z_0:=\hilb(X)\in \PP W\subset \PP^M
$
so that
\begin{equation}\label{PWW}
\PP^M=\PP( W\oplus \CC z_{0}\oplus
\faut(X)^{\bot }) \text{ with } \faut(X)^\bot\oplus \faut(X)=\fsl(N+1),
\end{equation}
where $W\oplus \CC\cdot z_0\oplus  \faut(X)^{\bot
}=\CC^{M+1}$ is a decomposition as $\aut(X)$-invariant
subspaces.

 In particular, this induces a representation $\rho:\aut(X)\to \SL(W)$. On the other hand, $\hilb(X)$ is fixed by $\aut(X)$. We let $\rho_X:\aut(X)\to \GG_m$ denote the character corresponding to the linearization of $\aut(X)$ on $\sO_{\HH^{\chi;N}}(1)|_{\hilb(X)}$ induced from the embedding $\aut(X)\subset \SL(N+1)$. Then we can introduce the following
\begin{defn}
A point $z\in \PP W $ is {\em GIT-polystable (resp. GIT-semistable)} if $z$ is {\em polystable(resp. semistable)} with respect the linearization $\rho\otimes \rho_X^{-1}$ on
$\sO_{\PP W}(1)\to \PP W$ in the GIT sense.
\end{defn}

{Before we apply the results in Section \ref{s-apendix}, notably Theorem \ref{S-pres} and Lemma \ref{fini}, to finish our proof of Theorem \ref{t-good}. Let us review the geometric consequences we obtained so far.}
{\begin{summ}\label{sum-slice}
Let us consider the set $\Sigma\subset\HH^{\chi,N}$ of  Hilbert points  corresponding to $\QQ$-Gorenstein smoothable K\"ahler-Einstein $\QQ$-Fano varieties via Tian's embedding.
By \cite{DS2012}, Theorem \ref{main} and results in Section \ref{s-K-semistable},
$\Sigma$ is  a  $\UU(N+1)$-invariant  compact subset that fits into the following diagram:
\begin{equation}\label{Sigma}
\xymatrix{
\Sigma\subset (Z^\ast)^\kps\ \  \ar@{>}[r] \ar@{>}[d]  & \HH^{\chi;N} \ar@{^{(}->}[r]^{\text{Pl\"ucker}} &\PP^M \ar@{>}[d]\\
     \Sigma/\UU(N+1)   \ar@{^{(}->}[rr]     & &\ \  \PP^M/\UU(N+1)\ ,
            }
\end{equation}
where $(Z^\ast)^\kps\subset Z^\ast$ denotes the locus of K-polystable points in $Z^\ast$,
and   there is a bijection between the quotient $\Sigma/U(N+1)$ and all isomorphic classes of $\QQ$-Gorenstein smoothable K\"ahler-Einstein $\QQ$-Fano varieties.
Moreover, we have
$\aut(X)=(\aut(X)\cap \UU(N+1))^\CC$ for all $\hilb(X)\in \Sigma$ (see Lemma \ref{G-U}) and $\Sigma$ satisfies Assumption \ref{bdd} in Section \ref{s-apendix}, \Red{since $\Sigma$ intersects each broken orbit  $\overline{BO}_z$ (cf. \eqref{Z*-chow}) at a {\em unique} $\UU(N+1)$-orbit.}
\end{summ}

\begin{lem}\label{G-U} Let $X$ be a $\QQ$-Gorenstein smoothable $\QQ$-Fano variety admitting weak K\"ahler-Einstein metric.  Then $\aut(X)=(\isom(X))^\CC$.  In particular, $\aut(X)=(\aut(X)\cap \UU(N+1))^\CC$.
\end{lem}
\begin{proof}
It follows from the proof of Theorem  4 in \cite{CDS3}.
\end{proof}

Our {first} main result of this section is the following:
\begin{thm}\label{K-luna}
There {are} an ${\rm Aut}(X)$-invariant  linear subspace $\PP W\subset {\PP^M}$ and an {$\aut(X)$-invariant} Zariski open neighborhood $\hilb(X)\in U_W\subset \PP W\times_{{\PP^M}} Z^*$ such that for any $\hilb(Y)\in U_W$, $Y$ is K-polystable if and only if $\hilb(Y)$ is GIT-polystable with respect to ${\rm Aut}(X)$-action on  $\PP W\times_{{\PP^M}} Z^*$.

Moreover, for all GIT-polystable $\hilb(Y)\in U_W$, we have $\aut(Y)<\aut(X)$, i.e. the local GIT presentation $U_W\sslash \aut(X)$ is {stabilizer preserving} in the sense of {\cite[Definition 2.5]{AFSV14}}.
\end{thm}

\begin{rem}
As we will see in Corollary \ref{semi-sp} that we are able to establish the stabilizer preserving property for all {\em GIT-semistable} $\hilb(Y)\in U_W$. This  property is {\em stronger} than the condition of  being {\em strongly} \'etale  introduced in \cite[Definition 2.5 ]{AFSV14}.
\end{rem}

Let
\begin{equation}\label{diag}
\begin{array}{cccc}
 \Delta:& Z^\ast& \longrightarrow & \HH^{\chi;N}\times Z^\ast\\
& z &\longmapsto & (z,z) \ .
\end{array}
\end{equation}
be the diagonal morphism, we define $O_{Z^\ast}:=\SL(N+1)\cdot \Delta (Z^\ast)\subset \HH^{\chi;N}\times Z^\ast$ where $\SL(N+1)$ acts {\em trivially} on $Z^\ast$ and   acts on $\HH^{\chi;N}$ via the action induced from $\PP^N$.  This allows us to construct the family of limiting orbits space associated to the family \eqref{Z*} as following:
\begin{equation}\label{Z*-chow}
\begin{CD}
\overline{BO}_z  @.\subset \overline{BO}_{Z^\ast}@>i>> \HH^{\chi;N}\times Z^\ast\\
@VVV @VVV @VV\pi_{Z^\ast}V\\
z@. \in  {Z^\ast}@={Z^\ast}
\end{CD}
\end{equation}
with $\overline{BO}_{Z^\ast}\subset  \HH^{\chi;N}\times Z^\ast$ be the closure  of  $O_{Z^\ast}$ and $\overline{BO}_z$ is the union of limiting {\em broken orbits}. Then by Theorem \ref{main} we know that there is a {\em unique} K-polystable orbit inside $\overline{BO}_z$. To see this, one only needs to notice that for any $z\in Z^{\ast}$,  we can always find a smooth curve $f:C\to Z^*$ that passes through $z$ and the image  $f(C)$ meets the dense open locus inside of $Z^\ast$ corresponding to  {\em smooth K-polystable Fano manifolds} with the {\em maximal} dimension of its $\SL(N+1)$-orbit space. Then our claim follows by applying Theorem \ref{main} to the pull back family over $C$.

For a K-polystable point $\hilb(X)\in Z^\ast$ (corresponding to the Tian's embedding of $X\subset \PP^N$ with respect to the K\"ahler-Einstein metric), 
by Lemma \ref{disj}, we can find a Zariski neighborhood $\hilb(X)\in U\subset Z^*$ and after a possible shrinking  we may assume
\begin{equation}\label{U-min}
U\cap \overline{BO}_{\hilb(X)} \text{ contains a  {\em unique minimal}  (cf. Lemma \ref{disj}) orbit }\SL(N+1)\cdot \hilb(X)\ .
\end{equation}
By Theorem \ref{ss}   (and its extension in  Remark \ref{r-smoothable}), every $z\in U$ can be specialized to a K-polystable point $\hat z$  unique up to $\SL(N+1)$-translation.  Moreover, we have the following

\begin{lem}\label{K-near}
Let  $\hilb(X)\in U\subset Z^\ast $ be as above, then  there is an  {\em analytic} open neighborhood $\hilb(X)\in U^{\rm ks}$ such that for any K-semistable points $z\in U^{\rm ks}$, we can specialize it to a K-polystable point $\hat z\in U$  via a 1-PS $\lam\subset\SL(N+1)$. Moreover, if $\displaystyle \lim_{i\to \infty} z_i=\hilb(X)$, then
$$\lim_{i\to\infty}\dist_{\HH^{\chi;N}}(\hilb(\cX_{\hat z_i},\omega_\KE(\hat z_i)), \rU(N+1)\cdot  \hilb(X))=0.$$
where $\hilb(\cX_{\hat z_i},\omega_\KE(\hat z_i))$  is the Hilbert point corresponding to the Tian's embedding of $\cX_{\hat z_i}$ with respect to the weak K\"ahler-Einstein metric $\omega_\KE(\hat z_i)$.
\end{lem}
\begin{proof} Suppose this is not the case, there is a sequence $z_i=\hilb(\cX_{z_i})\stackrel{i\to \infty}{\longrightarrow} \hilb(X)$ and
$$O_{\hat z_j}\cap U=\varnothing \text{ with } O_{\hat z}:=\SL(N+1)\cdot \hat z\ .$$
In particular, by equipping each $\cX_{\hat z_i}$ with a weak K\"ahler-Einstein metric $\omega_\KE(\hat z_i)$, and taking the Gromov-Hausdorff limit $Y$, which is also embedded in $\PP^N$ by Lemma \ref{K-bdd}, we obtain
$$\hilb(\cX_{\hat z_i},\omega_\KE(\hat z_i))\stackrel{i\to \infty}{\longrightarrow}g\cdot\hilb(Y)\in \overline{BO}_{\hilb(X)}\setminus U \mbox{ for some }g\in \rU(N+1)$$
contradicting to the fact the  limiting broken orbits $\overline{BO}_z$  contains a unique K-polystable orbit.
\end{proof}


Now we are ready to prove Theorem \ref{K-luna}.
\begin{proof}[Proof of Theorem  \ref{K-luna}]
Let $U$ be the open set constructed above satisfying \eqref{U-min} and
let
$$U^{\rm an}_W=(U^{\rm ks}\cap \PP W)\times_{{\PP^M}}  Z^\ast. \text{(cf. Lemma \ref{K-near})} $$
After a possible shrinking,  we may assume that all the points in $U^{\rm an}_W$
 are GIT-semistable and  every GIT-semistable point can be degenerated to a GIT-polystable point in $U^{\rm an}_W$.

Suppose $\hilb(Y)\in U^{\rm an}_W$ is GIT-polystable and strictly K-semistable. Then by Lemma \ref{K-near}, we can degenerate it to  a variety
$Y'\subset \PP^{N}$ which is K-polystable  such that
$$\hilb(Y')\in U\cap \overline{\SL(N+1)\cdot \hilb(Y)}\subset Z^\circ\subset \HH^{\chi;N}\ ,$$
and $\hilb(Y')$ is {\em close} to $\hilb(Y)$ in $\HH^{\chi;N}$ in the sense that there is {\em short} (with respect to the metric $\dist_{\HH^{\chi;N}}$) path inside $\overline{\SL(N+1)\cdot \hilb(Y)}$ joinning $\hilb(Y)$ and $\hilb(Y')$.

Using the transversality of the action of $\faut(X)^\perp\subset \fsl(N+1)$ on $\PP W\subset \PP^M$,  one can always find a $g\in \SL(N+1)$ {\em close} to the identity such that
 $$\hilb(Y''):=g\cdot \hilb(Y')\in   \PP W\times_{{\PP^M}} Z^\ast,$$ where $Y''\cong Y'$ is GIT-semistable. This allows us to find
a {\em short}  path inside $\overline{\SL(N+1)\cdot \hilb(Y)}$ joining $\hilb(Y)$ and $\hilb(Y'')$, which by transversality we may assume to be entirely contained in $\PP W$
and satisfies $\hilb(Y'')\in \overline{\aut(X)\cdot \hilb(Y)}$.
But this  is absurd since $\hilb(Y)$ is already GIT-polystable, no point on the boundary of $\overline{\aut(X)\cdot \hilb(Y)}$ is semistable.

 Conversely, suppose $\hilb(Y)\in U^{\rm an}_W$ and $Y$ is K-polystable but $\hilb(Y)$ is not GIT-polystable,  then there is a 1-PS $\lam \subset \aut(X)$ degenerating $\hilb(Y)$ to a nearby GIT-polystable  $$\hilb(Y')\in \overline{\aut(X)\cdot\hilb(Y)} \cap U^{\rm an}_W$$
 by the classical GIT.
 Thus $Y'$  is K-polystable by the previous paragraph,  contradicting to the assumption  $Y$ being K-polystable.  Hence our proof is completed.

To pass from the analytic neighborhood to a Zariski neighborhood, we need to investigate the geometry of $\aut(X)$-orbits.
Let $ U^{\rm ss}_W\subset \PP W$ containing $\hilb(X)$ be the Zariski open set of {\em GIT-semistable points}.  By \cite[Chapter 2, Proposition 2.14]{MFK} and \cite[Lemma 2.11 and Lemma 2.12]{Odaka14}, we know that the set of GIT-polystable points in $U^{\rm ss}_W$ forms a {\em constructible} set. On the other hand,  K-polystable points inside $U^{\rm ss}_W\cap  Z^\circ_{\rm red}$ also form a constructible sets (see Remark \ref{constr-Kpoly}) containing the  point $\hilb(X)$.  These two constructible sets coincide along  $U^{\rm an}_W$ after lifting to $\PP W\times_{{\PP^M}} Z^\ast$ by the proof above,  so they must coincide on a Zariski open set.

Finally, we establish the last statement.
 The slice $\Sigma$ obtained from Summary \ref{sum-slice} satisfies Assumption  \ref{bdd}. Thus, by applying Theorem \ref{S-pres} to our setting we can construct an analytic open set $U_W\subset \PP W\times_{{\PP^M}} Z^\ast $ that is stabilizer preserving. To obtain the Zariski openness, one  observes  that
\begin{eqnarray*}
\aut( Z^\ast):=\{(z, g)\in Z^\ast\times \SL(N+1)\mid g\cdot z=z \}=\phi_{\SL(N+1)}^{-1}(\De_{Z^\ast})
\end{eqnarray*}
is a  {\em closed}  subset of $Z^\ast\times \SL(N+1)$, where
\begin{equation*}
\begin{array}{ccc}
Z^\ast \times \SL(N+1) & \overset{\phi_{\SL(N+1)}%
}{ \xrightarrow{\hspace*{1.1cm}}} &Z^\ast \times Z^\ast   \\
(z,g) & \longmapsto & (z, g\cdot z)
\end{array}%
\text{ and   }
\De_{Z^\ast}=\{(z,z)|z\in Z^\ast\}\subset Z^\ast\times Z^\ast.
\end{equation*}
{Next let
\begin{equation*}
\mu:
\begin{array}{ccc}
\aut(Z^\ast)  & \overset{}{ \xrightarrow{\hspace*{1.1cm}}} &Z^\ast   \\
(z,g) & \longmapsto &  g\cdot \hilb(X)
\end{array}.
\end{equation*}
Then the locus of
$$\{\hilb(Y)\in Z^* \ | \ \Aut(Y)<\Aut(X)\}$$ is precisely the complement of ${\rm pr}_1\large(\mu^{-1}(Z^*\setminus \{\hilb(X) \} )\large)$ which is constructible, where ${\rm pr}_1:\aut(Z^\ast)\to Z^\ast$ is the projection to the first factor. So we can prolong $U_W$ from an analytic open subset  to a Zariski open one and our proof is completed.}
\end{proof}

\begin{rem}
One notice that contrast to Theorem \ref{K-luna}, there exists smooth Fano varieties admitting K\"ahler-Einstein metrics, which are not asymptotically Chow stable (see
\cite{OSY}).  On the other hand, Theorem \ref{K-luna} can be regarded as an extension of work  \cite{Sz2010} to  the case of $\QQ$-Gorenstein smoothable $\QQ$-Fano varieties.
\end{rem}

Next we show that for each closed point $[z]\in [Z^\ast/ \SL(N+1)]$,  $\overline{\{z\}}$ has a {\em good moduli} space  in the sense of {\cite[Theorem 1.2 and Proposition 3.1]{AFSV14}}. To to that, let us first establish the Assumption \ref{sp} in Section \ref{s-apendix}.
Let $z=\hilb (Y)\in U_W$  specializing  to $ z_0=\hilb(X)\in U_W\subset \HH^{\chi;N}$ via a 1-PS $\lam(t):\GG_m\to \aut(X)< \SL(N+1)$. Let $(\cY=\cX|_C, X)\to (C=\overline{\lam(t)\cdot z},z_0)\subset U_W$ be the restriction of the universal family  $\cX\to Z^\ast $ to the pointed curve $(C,z_0)$
and also we prefix a basis $\{s_i\}\subset \sO_{\cY}(-rK_{\cY/C})$.

\begin{lem} \label{aut-Y}
Under the notation introduced above, we have $\aut(Y)<\aut(X)$ for $z:=\hilb(Y)$ close to $z_0=\hilb(X)$ with respect to the analytic topology.
\end{lem}
\begin{proof}
By property \eqref{tv} in the proof of Lemma \ref{disj}, for $z=\hilb(Y)\in U_0$ we have $\faut(Y)\subset \faut(X)$, hence the identity component of $\aut(Y)$ lies in $\aut(X)$.  We will assume from now on that $z=\hilb(Y)\in \PP W$ lies in a {\em small} analytic neighborhood of  $z_0=\hilb(X)\in U_1$, i.e. $z$ is {\em very close} to $z_0$.  This together with the fact that
there always exists   a {\em finite} subgroup $H<\aut(Y)$ that  meets {\em every} connected component of $\aut(Y)<\SL(N+1)$ imply that
all we need is that:  for any finite subgroup $H<\aut(Y)$,
we have $H<\aut(X)$.
 To achieve that,  let us choose $H$-invariant {\em smoothable} divisor $E\in |-mK_Y|^H$ so that $(Y, \frac{E}{m})$ is klt, the existence of such $E\subset Y$ is guaranteed by the following result.

\begin{claim}Let $Y$ be a {\em $\QQ$-Gorenstein smoothable} $\mathbb{Q}$-Fano variety. Fix a finite group $H\subset {\rm Aut}(Y)$. For $m$ sufficiently divisible there is an invariant section $E\in |-mK_Y|^H$ such that $(Y,(1-\ep)E)$ is klt for any $0<\ep\leq 1$ and {\em $\QQ$-Gorenstein smoothable}. In particular, $(Y,\frac{1}{m}E)$ is {\em $\QQ$-Gorenstein smoothable} and klt for $m>1$. Moreover, $m$ can be uniformly bounded provided $Y$ is inside a bounded family.
\end{claim}
\begin{proof}
Let $\mu\colon Y\to \ti Y$ be the quotient of $Y$ by $H$, and $D$ be the branched divisor. So $\mu^*(K_{\ti Y}+D)=K_Y$. In particular, $(\ti Y,D)$ is klt (since  klt is preserved under finite quotient \cite[Theorem 5.20]{KM98}) and $-(K_{\ti Y}+D)$ is ample. Thus for a sufficiently divisible $m$ satisfying $-m(K_{\ti Y}+D)$ being very ample, we can choose a general section $F\in |-m(K_{\ti Y}+D)|$ so that $(\ti Y,D+(1-\epsilon)F)$ is klt for any $0<\epsilon \le 1$.
Then $E:=\mu^*(F)$ is $H$-invariant and $(Y,(1-\epsilon)E)$ is klt for any $0<\ep\le 1$. Finally, we justify that  $(Y,E)$ is actually {\em $\QQ$-Gorenstein smoothable} as long as $Y$ is. Since $Y$ is a degeneration a smooth family $\{Y_t\}_t$, and every element in $|-mK_{Y}|$ can be represented as a degeneration of  general members of $|-mK_{Y_t}|$, from which we  conclude $(Y,E)$ is a degeneration of  smooth pairs $\{(Y_t,E_t)\}_t$.
\end{proof}
Then by Theorem \ref{log-main} and \ref{unique},  $(Y,\frac{E}{m})$ admits a continuous family of K\"ahler metric $\{\omega_Y(\be)\}$
solving
$$
\Ric(\omega_Y(\be))=\be \omega_Y(\be)+\frac{1-\be}{m}[E] \text{ on } Y\ ,
$$
from which we obtain
\begin{equation}\label{Y-be}
\hilb(Y,\omega_Y(\be))\stackrel{\be\to 1}{\longrightarrow} \UU(N+1)\cdot \hilb(X)\subset \HH^{\chi;N}
\end{equation}
thanks to Theorem \ref{ss} and  the fact that $m$ is uniformly bounded by Lemma  \ref{K-bdd}, where $\hilb(Y,\omega_Y(\be))$ is the Hilbert point corresponding to the Tian's embedding of $Y\subset \PP^N$ with respect to the metric $\omega_Y(\be)$ on $Y\subset \PP^N$ and any prefixed basis $\{s_i\}\subset H^0(\sO_Y(-rK_Y))$.
This allows us to introduce a {\em continuous} family of Hermitian metric $h_\KE(\be(t))$ with $\be(t):=1-|t|$ on $\sO_{\cY_t}(-K_{\cY_t})\to \cY_t$ for $0<|t|:=\dist_C(t,0)<1$, such that $\omega_Y(\be(t))=-\ddbar \log h_\KE(\beta)$. By \eqref{Y-be}, the metric $h_\KE(\be(t))$ can be continuously extended to $0\in C$. Now  let  $\{s_i\}$ be the local basis of  $\pi_\ast\sO_{\cY}(-rK_{\cY/C})|_{\{|t|<1\}\subset C}=\pi_\ast( \sO_{\PP^N}(1)|_\cY)$ corresponding to the coordinate sections of $\sO_{\PP^N}(1)$ such that $\{s_i(0)\}$ induces Tian's embedding for $z_0=\hilb(X)$  and define
$$A_\KE(t,\be(t))=[( s_i,s_j)_{\KE,\be(t)}(t)] $$
with
$$
(s_i,s_j )_{\KE,\be}(t)=\int_{\cY_t} \langle s_i(t),s_j (t)\rangle_{{h_\KE^{\otimes r}}(\be(t))}\omega_Y^n(\be(t))\ ,
$$
then we obtain a family of  Tian's embedding
\begin{equation}\label{Ts}
T:(\cY_t,\cE_t;\omega_Y(\be(t)))\longrightarrow \PP^N \text{ with } (\cY_t,\cE_t)\cong(Y,E) \text{ for }t\ne 0.
\end{equation}
given by $\{g(t)\circ s_j(t)\}_{j=0}^N$ with $g(t)=A^{-1/2}_\KE(\be(t))$. The map $T$ extends to $\cY_0=X$ thanks to the continuity of the metric $h_\KE(\be(t))$ at $0\in C$.

Now by our choice of $z_0$ and basis $\{s_i(t)\}$, we have $A_\KE(0,1)=I_{N+1}\in \SL(N+1)$, and hence
\begin{equation}\label{g=1}
g(t)\sim I_{N+1}+O(t).
\end{equation}
This implies that
  \begin{equation}\label{U1-ep}
  \ti z(t):=\hilb(\cY_t,\omega_Y(\be(t)))=g(t)\cdot z_t\in U_{1,\ep}:=\exp (\faut(X)_{<\ep}^\perp)\cdot U_1
  \end{equation}
for $0<|t|\ll 1$, where $z_t=\lam(t)\cdot \hilb(Y)$.
Since $\omega_Y(\be(t))$ is a conical K\"ahler-Einstein metric on $\cY_t$, it follows from the log version of Lemma \ref{G-U} (cf. \cite[Theorem 4]{CDS3}) that
$$ H_{\ti z(t)}=g(t)\cdot H_{z(t)}\cdot g(t)^{-1}< U(N+1) \text{ where }H_{z(t)}= \lam(t) \cdot H\cdot \lam(t)^{-1} .
$$
By  Lemma \ref{Gz0} and \eqref{g=1}, we obtain that
$H_{z(t)}<\aut(X)$ and hence $H<\aut(X)$ as $\lam(t)<\aut(X)$ by our choice.
On the other hand,  by transversality of $\faut(X)^\perp$-action on $U_{1,\ep}$, for $0<t\ll 1$ we have
$\aut_0(Y)<\aut(X)$ where $\aut_0(Y)$ is the identity component of $\aut(Y)$. This implies that
$\aut(Y)=\la\aut_0(Y),H\ra<\aut(X)$ where $\la\aut(Y),H\ra$ is the subgroup generated by $H$ and  $\aut_0(Y)$ and
 our proof is completed.
\end{proof}
As a consequence, 
we have the following the statement, which implies Assumption \ref{sp}.
\begin{cor}\label{semi-sp}
After a possible shrinking of the Zariski open neighborhood  $z_0\in U_W\subset \PP W\times_{\PP^M} Z^\ast$, we have
$$\SL(N+1)_z<\aut(X), \ \forall z\in U_W$$
 where $\SL(N+1)_z$ is the stabilizer of $z$ inside $\SL(N+1)$. In other words, $U_W$ satisfies {\em Assumption \ref{sp}}.
\end{cor}
Next in order to apply Lemma \ref{fini} in Section \ref{s-apendix}, we now establish Assumption \ref{nr}.  Let us fix $G=\SL(N+1)$ and $G_{z_0}=\aut(X)$. Recall from  Assumption \ref{nr}: An analytic {\em open neighborhood} of $z_0\in U^\nr\subset \PP W$  is of {\em finite distance} if there is a {\em compact}
subset  $G_{U^\nr}\Subset G/G_{z_0}$ depending only on $U^\nr$ and $z_0$ such that for any pair $(z, g)\in U^\nr\times G$ satisfying  $g\cdot z\in U^\nr $, there is an $h\in G,\  [h]\in G_{U^\nr}\Subset G/G_{z_0}$ such that $g\cdot z=h\cdot z$.
\begin{lem}\label{NR}
Let $z_0\in U_r\subset \PP W$ be defined in Definition \ref{Ur} and
$$U_{Z^\ast,r}:=U_r\times_{{\PP^M}} Z^\ast .$$
Then for $0<r$ sufficient small, {$U_{Z^\ast,r}$ is a $G_{z_0}$-invariant subset  of {\em finite distance} in the sense of {\em Assumption \ref{nr}}.}
\end{lem}
\begin{proof}
In order to better illustrate the idea, let us first deal with the case that $z_0$ is K-stable, hence $G_{z_0}<\infty$.  As we have seen in the proof of Theorem \ref{K-luna}, there is a {\em proper} $\UU(N+1)$-invariant slice $z_0\in\Sigma\subset \HH^{\chi;N}$ obtained via Tian's embedding. By the continuity of $\Sigma$ and transversality of the $\fg_{z_0}^\perp$-action on $U_0$(cf. the proof Lemma \ref{disj}),  for  some $0<r''<r'\ll 1$  and $0<\ep\ll 1$ we have
\begin{equation}\label{BZ-r}
 B_{Z^\ast}(z_0,r'')\subset U_{r'}\cap \exp \fg_{z_0,<\ep}^\perp\cdot \Sigma,
\end{equation}
where $\fg_{z_0,<\ep}^\perp:=\{\xi\in \fg^{\perp}_{z_0}\mid |\xi|<\ep\}$ and $B_{Z^\ast}(z_0,r'')$ denotes the ball of radius $\ep$ centered at $z_0\in Z^\ast$  with respect to a prefixed continuous metric on $Z^\ast$. Moreover, by choosing a small $r$ if necessary, we may assume $\cX_z$ is K-stable for all $z\in B_{Z^\ast}(z_0,r'')$.

To see the lemma, let $\{s_i\}$ be the local basis of $\pi_\ast (\sO_{\PP^N}(1)|_\cX)$ corresponding to the coordinate sections of $\PP^N$  such that the induced embedding of $X=\cX_{z_0}\subset\PP^N$ gives rise to $\hilb(X)$.  Now let us equip the line bundle $\sO_{\cX}(-rK_{\cX/Z^{\ast,\kps}})\cong \sO_{\PP^N}(1)|_\cX$ with a Hermitian metric which gives rise to the {\em unique} K\"ahler-Einstein metric when restricted to each $\cX_z$ with $z\in B_{Z^\ast}(z_0,r'')$, and we can introduce the matrix $A_\KE(z)$ as in the proof of Lemma \ref{aut-Y}. Then \eqref{BZ-r} follows from the continuity of $A_\KE(z)$ with respect to $z\in Z^{\ast}$ and  $A_\KE(z_0)=I_{N+1}$ (as $X\subset\PP^N$ is a  Tian's embedding).

As a consequence, for any pair
$(z,g)\in B_{Z^\ast}(z_0,r'')\times G$ satisfying  $g\cdot z\in B_{Z^\ast}(z_0,r'')$,  there are $h',h''\in G$ such that under the quotient map
 $$[\cdot]:G\to G/G_{z_0},$$
  $[h'], [h'']\in G/G_{z_0}$ are perturbations of $[1]\in G/G_{z_0}$ and $h'\cdot z, h''\cdot g\cdot z\in \Sigma$.  Since both $h\cdot z$ and $h'\cdot g\cdot z$ are the Hilbert points of Tian's embedding of the {\em same} $\QQ$-Fano variety, we know that $u:=h'^{-1}\cdot h''\cdot g\in \UU(N+1)$. This implies that $g\cdot z=h\cdot z$ with $h=h''^{-1}\cdot h'\cdot u$ and $[h]$ being {\em uniformly bounded} (with the bound depending only on $ B_{Z^\ast}(z_0,r'')$ and $z_0$)   in $G/G_{z_0}$. Since  the property whether or not $z$ lies in  ${U_{Z^\ast,r}}$  is independent of the $G_{z_0}$-translation,  we conclude that  Assumption \ref{nr} holds for all  points  in $ U_{Z^\ast, r}\subset G_{z_0}\cdot B_{Z^\ast}(z_0,r'')$ for some $0<r<r''$.

 For the general case, let us introduce   a  general divisor $\cD\in |-mK_{\cX}|$ for sufficiently divisible $m$ such  that
 \begin{enumerate}
 \item $(\cX,\cD)|_{U_W}$  (, where $U_W$ is given in the proof of Theorem \ref{K-luna}) are family of $\QQ$-Fano variety;
 \item  $\cD_z$ is smooth whenever $\cX_z$ is for $z\in U_W$.
 \end{enumerate}
   Then by Theorem \ref{log-main} we can construct a
    {\em proper}
   $\UU(N+1)$-invariant slice
   $\Sigma_{\frac{1-\be}{m}\cD}\subset \HH^{\chi;N}$ using Tian's embedding of $\cX_z\subset\PP^N$ with respect to the the {\em unique} conical K\"ahler-Einstein metric
 $$
\Ric(\omega_{\cX_z}(\be))=\be \omega_{\cX_z}(\be)+\frac{1-\be}{m}[\cD_z] \text{ on } \cX_z\
$$
for all $z\in U_W$ near $z_0$.
In particular,  Theorem \ref{log-main} and \ref{ss} imply that $\Sigma_{\frac{1-\be}{m}\cD}\to \Sigma$ in the sense that $\forall \ep>0$, $\Sigma_{\frac{1-\be}{m}\cD}$ falls into a $\ep$-tubular neighborhood of $\Sigma$ as $\be\to 1$. This implies that for $0<r'\ll1$ and $z,z'\in B_{Z^\ast}(z_0,r')$ that are contained in
$$
(G\cdot \hilb(\cX_z))\ \bigcap\  (\UU(N+1)\cdot\exp \sqrt{-1}\fg_{z_0,<\ep}^\perp)\cdot B_{Z^\ast}(z_0,r')\  \text{(cf. \eqref{bdd-G-G_0})}
$$
with $\fg_{z_0,<\ep}^\perp:=\{\xi\in \fg_{z_0}\mid |\xi|<\ep\}$,  the $\UU(N+1)$-orbits for Tian's embedding of $(\cX_z,\cD_z)$ and $(\cX_{z'},\cD_{z'})$ are very close in the sense that they can be translated to each other by an element $h\in \UU(N+1)\cdot\exp\sqrt{-1} \fg_{z_0,<\ep}^\perp\cdot G_{z_0}\subset G$ (i.e. $[h]\in G/G_{z_0}$ is  bounded in the sense of  \eqref{bdd-G-G_0}).
In particular, this allows us to  treat these two $\UU(N+1)$-orbits as {\em almost identical} one and we argue exactly the same way as  the K-stable case. This completes the justification of Assumption \ref{nr} for  a neighborhood of $z_0\in U_{Z^\ast, r}$ for some sufficient small $r>0$.
\end{proof}

{Finally, with all the preparation above we are ready to finish our main construction of this section.}
\begin{proof}[Proof of Theorem \ref{t-good}]
By Theorem \ref{t-AFScriterion}, proving our statement boils down to establishing the following:  for any closed point $[z_0]\in [Z^\ast/\SL(N+1)]$ there is an {\em affine} neighbourhood $z_0:=\hilb(X)\in U_W\subset \PP W$ determined in $\eqref{PWW}$ such that
\begin{enumerate}
\item The morphism $[U_W/G_{z_0}]\to [Z^\ast/G]$ is  stabilizer preserving and sending closed point to closed point, and
\item  For any $\mathbb{C}$-point $z\in Z^\ast$ specializing to $z_0$ under $G$-action, the closure of substack $[z]$ inside $[Z^\ast/G]$, $\overline{\{[z]\}}\subset [Z^\ast/G]$ admits a good moduli space,
\end{enumerate}
with $G=\SL(N+1)$ and $G_{z_0}=\aut(X)$ as before.

\Blue{We have shown the  \'etale  morphism $[U_W/G_{z_0}]\to [Z^\ast/G]$ is stabilizer preserving and sending closed points to closed points by Theorem \ref{K-luna}. Next we prove the morphism is affine to conclude it is a local presentation satisfying Condition (1) in Theorem \ref{t-AFScriterion}.}
Since $Z^*\to [Z^*/G]$ is faithfully flat, it suffices to show that
$$
\phi\colon G\times_{G_{z_0}} U_W \longrightarrow Z^* 
$$
is affine.
Since $\phi$ is quasi-finite and $Z^*$ is separated, it suffices to choose $U_W$ such that $G\times_{G_{z_0}} U_W$ is affine.
Let $U_W\subset Z^\ast\cap \mathbb{P}(W)$ be a $G_{z_0}$-invariant affine open set then we know $G\times_{G_{z_0}} U_W$ is affine since
it is a quotient of the affine scheme $G\times U_W$ by the free action of the reductive group $G_{z_0}$. Furthermore,  we have an isomorphism
$$
(G\times_{G_{z_0}}U_W)/\!\!/G\cong U_W/\!\!/G_{z_0}
$$
and $G\times_{G_{z_0}}U_W$ is the inverse image of the affine neigborhood
$$\pi_W|_{U_W}(z_0)=0\in U_W/\!\!/G_{z_0} \text{ with }\pi_W \text{ defined in } \eqref{pi-w} $$
 under the GIT quotient by $G$.

Now we  establish the second condition. Since we have already established
the  uniqueness of minimal orbit contained in $\overline{BO}_{z_0}$  stated after  diagram \eqref{Z*-chow},
all we need is the affineness of $G\cdot \pi^{-1}_W(0)$ as it implies that for any $z\in Z^\ast$ satisfying $\overline{G\cdot z}\ni z_0$ the closure of $[z]\in [Z^*/G]$   is a closed substack of $[G\cdot \pi^{-1}_W(0)/G]$, which can be written as the form $[{\rm Spec}(A)/G]$ for some affine scheme ${\rm Spec}(A)$, hence $\overline{[z]}$ admits a good moduli space.

To obtain the affineness,   one notices that Theorem \ref{K-luna} and Corollary \ref{semi-sp} guarantee the Assumption \ref{sp}, also we have already established Assumption \ref{nr} by Lemma \ref{NR}.  Thus the morphism
\begin{equation}\label{phi-r}
\phi|_{G\times_{G_{z_0}} U_{Z^\ast,r}}:G\times_{G_{z_0}} U_{Z^\ast,r} \to G\cdot U_{Z^\ast,r} 
\end{equation}
is a  {\em finite} morphism  for $0<r\ll 1$ by Lemma \ref{fini} in Section \ref{s-apendix}. By choosing  $0<r$ even smaller, we may conclude that $\phi|_{G\times_{G_{z_0}} U_r}$ is an {\em analytic isomorphism}, since
$\phi|_{G\cdot z_0}$
is an isomorphism and immersion near $G\cdot z_0$.
  Now we restrict $\phi$ to the fiber  over  $[z_0]\in [Z^\ast /G]$, we have a finite morphism
$$G\times_{G_{z_0}} \pi^{-1}_W(0)\longrightarrow  G\cdot \pi^{-1}_W(0)\ .$$
Since  $G\times_{G_{z_0}} \pi^{-1}_W(0)$ is a fiber of a GIT quotient morphism,  we conclude that $G\cdot \pi^{-1}_W(0)$ is affine.


As a consequence, the \'etale chart $\phi/G:(G\times_{G_{z_0}} U_W)/\!\!/G\to G\cdot U_W/G$ is actually a {\em finite} morphism, which implies  $G\cdot U_W/G$  is affine. This gives an  {\em affine}  neighborhood of $[z_0]\in \mathcal{KF}_N$. This  proves that the algebraic space $\mathcal{KF}_N$ is actually a {\em scheme}.
Finally to prove the last statement of Theorem \ref{t-good}, we observe that Lemma \ref{K-bdd} implies that the closed points of $\mathcal{KF}_N$ stabilizes. However, since $\mathcal{KF}_N$ is semi-normal, we indeed know that they are isomorphic (see \cite[7.2]{Kol96}).
\end{proof}

\begin{rem}
We want to point out that by shrinking $U_W$ if necessary  the map $\phi:G\times_{G_{z_0}} U_W \to G\cdot U_W$ is actually {\em strongly \'etale} in the sense of \cite[page 198]{MFK}, i.e. $U_W$ is a Luna's \'etale slice.  To see that one notices  that  we have already established in the above that the categorical quotient $(G\cdot U_W)/G$ is in fact a {\em good quotient} (see also \cite[Definition 2.12]{Dre2004}) and moreover the map $\phi$ induces an \'etale morphism
$$\phi/G: (G\times_{G_{z_0}} U_W)/G \to (G\cdot U_W)/G.$$
  So all we need to show is
\begin{equation*}
(\phi,\pi_{G\times_{G_{z_0}} U_W}):
\begin{array}{cccc}
G\times_{G_{z_0}} U_W & \overset{\phi%
}{ \xrightarrow{\hspace*{1.3cm}}} & G\cdot U_W\times_{(G\cdot U_W)/G}(G\times_{G_{z_0}} U_W)/G & \\
( g,w) & \longmapsto & (g\cdot w, [g\cdot w]) &
\end{array}%
\end{equation*}%
is an isomorphism, where $\pi_{G\times_{G_{z_0}} U_W}:G\times_{G_{z_0}} U_W\to (G\times_{G_{z_0}} U_W)/G\cong U_W/\!\!/G_{z_0}$ is the GIT quotient map. But this follows from the fact that   $\phi|_{G\times_{G_{z_0}} U_r}$ in \eqref{phi-r} is an analytic {\em isomorphism} for small $r$ and $\phi$ is {\em finite}.
\end{rem}

\begin{rem}\label{HiCh}

Notice that we can take the local GIT quotient of a similarly defined $Z^\circ_{\red}$ for each $N_r=\chi(X,\sO_X(-rK_X))-1$. Although we are {\em unable} conclude that those local GIT quotients we constructed in this section will be stabilized for $N \gg 1$, their semi-normalizations indeed will be.
Another reason we work over a seminormal base is that the condition of being smoothable does not yield a reasonable moduli functor for schemes, e.g., in general there is no good definition of smoothable varieties over an Artinian ring.

We also remark that there is no difference to work on Hilbert scheme or Chow variety in our case, at least after the seminormalization. This is because by our definition  of $Z^\circ$ (see Definition \ref{Z0}), the closed points correspond to $\QQ$-Fano varieties (see Remark \ref{Y-E}) which in particular are geometrically reduced,  hence the Hilbert-to-Chow morphism is a bijection (see \cite[Chapter I, Theorem 6.3]{Kol 96}) when restricted to  $Z^{\circ}$, thus they share the same semi-normalization (\cite[Section 3.15]{Kol96}).

\end{rem}

\section{Appendix}\label{s-apendix}
\subsection{Constructibility of $\kst$}\label{s-constru}

In this section, we  will prove Proposition \ref{p-con} in a more general setting. First, let us  recall some basics from \cite[section 2 of Chapter 2]{MFK}.   Let $G$ be a reductive group acting on a (quasi-)projective variety $(Z,L)$ polarised by a $G$-linearized  very ample line bundle $L$.

\begin{defn} The {\em rational flag complex $\De(G)$} is the set of non-trivial 1-PS's $\lam$ of $G$ modulo the {\em equivalence} relation:
$\lam_1\sim\lam_2$ if there are positive integers $n_1$ and $n_2$ and a point $\ga\in P(\lam_1)$ such that
$$\lam_2(t^{n_2})=\ga^{-1}\lam_1(t^{n_1})\ga\text{ for all } t\in \GG_m$$
where
$$P(\lam):=\left\{\ga\in G\left | \lim_{t\to 0}\lam(t)\ga\lam(t^{-1})\text{ exists }\right.\right\}\subset G$$
is the unique {\em parabolic subgroup} associated to $\lam$.
The point of $\De(G)$ defined by $\lam$ will be denoted by $\De(\lam)$. In particular, for a maximal torus $T\subset G$, $\De(T)=\Hom_\QQ(\GG_m,T)$.
\end{defn}

Then we have the following
\begin{lem}[Chapter 2, Proposition 2.7, \cite{MFK}]\label{mu-ga}
For any 1-PS $\lam:\GG_m\to G$, let $\mu^L(z,\lam)$ denote the {\em $\lam$-weight} of $z\in Z$ with respect to the $G$-linearization of $L$. Then for any $(\ga,z)\in G\times Z$, we have
$$\mu^{L}(z,\lam)=\mu^{L}(\ga z,\ga\lam\ga^{-1})\ .$$
Moreover, if  $\ga\in P(\lam)$ then $\mu^{L}(z,\lam)=\mu^{L}(z,\ga\lam\ga^{-1})\ .$
\end{lem}
The next Lemma is a slight extension of \cite[Chapter 2, Proposition 2.14]{MFK} essentially contained in \cite[proof of Lemma 2.11]{Odaka14}, hence the proof will be omitted.
\begin{lem}\label{T-L-M}
 Let $T\subset G$ be a  maximal torus and $L_i, i=1,2$ be two $G$-linearized {\em ample} line bundles over $Z$.  Then there is a finite set of {\em linear functional} $l^{L_i}_1,\cdots,l^{L_i}_{r_{L_i}}, i=1,2$ which are rational on  $\Hom_\QQ(\GG_m,T)$ with the following property:
 \begin{equation}\label{I}
 \forall z\in Z, \ \exists\  I(z,L_i)\subset \{1,\cdots, r_{L_i}\}, \ I(z,L_2)\subset \{1,\cdots,r_{L_2}\}
 \end{equation}
 such that the {\em $\lam$-weight} of $z\in Z$ with respect to the linearization of $G$ on $L_1\otimes L_2^{-1}$ is given by
\begin{eqnarray*}
\mu^{L_1}(z,\lam)-\mu^{L_2}(z,\lam)
&=& \max\{l^{L_1}_i(\lam)\mid i\in I(z,L_1)\}-\max\{l^{L_2}_i(\lam)\mid i\in I(z,L_2)\}
\end{eqnarray*}
 for all 1-PS  $ \lam\subset T $.
Moreover,  the function
$$
\begin{array}{cccc}
 \psi^{L_1,L_2}:& Z & \longrightarrow &2^{\{1,\cdots,r_{L_1}\}}\sqcup 2^{\{1,\cdots, r_{L_2}\}}\\
& z&\longmapsto & I(z; L_1,L_2):=I(z,L_1)\sqcup I(z,L_2)
\end{array}
$$
are {\em constructible} in the sense that $\forall I\in 2^{\{1,\cdots,r_{L_1}\}}\sqcup 2^{\{1,\cdots, r_{L_2}\}}$, the set  $\psi^{-1}(I)\subset Z$ is constructible.
\end{lem}



For any line bundle that can be written as $L_1\otimes L_2^{-1}$ with  $L_1$ and $L_2$ both being $G$-linearized and very ample,  we can similarly show that $Z$ can be decomposed into a union of finitely many constructible sets indexed by $2^{\{1,\cdots,r_{L_1}\}}\sqcup 2^{\{1,\cdots,r_{L_2}\}}$, such that restricted on each piece,
$$\mu^{L_1\otimes L_2^{-1}}(z,\lam)=\mu^{L_1}(z,\lam)-\mu^{L_2}(z,\lam)$$
is a rational function on ${\rm Hom}_{\QQ}(\GG_m,T)$.
\begin{prop}\label{be-constr} Let $G$ act on an polarized variety $(Z,L)$.
 Let $M_i,i=1,2$ be two  $G$-linearized line bundles on $Z$ (not necessarily being ample). For $z\in Z$ and $\de\in \De(G)$, we define
$$\nu_{1-\be}^{M_1,M_2}(z,\de):=\frac{\mu^{M_1}(z,\lam)-(1-\be)\mu^{M_2}(z,\lam)}{|\lam|}\text{  with } \De(\lam)=\de$$
 and define
$$\vpi_G^{M_1, M_2}(z):=\sup \left \{\beta \in (0,1]\left | \displaystyle \inf_{\de\in \De(G)} \nu_{1-\be'}^{L,M}(z,\de)\geq0,\  \forall \be'\in [0,\be)\right.\right\}$$
or 0 if the set on the right hand side is an empty set.
Suppose $S\subset Z$ is a {\em constructible } set such that $\left.\vpi_G^{M_1,M_2}\right |_S>0$.
Then $\vpi(M_1,M_2)$ defines a $\QQ$-valued constructible function on $S$, i.e. $S=\sqcup_i S_i$ is a union of finite constructible sets  with  $\vpi(M_1,M_2)$ being  constant on each $S_i$.
\end{prop}

\begin{proof}
We replace $L\to Z$ by its power such that  $L_1:=L\otimes M_1$ and $L_2:=L\otimes M_2$ are both ample. Then we fix a maximal torus $T\subset G$ and  let $\{l_i^{L_1}\}$  and $\{l_i^{L_2}\}$ be the {\em rational} linear functionals on $\Hom_\QQ(\GG_m,T)$ associated to $L_i,i=1,2$.  By Lemma \ref{T-L-M}, for any $I\in 2^{\{1,\cdots,r_{L_1}\}}\sqcup 2^{\{1,\cdots,r_{L_2}\}} $,  $S^T_I:=\psi^{-1}(I)\cap S$ is a constructible set. Now we define
$$\vpi_T^{M_1, M_2}(z):=\sup \left \{\beta \in (0,1]\left | \displaystyle \inf_{\de\in \De(T)} \nu_{1-\be'}^{L,M}(z,\de)\geq0,\  \forall \be'\in [0,\be)\right.\right\}$$
or 0 if the right hand side is an empty set.
In other words, it is the {\em first time} such that  the difference of two {\em rational} piecewise linear {\em convex} functions
$$\mu^{L_1}(z,\cdot)-(1-\beta)\mu^{L_2}(z,\cdot)-\beta\mu^{L}(z,\cdot)=\mu^{M_1}(z,\cdot)-(1-\be)\mu^{M_2}(z,\cdot)$$ vanishes along a ray in $\Hom_\QQ(\GG_m,T)$ or in \{0,1\}.  Clearly, we have $\be_I\in \QQ$ and they are independent of the choice of $L$.

Now in order to pass from $\vpi_T^{M_1,M_2}$ to $\vpi_G^{M_1,M_2}$,  let us recall Chevalley's Lemma \cite[Chapter II, Exercise 3.19]{Har77} which states that the image of constructible set under an algebro-geometric morphism is again constructible.  By applying it to the group action morphism
$$
G\times Z\longrightarrow Z \ ,
$$
we obtain that $S^G_I:=G\cdot  (\psi^{-1}(I)\cap S)\supset S^T_I$ are all constructible $\forall I\in2^{\{1,\cdots,r_{L_1}\}}\sqcup 2^{\{1,\cdots,r_{L_2} \}}$ . Now for any 1-PS $\lam$, there is a $\ga\in G$ such that $\ga \lam \ga^{-1}\subset T$. By Lemma \ref{mu-ga}, we have $\mu^{L_i}(z,\lam)=\mu^{L_i}(\ga z,\ga\lam\ga^{-1})\ , i=1,2$, which implies that
\begin{eqnarray*}
\vpi_G^{M_1,M_2}(z)
&=&\min\left\{\be_J\left | S^T_J\cap G\cdot z\ne\varnothing \text{ for }J\in 2^{\{1,\cdots,r_{L_1}\}}\sqcup 2^{\{1,\cdots,r_{L_2} \}}\right.\right\}\\
\end{eqnarray*}
To see it is  a constructible function on the constructible set $G\cdot S$,  one  notices that  all possible finite intersections of $\{S^G_J\}_J$ form a stratification of $G\cdot S$ into constructible sets and
$\vpi_G^{M_1,M_2}$ is constant on each stratum.

\end{proof}

Now to apply the above set up to the $\be$-K-stability of  $(X,D)\subset \PP^N$ with respect to the ${\rm SL}(N+1)$ action. Let $N+1=\dim H^0(X,K_X^{\otimes(-r)})$ and we define an open subscheme
\begin{equation}\label{Z}
Z:=\left \{\hilb(X,D)\left | \mbox{\stackbox[v][m]{{\footnotesize
$(X,D)\subset \PP^N\times \PP^N$ be a {\em klt} pair with Hilbert polynomial $\bchi=(\chi,\ti\chi)$\\  satisfying: $D\subset X$,
$D\in|\!-\!mK_X\!|$ 
and  $\sO_{\PP^N}(1)|_X\cong K_X^{-\otimes r}$.
}}}\right .\right \}\subset \HH^{\bchi;N}.
\end{equation}
Let  $\lam_\CM\to Z$ (cf. {\cite{FS90}}, \cite[Definition 2.3]{FiRo06} or \cite[equation (2.4)]{PaTi06}) be the CM-line bundle over $Z$ normalized in such a way that the corresponding weight for any one parameter subgroup of $\SL(N+1)$ is exactly  the $\DF$ introduced in Definition \ref{be-K}, and
$$\lam_\chow(\cX):=\lam_\CH(\cX,\sO_\cX(-rK_\cX))\to Z\text{ and }\lam_\chow(\cD):=\lam_\CH(\cD,\sO_\cX(-rK_\cX)|_\cD)\to Z$$
be the Chow line bundles introduced in \cite[equation (3.3)]{FiRo06}) for the flat family $\cX\to Z$ and $\cD\to Z$ respectively.

\begin{proof}[Proof of Proposition \ref{p-con}]
Let   us introduce

$$
M_1 :=\lam_\CM \text{ and }  M_2:=\left(\lam_\chow(\cX)^{\otimes\frac{nm}{(n+1)r}}\otimes\lam_\chow(\cD)^{-1}\right)^{\otimes\frac{1}{2m r^n\cdot (-K_X)^n}}
  \text{ (cf.\eqref{CH})}.
$$

By Theorem \ref{unique}, we know $(\cX_t,\cD_t)$ is $\beta$-K-stable $\forall t\in C$ and $\beta\in (0,\beta_0]$. After removing finite number of points from $C$, we obtain a quasiprojective $0\in S\subset C$ over which $\pi_\ast \sO_\cX(-rK_{\cX/C})|_S\cong \sO_S^{\oplus N+1}$ .  By fixing a basis of  $\pi_\ast\sO_\cX(-rK_{\cX/C})|_S$, we obtain an embedding
$$\iota: (\cX, \cD;\sO_\cX(-rK_{\cX/C}))\times_C S\longrightarrow \PP^N\times \PP^N\times S$$
which in turn  induces an embedding $S\subset Z$ with $S$ being  constructible and $\vpi^{M_1,M_2}_{{\rm SL}(N+1)}\geq \be_0>0$.  By  applying Proposition \ref{be-constr} to $S\subset Z$, we obtain $\kst(\cX_t,\cD_t)=\vpi^{M_1,M_2}_{{\rm SL}(N+1)}(t),\ \forall t\in S$ is a constructible function. Our proof is completed.
\end{proof}
\begin{rem}\label{constr-Kpoly}
The above argument {first appeared in \cite{Pa2012} and\cite{Odaka14}  independently}, they observed that one can also conclude that  the K-polystable locus in $S$ is also constructible.
\end{rem}

\subsection{Stabilizer preserving and finite distance properties}\label{st-pre}
In this section, we will establish  the criteria that guarantee the {\em stabilizer preserving} condition and ingredients needed to prove the {\em existence of good moduli} for the closed substack $\overline{\{[z]\}}$ for any $\mathbb{C}$-point $[z]\in [Z^\ast/\SL(N+1)]$.
\subsubsection{Stabilizer preserving}
The following example indicates that stabilizer preserving condition is a condition of {\em properness} and cannot be deduced from the reductivity of stabilizer  alone.}
\begin{exmp}[Richardson's example]\label{Richard}
Consider $\SL(2,\CC)$-action on
$$\mathrm{Sym}^{\otimes 3} \CC^2=H^0(\sO_{\PP^1}(3))=\rm{Span}_\CC\{X^3,X^2Y,XY^2,Y^3\} $$
induced by the standard action on $\CC^2$.
 Then the stabilizer of $p_0(X,Y)=(X-Y)(X+Y)^2$ is trivial and the stabilizer of $p(X,Y)=(X-Y)(X-\omega Y)(X-\omega^2 Y)$ is given by
$$
\begin{bmatrix} \omega & 0\\ 0& \omega^{-1} \end{bmatrix}\in\SL(2,\CC)\text{ with } \omega^3=1.
$$
Let
$$
\al(t)
=\frac{1}{2}\begin{bmatrix}1 & 1\\ -1 & 1\end{bmatrix}\begin{bmatrix}t^2 &0\\ 0 & t^{-1}\end{bmatrix}\begin{bmatrix} 1 & -1\\ 1 & 1\end{bmatrix}\in \rm{GL}(2,\CC)
=\frac{1}{2}\begin{bmatrix}t^2+1/t & -t^2+1/t\\ -t^2+1/t & t^2+1/t\end{bmatrix}\in \rm{GL}(2,\CC)
$$
then
$$\al(t):\left\{
\begin{aligned}
X-Y &\longrightarrow & t^2(X-Y) \\
X+Y &\longrightarrow & t^{-1}(X+Y)
\end{aligned}
\right.
$$
hence  fixes $p_0(X,Y)=\frac{3}{4}(X-Y)(X+Y)^2\in\rm{Sym}^{\otimes 3} \CC^2$.
Now let us  define
\begin{eqnarray*}
&&p_t(X,Y)=p(\al(t)\cdot X,\al(t)\cdot Y)\\
&=&\frac{1}{4}t^2(X-Y)(t^2(1+\omega)+t^{-1}(1-\omega))X+(-t^2(1+\omega)+t^{-1}(1-\omega) )Y)\cdot \\
&&\cdot (t^2(1+\omega^2)+t^{-1}(1-\omega^2))X+(-t^2(1+\omega^2)+t^{-1}(1-\omega^2) )Y)\\
&=&\frac{1}{4}(X-Y)(t^3(1+\omega)+(1-\omega))X+(-t^3(1+\omega)+(1-\omega) )Y)\cdot \\
&&\cdot (t^3(1+\omega^2)+(1-\omega^2))X+(-t^3(1+\omega^2)+(1-\omega^2) )Y),
\end{eqnarray*}
then we have
$$\displaystyle \lim_{t\to 0}p_t(X,Y)=\frac{3}{4}(X-Y)(X+Y)^2$$
 and the  stabilizer  of $p_t$  is the subgroup $\langle \zeta_t:=\zeta_{p_t}\rangle\subset \rm{SL}(2)$ with
$$
\zeta_{p_t}:=\al(t^{-1})\begin{bmatrix}\omega &0\\ 0 & \omega^{-1}\end{bmatrix}\al(t)=\frac{1}{2}\begin{bmatrix}1 & 1\\ -1 & 1\end{bmatrix}\begin{bmatrix}\omega+\omega^{-1} & t^{-3}(\omega-\omega^{-1})\\ t^3(\omega-\omega^{-1}) & \omega+\omega^{-1}\end{bmatrix}\begin{bmatrix} 1 & -1\\ 1 & 1\end{bmatrix}\stackrel{t\to 0}{\longrightarrow }\infty ,
$$
In particular, the family of stabilizers $\langle \zeta_{t}\rangle\subset \rm{SL}(2,\CC)$ is {\em unbounded} as $t\to 0$ unless $\omega=1$.
\end{exmp}

Our goal here is to find conditions preventing the existence of  pathological examples like above. Let us collect some basic facts on compact Lie groups acting on $\PP^M$. Although our main application is to the situation in Summary \ref{sum-slice}, we will proceed in a more general fashion as it might be valuable for future applications.

Let $K$ be a compact Lie group \ and $\rho :K\rightarrow \SU( M+1)
$ be a linear representation and $\rho ^{\mathbb{C}}:G=K^{\mathbb{C}%
}\to \SL( M+1) $ be its complexification. Let $z_0\in
\PP^N$ with its stabilizer satisfying:
\begin{equation}\label{G_0-red}
G_{z_0}=(K_{z_0})^\CC:=( G_{z_0}\cap K) ^\CC.
\end{equation}
We fix a bi-invariant inner product $\langle \cdot ,\cdot
\rangle _{\fk}$ on $\fk$ and let $\fk_{z_0}^{\bot }\subset \fk$ denote the orthogonal complement  of $\fk_{z_0}=\mathrm{Lie}(K_{z_0})\subset\fk$ with
respect to $\langle \cdot ,\cdot \rangle _{\fk}.$ Then
the infinitesimal action $\si_{z_0}:\fg\longrightarrow T_{z_0}\PP^M$ is $G_{z_0}$-equivariant in the sense that
$$\si_{z_0}(\mathrm{Ad}_g\xi)=g\cdot \si_{z_0}(\xi)\ \text{ for all } g\in G_{z_0}\ , $$
and
there is a $G_{z_0}$-invariant linear subspace $W'\subset \CC^{M+1}$
such that
$$\CC^{M+1}=W\oplus ( \fk_{z_0}^{\bot})^\CC:= W'\oplus \CC \hat z_{0}\oplus ( \fk_{z_0}^{\bot})^\CC
\text{ with }( \fk^\bot_{z_0})^\CC:=\fk^\bot_{z_0}\otimes \CC$$
is a decomposition as $G_{z_{0}}$-module. Hence we have
\begin{equation}\label{PW}
\PP^M=\PP( W\oplus (
\fk_{z_0}^{\bot })^ \CC)=\PP( W'\oplus \CC \hat z_{0}\oplus (
\fk_{z_0}^{\bot })^ \CC),
\end{equation}
where $0\ne \hat z_0\in \CC^{M+1}$ is a lift of $z_0\in \PP^M$.


Consider the map%
\begin{equation}\label{phi}
\begin{array}{cccc}
G\times\PP W & \overset{\phi
}{ \xrightarrow{\hspace*{1.3cm}}} & G\cdot\PP W & \subset \PP^{M} \\
( g,w) & \longmapsto & g\cdot w &
\end{array}%
\end{equation}%
then for $\xi \in \fg_{z_0}$ and $\delta w\in
T_{z_{0}}\PP W$ we have
\begin{equation*}
d\phi|_{( e,z_{0}) }( \xi ,\delta w)
=\sigma _{z_{0}}( \xi ) +\delta w\in T_{z_{0}}\PP^{N}\cong  (\fk_{z_0}^\perp)^\CC\oplus T_{z_0}\PP W
\end{equation*}%
where $\sigma_{z_0}:\fg=\fk^\CC\to T_{z_0}\PP^M$ denotes the infinitesimal action and $e\in G$ denotes the identity,
and as a consequence $\ker d\phi|_{(e,z_0)}=\fg_{z_0}$.
Now let us define  an open set
$$U_0:=\left\{w\in \PP W \left | \mathrm{rk}\left(q\circ d\phi|_{\{1\}\times \PP W}: \fg\times T\PP W\to (T\PP^N|_{\PP W})/T\PP W\right)=\dim \fg_{z_0}^\perp \right. \right \}\subset \PP W $$
with $q: T\PP^N|_{\PP W}\to (T\PP^N|_{\PP W})/T\PP W$ being the quotient morphism between vector bundles over $\PP W$.  Then we have
\begin{lem}\label{G0-inv}
 $U_0\subset \PP W$ is a {\em $G_{z_0}$-invariant} Zariski open set.
\end{lem}

\begin{proof} Note that the Zariski openness follows from the fact that $q\circ d\phi\in H^0(\PP W, T(G\times \PP W)|_{\{1\}\times \PP W}^\vee\otimes (T\PP^N|_{\PP W})/T\PP W)$. So all we need is the $G_{z_0}$-invariance. To achieve that, one notices that for any $g\in G_{z_0},\ \xi\in \fg$ and $w\in \PP W$ we have
$$(g\cdot)_\ast \si_w(\xi)=\si_{g\cdot w}(\mathrm{Ad}_g\xi),$$
which implies that
$$\si_w(\xi)\in T_w\PP W\ \Longleftrightarrow \ \si_{g\cdot w}(\mathrm{Ad}_g\xi)\in T_{g\cdot w}\PP W\ .$$

Now  $w\in U_0$ can be characterized as  $q\circ d\phi$ being of full rank which is also equivalent to
\begin{equation}\label{tr-U0}
\si_w(\xi)\in T_w\PP W\Longleftrightarrow \xi\in \fg_{z_0}\ .
\end{equation}
If $g\cdot w\not\in U_0$ then there is a $0\ne\mathrm{Ad}_g\xi\in \fg_{z_0}^\perp$ such that $\si_{g\cdot w}(\mathrm{Ad}_g\xi)\in T_{g\cdot w}\PP W$, and hence $\si_w(\xi)\in T_w\PP W$.
On the other hand, we have decomposition  $\fg=\fg_{z_0}\oplus \fg_{z_0}^\perp$ as a {\em $G_{z_0}$-module} via the Adjoint action thanks to the {\em reductivity} of $G_{z_0}$. This implies that  $0\ne\xi\in \fg_{z_0}^\perp$, contradicting to \eqref{tr-U0} and the assumption that $w\in U_0$.  Thus our proof is completed.
\end{proof}

Now $\phi$ is {\em $G_{z_0}$-invariant} with respect to  the action $h\cdot (g,w)=(gh^{-1},h\cdot w)$, hence it  descends to a $K$-invariant map, which by abusing of notation is still denoted by
\begin{equation}\label{phi/}
\begin{array}{cccc}
G\times _{G_{z_{0}}}\PP W & \overset{\phi%
}{ \xrightarrow{\hspace*{1.3cm}}} & G\cdot\PP W & \subset \PP^{M} \ .\\
( g,w) & \longmapsto & g\cdot w &
\end{array}%
\end{equation}%
Moreover, it is a {\em bi-holomorphism} (see the proof of \cite[ Theorem 1.12]{Sj1995})
from a $K$-invariant tubular neighborhood
\begin{equation}\label{U-ep}
U_\ep :=\left\{ \left. ( g\exp \sqrt{-1}\xi ,w) \in G\times _{G_{z_0}}V\right |\ g\in K,\ \xi \in
\fk_{<\ep} \right \} \text{ with } \fk_{<\ep}:=\{\xi\in \fk\mid | \xi | <\ep  \}
\end{equation}
of  the orbit $K\cdot z_0\cong K/K_{z_0}$ onto $\phi(U_\ep)=K\cdot \exp\fk_{<\ep}\cdot V$ for $0<\ep\ll 1$, where $z_0\in V\subset \PP W$ is a $K$-invariant {\em analytic} open neighborhood.


Now suppose $\ti {g}=g\cdot\exp \sqrt{-1}\xi $ satisfies $g\in K$ and  $\xi\in \fk$ with $|\xi|
<\ep $ $\ $such that $\ti {g}\cdot w=w$ then:
\begin{equation*}
\phi( g\cdot \exp \sqrt{-1}\xi ,w) =\phi(
\ti {g},w) =\ti {g}\cdot w=w=\phi( e,w) \text{ and  } (\ti g,w)\in U_\ep
\end{equation*}
these together with the fact that $\phi|_{U_\ep}$ is bi-holomorphic imply that
$$( \ti {g},w) \stackrel{\scalebox{0.48}{$G_{z_0}$ } }{\sim} ( e,w)\in G\times\PP W\ $$
i.e. there is a $h\in G_{z_0}$ such that  $\ ( \ti {g} h^{-1},hw) =( e,w) $,
hence $\ti g=h\in G_{z_0}\cap G_w$.  In conclusion, we obtain the following:

\begin{lem}[{\bf Local Rigidity}]
\label{Gz0}Let  $w\in V\subset\PP W$ (defined in \eqref{U-ep}) and suppose $\ti g\in G_w$ is of the form
$\ti g=g\cdot \exp \xi$ with $g\in K$ and $\xi \in\fg$ satisfies $\ |\xi |<\ep $.  Then $\tilde{g}\in G_{z_0}.$
\end{lem}

\begin{assumption}[{\bf Properness}]\label{bdd}
There is a  {\em closed} $K$-invariant subset
$$
\xymatrix{
\Sigma\ \   \ar@{^{(}->}[r]   & \PP^M 
            }
$$
satisfying:
\begin{enumerate}
\item $\forall z\in \PP^M$, $(G\cdot z)\cap \Sigma$ consists  of {\em at most one} $K$-orbit.
$\Sigma$ is {\em continuous } in the sense that for any sequence  $\{z_i\}_{i=1}^\infty\subset\PP^M$ satisfying $(G\cdot z_i)\cap \Sigma\ne\varnothing$ and $\displaystyle\lim_{i\to \infty}z_i=z_\infty\in \Sigma$, we have
$$
\lim_{i\to \infty}\dist_{\PP^M}((G\cdot z_i)\cap \Sigma, K\cdot z_\infty)=0.
$$
\item
$G_z=( G_z\cap K)^{\CC}$ for all $z\in \Sigma .$
\end{enumerate}
\end{assumption}

\begin{thm}\label{S-pres}
Let $K$ be a compact Lie group acting on $\PP^{M}$ via a
representation $K\rightarrow \UU( M+1)$ and $G=K^\CC$ be its complexification, and $z_{0}\in \PP
^{M}$ with its stabilizer $G_{z_0}$ satisfying $G_{z_0}=( G_{z_0}\cap K) ^{\CC}$ and $z_{0}\in \Sigma \subset \PP^{M}$
satisfying {\rm Assumption} \ref{bdd}. Then \ there is an   $G_{z_0}$-invariant {\em Zariski open}  neighborhood $%
z_{0}\in U^{\rm sp}\subset \PP W$  such that for $\forall w\in U^{\rm sp}\cap (G\cdot \Sigma) $ we
have  $G_w<G_{z_0}.$
\end{thm}

\begin{proof}
We will first prove that our statement holds true for an {\em analytic} neighborhood, then we can pass from analytic open to a Zariski open neighborhood by the constructibility.

Suppose Assumption \ref{bdd} holds, then the {\em continuity} of $\Sigma$ implies that there is a sufficiently small  analytic $K_{z_0}$-invariant neighborhood  $z_0\in \ti V\subset V\subset \PP W$ such that for any $w\in\ti V\cap (G\cdot \Sigma)$,
there is a  $\xi \in ( \fk_{z_{0}}^{\bot }) ^{\CC}$
satisfying $|\xi | <\delta <\ep $ and $z\in
\Sigma $ such that $w=\exp \xi \cdot z.$ In particular, $\exp \xi \cdot
K_{z}\cdot \exp ( -\xi)\subset G_{w}$ is a maxmal
compact subgroup of $G_{w}$. \ Since $K_{z}<K$ is compact
we have
\begin{eqnarray*}
\exp \xi\cdot K_z\cdot \exp (-\xi )
&=&\{h\cdot \exp ( \mathrm{Ad}_{h^{-1}}\xi ) \cdot \exp ( -\xi )\mid h\in K_z\}\\
&\subset&\{ g\cdot\exp\sqrt{-1}\zeta | \zeta\in \fg,\ |\zeta|<\ep \text{ and } g\in K\}\ .
\end{eqnarray*}
By Lemma \ref{Gz0}, we
must have $\exp (-\xi)\cdot K_z\cdot \exp \xi  \subset G_{z_0}$.
Hence
$$G_{z_0}\supset \Big( \exp (-\xi)\cdot K_z\cdot \exp \xi  \Big)^{\CC}=G_w,$$
since $G_{z_0}$ is reductive.
Finally, one notices that the set
$$ \{w\in \PP W\mid G_w<G_{z_0}\}\supset G_{z_0}\cdot \ti V$$
is $G_{z_0}$-invariant and constructible.  This allows us to choose   a $G_{z_0}$-invariant Zariski open subset $U^{\rm sp}\supset G_{z_0}\cdot \ti V$, and our proof is completed.
\end{proof}

\begin{assumption}[{\bf Stabilizer Preserving}]\label{sp}
There is a $G_{z_0}$-invariant Zariski open neighborhood of $z_0\in U^\SP\subset\PP W$ such that $G_w<G_{z_0}$ for all $w\in U^\SP$.
\end{assumption}
\begin{exmp}Notice that Assumption \ref{sp}  does not hold in general, even in the situation that a 1-PS $\al(t)$ degenerating $\displaystyle \lim_{t\to 0} \al(t)\cdot z=z_0$, we cannot conclude that $G_{z_t}<G_{z_0}$.
Consider the $\SL(2)$-action on $\PP(\sym^{\otimes 3}\CC^2)$ as in Example \ref{Richard}. The 1-PS
 $$
\al(t)
=\frac{1}{2}\begin{bmatrix}1 & 1\\ -1 & 1\end{bmatrix}\begin{bmatrix}t &0\\ 0 & t^{-1}\end{bmatrix}\begin{bmatrix} 1 & -1\\ 1 & 1\end{bmatrix}
=\frac{1}{2}\begin{bmatrix}t+1/t & -t+1/t\\ -t+1/t & t+1/t\end{bmatrix}\in \rm{SL}(2,\CC)
$$
degenerates $p(X,Y)$ to $p_0(X,Y)\in \PP(\sym^{\otimes 3}\CC^2)$. Then $\ZZ/3\ZZ\cong\SL(2)_{p_t}\not\subset \SL(2)_{p_0}=\langle \al(t)\rangle\cong\GG_m$, and the map
$$
\SL(2)\times_{\GG_m}\PP W\longrightarrow \SL(2)\cdot\PP W
$$
is {\em not} finite.
\end{exmp}

\subsubsection{Finite distance property}
In this subsection, we establish the criteria that guarantees the properness of the map $\phi$ (defined  in \eqref{phi/}) near $z_0$, which is crucial to prove the existence of a good moduli space of  $\overline{\{[z]\}}\subset [Z^\ast/G]$ for any $[z]$ specializing to $[z_0]\in [Z^\ast/G]$ in Section \ref{ss-luna}.

Twisting the linearization of $G_{z_0}$ on $\sO_{\PP ^M}(1)|_{\PP W}$ by the inverse of the character corresponding to the action  $G_{z_0}\curvearrowright\sO_{\PP^M}(1)|_{z_0}$  as in the proof of Lemma \ref{disj}, we  obtain that $z_0\in \PP W$ is  GIT-polystable with respect to  the new  $G_{z_0}$-linearization on $\sO_{\PP ^M}(1)|_{\PP W}$. Let $U^\mathrm{ss}\subset \PP W$ denote the GIT-semistable points with respect to this linearization and
\begin{equation}\label{pi-w}
\pi_W: \PP W\supset U^\mathrm{ss}\longrightarrow \cM:=\PP W/\!\!/G_{z_0} \text{ with } \pi_W(z_0)=0\in \cM
\end{equation}
denote the GIT quotient map. Let $0\in B_\cM(0,r)\subset \cM$ be the {\em  open} ball of radius $r$ with respect to a prefixed continuous metric.
Then for  each $r>0$, we introduce
\begin{defn}\label{Ur}
 Let $U_r$  be the {\em connected component} of
 $$(G\cdot \pi_W^{-1}(B(0,r)))\cap \PP W\subset  U^\mathrm{ss} $$
 containing $z_0$.  In particular, $U_r$ is {\em $G_{z_0}$-invariant.}
\end{defn}

Let $[\cdot ]:G\to G/G_{z_0}$ denote the quotient map. We say a sequence $\{h_i\}\subset G$ is {\em bounded} in $G/G_{z_0}$ if and only if  $\{\psi^{-1}([h_i])\}$ is  contained in  a {\em bounded} subset of $K\times _{K_{z_{0}}}(\sqrt{-1}\fk_{z_0}^\perp)$, where $\psi$ is the Cartan decomposition (cf. \cite[equation (1.8)]{Sj1995})
\begin{equation}\label{bdd-G-G_0}
\begin{array}{cccc}
\psi:K\times _{K_{z_{0}}}(\sqrt{-1}\fk_{z_0}^\perp)& \overset{}{ \xrightarrow{\hspace*{1.3cm}}} & G/G_{z_0} &,\\
( g,\sqrt{-1}\xi ) & \longmapsto & (g\cdot\exp\sqrt{-1}\xi) \cdot G_{z_0} &
\end{array}%
\end{equation}
which is a $K$-equivariant {\em diffeomophism}.

\begin{assumption}[\bf Finite Distance]\label{nr}
 An analytic {\em open neighborhood} of $z_0\in U^\nr\subset \PP W$  is of {\em finite distance} if there is a {\em bounded} (in the sense above)
set  $G_{U^\nr}\Subset G/G_{z_0}$ depending only on $U^\nr$ and $z_0$ such that for any pair $(z, g)\in U^\nr\times G$ satisfying  $g\cdot z\in U^\nr $, there is an $h\in G,\  [h]\in G_{U^\nr}\Subset G/G_{z_0}$ such that $g\cdot z=h\cdot z$, where $[\cdot ]: G\to G/G_{z_0}$ is the {\em quotient} map, and $\Subset$ stands for the {\em compact embedding} with respect to the analytic topology.  It follows from the definition that $U^\nr$ is {\em $G_{z_0}$-invariant}.
\end{assumption}

\begin{lem} \label{fini}
Suppose both {\em Assumption \ref{sp} and \ref{nr}} are  satisfied. Then there is a  positive $\ep>0$ such that for any $0<r<\ep$,
$U_r$ (defined in Definition \ref{Ur}) satisfies the following: for any sequence
 $\{(g_i,y_i)\}\in G\times_{G_{z_0}}U_r$ satisfying $z_i=g_i\cdot y_i\to z_\infty\in G\cdot U_r$,
as $i\to \infty$, after passing to a subsequence, there is a
$$(g_\infty, y_\infty)\in \overline{\{(g_i,y_i)\}_i}\subset G\times_{G_{z_0}}U_r \text{ such that } g_\infty\cdot y_\infty=z_\infty.$$
In particular, the map $\phi|_{G\times_{G_{z_0}} U_r}:G\times_{G_{z_0}} U_r \to G\cdot U_r$ is a {\em finite} morphism.
\end{lem}
\begin{proof}

First, we notice that after translating $z_\infty$ by  a $g\in G$ if necessary, we may assume that $z_\infty\in U_r$.
Since $\overline{U}_r\subset\PP W$ is compact for  small $r>0$ by Definition \ref{Ur}, by passing to a subsequence we may and will assume $y_i\stackrel{i\to \infty}{\longrightarrow} z_\infty\in U_r$  after a possible decreasing of $r$.

By Assumption \ref{nr}, we may choose $0<r\ll 1$ such that $U_r\subset U^\nr$ then  there is a sequence $\{h_i\}\subset G$, with $\{[h_i]\}$  being bounded in $G/G_{z_0}$ and satisfing $g_{i}\cdot y_i=h_i\cdot y_i$, hence $h_i^{-1}\cdot g_i\in G_{y_i},\ \forall i$. Now by Assumption \ref{sp}, we have
$$h_i^{-1}\cdot g_i\in G_{y_i}< G_{z_0},\ \forall i,$$
from which we conclude that $\{[g_i]\}$ is {\em bounded} in $G/G_{z_0}$ and hence the set $\{(g_i,y_i)\}\subset G\times_{G_{z_0}} U_r$  is precompact.
Thus the morphism $\phi|_{G\times_{G_{z_0}} U_r}:G\times_{G_{z_0}} U_r \to G\cdot U_r$ is a proper and \'etale morphism hence finite.
\end{proof}
\begin{rem}\label{r-finitedistance}

Assumption \ref{nr} is introduced  to guarantee that the multiplication morphism
$$\phi|_{G\times_{G_{z_0}} U_r}:G\times_{G_{z_0}} U_r \to G\cdot U_r$$
is {\em proper}.
For that purpose, we want to make sure that  for $0<r\ll1$, any point $z\in U_r$ and any infinite sequence $\{g_i\}_{i=1}^\infty\subset G$ satisfying  $G/G_{z_0}\ni [g_i]\to \infty$ (with respect to the analytic topology) there is no  infinite recurrence of points  $g_i\cdot z$  inside $U_r$.
As we have seen in the proof of Lemma \ref{NR} that the {\em properness} of  slice $\Sigma$ obtained via Tian's embedding guarantee the Assumption \ref{nr}.  
\end{rem}

\vspace{.4cm}


\begin{bibdiv}
\begin{biblist}


\bib{AFSV14}{article}{
author={Alper, Jarod }
author={Fedorchuk, Maksym }
author={Smyth, David Ishii }
title={Second flip in the Hassett--Keel program: existence of good moduli spaces}
journal={Compos. Math.}
date={2017}
volume={153}
number={8}
pages={1584-1609}
}


\bib{Alp13}{article}{
author={Alper, Jarod }
title={Good moduli spaces for Artin stacks.}
journal={Ann. Inst. Fourier (Grenoble)}
date={2013}
volume={63}
number={6}
pages={2349-2042}
}




\bib{Aubin78}{article}{
author={Aubin, Thierry}
title={\'Equations du type Monge-Amp\`ere sur les vari\'et\'es k\"ahl\'eriennes compactes.}
journal={Bull. Sci. Math.(2)}
date={1978}
volume={102}
number={1}
pages={63-95}}


\bib{BBEGZ}{article}{
  author={Berman, Robert J.}
  author={Boucksom, S\'ebstien}
  author={Eyssidieux,Philippe}
  author={Guedj, Vincent}
  author={Zariahi, Ahmed}
  title={K\"ahler-Einstein metrics and the K\"ahler-Ricci flow on log Fano varieties}
  date={2011}
  journal={to appear in J. Reine Angew. Math., arXiv:1111.7158}
}

\bib{BCHM}{article}{
   author={Birkar, Caucher},
  author={Cascini, Paolo},
  author={Hacon, Christopher D.},
  author={McKernan, James},
  title={Existence of minimal models for varieties of log general type},
  journal={J. Amer. Math. Soc.},
   volume={23},
   date={2010},
   number={2},
   pages={405--468},
   }


\bib{BBer2015}{article}{
  author={Berndtsson, Bo}
  title={A Brunn-Minkowski type inequality  for Fano manifolds and some uniqueness theorems in K\"ahler geometry}
  JOURNAL = {Invent. Math.},
  FJOURNAL = {Inventiones Mathematicae},
    VOLUME = {200},
      YEAR = {2015},
    NUMBER = {1},
     PAGES = {149--200},
      ISSN = {0020-9910},
   MRCLASS = {53C55 (32Q20 53C25)},
  MRNUMBER = {3323577},
MRREVIEWER = {Vincent Guedj},
       URL = {http://dx.doi.org/10.1007/s00222-014-0532-1},
}


\bib{Ber12}{article}{
  author={Berman, Robert J.}
    TITLE = {K-polystability of {${\Bbb Q}$}-{F}ano varieties admitting
              {K}\"ahler-{E}instein metrics},
   JOURNAL = {Invent. Math.},
  FJOURNAL = {Inventiones Mathematicae},
    VOLUME = {203},
      YEAR = {2016},
    NUMBER = {3},
     PAGES = {973--1025},
      ISSN = {0020-9910},
   MRCLASS = {53C55 (14J45 53C25)},
  MRNUMBER = {3461370},
       URL = {http://dx.doi.org/10.1007/s00222-015-0607-7},
}
		

\bib{BG13}{article}{
  author={Berman, Robert J.}
   author={Guenancia, Henri}
  title={K\"ahler-Einstein metrics on stable varieties and log canonical pairs}
 JOURNAL = {Geom. Funct. Anal.},
  FJOURNAL = {Geometric and Functional Analysis},
    VOLUME = {24},
      YEAR = {2014},
    NUMBER = {6},
     PAGES = {1683--1730},
      ISSN = {1016-443X},
   MRCLASS = {53C25 (14C20 32Q20 53C55)},
  MRNUMBER = {3283927},
MRREVIEWER = {V. V. Chueshev},
       URL = {http://dx.doi.org/10.1007/s00039-014-0301-8},
}


\bib{CDS1}{article}{
 author={Chen, Xiuxiong}
 author={Donaldson, Simon}
 author={Sun, Song}
 title={K\"ahler-Einstein metrics on Fano manifolds. I: Approximation of metrics with cone singularities},
 journal={ J. Amer. Math. Soc.  },
 volume={28},
 number={1},
 date={2015},
 pages={183-197},
 }


\bib{CDS2}{article}{
 author={Chen, Xiuxiong}
 author={Donaldson, Simon}
 author={Sun, Song}
 title={K\"ahler-Einstein metrics on Fano manifolds. II: Limits with cone angle less than $2\pi$},
 journal={ J. Amer. Math. Soc.  },
 volume={28},
 number={1},
 date={2015},
 pages={199-234},
 }


\bib{CDS3}{article}{
 author={Chen, Xiuxiong}
 author={Donaldson, Simon}
 author={Sun, Song}
 title={K\"ahler-Einstein metrics on Fano manifolds. III: Limits as cone angle approaches $2\pi$ and completion of the main proof},
 journal={ J. Amer. Math. Soc.  },
 volume={28},
 number={1},
 date={2015},
 pages={235-278},
 }



 \bib{CS2014}{article}{
 author={Chen, Xiuxiong}
 author={Sun, Song}
 title={Calabi flow, geodesic rays, and uniqueness of constant scalar curvature K\"ahler metrics.}
 journal={Ann. of Math. (2)}
 volume={180}
 number={2}
 year={2014}
 pages={407-454}

 }


	
	
 \bib{Don01}{article}{
author={Donaldson, Simon K.},
title={Scalar curvature and projective embeddings. I.},
journal={J. Differential Geom.},
volume={59},
year={2001},
number= {3},
pages={479-522},
}


 \bib{Don00}{article}{
 author={Donaldson, Simon K.},
 title={Scalar curvature and stability of toric varieties.}
 journal={J. Differential Geom.},
 volume={62},
 year={2002},
 pages={289-349},
 }
	

\bib{Don2008}{article}{
author={Donaldson, Simon K.},
title={K\"ahler geometry on toric manifolds, and some other manifolds with large symmetry.}
note={Adv. Lect. Math. (ALM), Handbook of geometric analysis. No. 1, Int. Press, Somerville, MA}
volume={7}
year={2008}
pages={29-75}
}


\bib{Don2009}{article}{
author={Donaldson, Simon K.},
title={Discussion of the K\"ahler-Einstein problem.},
note={http://www.imperial.ac.uk/~skdona/KENOTES.PDF}
year={2009}
}


\bib{Don2011}{article}{
   author={Donaldson, Simon K.},
   title={K\"ahler metrics with cone singularities along a divisor},
   journal={Essays in mathematics and its applications, Springer, Heidelberg. }
   pages={49-79}
   date={2012}
    }



\bib{Don2012}{article}{
   author={Donaldson, Simon K.},
   title={Stability, birational transformations and the K\"ahler-Einstein problem},
   journal={Surv. Differ. Geom.},
   number={17},
   date={2012},
   pages={203-228},
}

\bib{Don2013}{article}{
author={Donaldson, Simon K.},
     TITLE = {Algebraic families of constant scalar curvature {K}\"ahler
              metrics},
 BOOKTITLE = {Surveys in differential geometry 2014. {R}egularity and
              evolution of nonlinear equations},
    SERIES = {Surv. Differ. Geom.},
    VOLUME = {19},
     PAGES = {111--137},
 PUBLISHER = {Int. Press, Somerville, MA},
      YEAR = {2015},
   MRCLASS = {53C55},
  MRNUMBER = {3381498},
       URL = {http://dx.doi.org/10.4310/SDG.2014.v19.n1.a5},
}

\bib{Dre2004}{article}{
author={Dr\'ezet, Jean-Marc}
title={Luna's slice theorem and applications},
booktitle={Algebraic group actions and quotients,\ Hindawi Publ. Corp. Cairo},
date={2004},
pages={39-89},
publisher={Hindawi Publ. Corp. Cairo},
}

\bib{DS2012}{article}{
   author={Donaldson, Simon K.}
   author={Sun, Song},
   title={Gromov-Hausdorff limits of K\"ahler manifolds and algebraic geometry},
   journal={Acta Math. },
   volume={213}
   number={1},
   date={2014},
   pages={63-106},
   issn={0073-8301},
}





 \bib{FiRo06}{article}{
 author={Fine, Joel}
 author={Ross, Julius},
 title={A note on positivity of the CM line bundle. },
 journal={Int. Math. Res. Not., Art. ID 95875,}
 volume= {}
 number={},
 pages={14 pp.},
 date={2006}

 }

\bib{FS90}{article}{
author={Fujiki,A.},
author={Schumacher,G.},
title={The moduli space of extremal compact K\"{a}hler
manifolds and generalized Weil-Petersson metrics},
journal={Publ. Res. Inst. Math.},
volume={26},
year={1990},
pages={101-183},
}


\bib{Har77}{book}{
    AUTHOR = {Hartshorne, Robin},
     TITLE = {Algebraic geometry},
      NOTE = {Graduate Texts in Mathematics, No. 52},
 PUBLISHER = {Springer-Verlag, New York-Heidelberg},
      YEAR = {1977},
     PAGES = {xvi+496},
 }
	

\bib{HMX2014}{article}{
 author={Hacon, Christopher}
 author={McKernan, James}
 author={Xu, Chenyang}
 title={ACC for log canonical thresholds},
 journal={Ann. of Math. (2)},
 volume={180},
 number={2},
 date={2014},
 pages={523-571},
}


\bib{HX11}{incollection}{
     author={Hacon, Christopher}
 author={Xu, Chenyang}
   title={On Finiteness of B-representations and Semi-log Canonical Abundance},
 BOOKTITLE = {Minimal models and extremal rays (Kyoto, 2011)},
    SERIES = {Adv. Stud. Pure Math.},
    VOLUME = {70},
     PAGES = {361--378},
 PUBLISHER = {Math. Soc. Japan, Tokyo},
      YEAR = {2016},
 }


\bib{Kol96}{book}{
    AUTHOR = {Koll{\'a}r, J{\'a}nos}
     TITLE = {Rational curves on algebraic varieties},
    SERIES = {Ergebnisse der Mathematik und ihrer Grenzgebiete. 3. Folge. A
              Series of Modern Surveys in Mathematics [Results in
              Mathematics and Related Areas. 3rd Series. A Series of Modern
              Surveys in Mathematics]},
    VOLUME = {32},
 PUBLISHER = {Springer-Verlag, Berlin},
      YEAR = {1996},
     PAGES = {viii+320},
  }




\bib{Kollar13}{article}{
    AUTHOR = {Koll{\'a}r, J{\'a}nos},
     TITLE = {Moduli of varieties of general type},
 BOOKTITLE = {Handbook of moduli. {V}ol. {II}},
    SERIES = {Adv. Lect. Math. (ALM)},
    VOLUME = {25},
     PAGES = {131--157},
 PUBLISHER = {Int. Press, Somerville, MA},
      YEAR = {2013},
      }




\bib{KM98}{book}{
   author={Koll{\'a}r, J{\'a}nos}
   author={Mori, Shigefumi}
   title={Birational geometry of algebraic varieties},
   series={Cambridge Tracts in Mathematics},
   volume={134},
   note={With the collaboration of C. H. Clemens and A. Corti;
   Translated from the 1998 Japanese original},
   publisher={Cambridge University Press},
   place={Cambridge},
   date={1998},
   pages={viii+254},
}

\bib{KMM92}{article}{
    AUTHOR = {Koll{\'a}r, J{\'a}nos}
    author={Miyaoka, Yoichi}
    author={ Mori, Shigefumi},
     TITLE = {Rational connectedness and boundedness of {F}ano manifolds},
   JOURNAL = {J. Differential Geom.},
  FJOURNAL = {Journal of Differential Geometry},
    VOLUME = {36},
      YEAR = {1992},
    NUMBER = {3},
     PAGES = {765--779},
}

\bib{KSB}{article}{
author={Koll\'ar, J\'anos},
author={Shepherd-Barron, N. I.},
title={Threefolds and deformations of surface singularities},
journal={Invent. Math. },
volume={91},
year= {1988},
number={2},
pages={299-338},
}



\bib{LS2014}{article}{
   author={Li, Chi},
   author={Sun, Song},
   title={Conical K\"ahler-Einstein metrics revisited},
    journal={Comm. Math. Phys.},
    volume={331}
   number={3},
   date={2014},
   pages={927-973},
}

\bib{LW2015}{article}{
author={Li, Jun},
author={ Wang, Xiaowei},
title={Hilbert-Mumford criterion for nodal curves},
journal={Compos. Math.},
volume={151},
date={2015},
pages={2076-2130},
}



\bib{LWX2014}{article}{
   author={Li, Chi},
   author={Wang, Xiaowei}
   author={Xu, Chenyang},
   title={Degeneration of Fano K\"ahler-Einstein manifolds},
   journal={ArXiv:1411.0761 v1},
   date={2014}
}


\bib{LWX2018}{article}{
   author={Li, Chi},
   author={Wang, Xiaowei}
   author={Xu, Chenyang},
   title={Quasi-projectivity of the moduli space of smooth K\"{a}hler-Einstein Fano manifolds},
   journal={Ann. Sci. \'Ecole Norm. Sup. (4)},
   volume={51},
   date={2018},
   number={3},
   pages={739-772},
}



\bib{LX}{article}{
   author={Li, Chi},
   author={Xu, Chenyang},
   title={Special test configurations and K-stability of Fano varieties},
   journal={Annals of Math.},
   volume={180}
   number={1},
   pages={197-232},
   date={2014}
}

\bib{MFK}{book}{
  author={Mumford, D.},
  author={Forgarty, J.},
  author={Kirwan, F.},
  title={Geometric Invariant Theory},
 series={Ergebnisse der Mathematik und ihrer Grenzgebiete, Vol. 34},
  publisher={Springer-Verlag},
 place={Berlin},
 date={1994},
 pages={xiv+292},
 }



\bib{Odaka13}{article}{
 author={Odaka, Yuji}
 title={The GIT stability of polarized varieties via discrepancy},
 journal={ Annals of Math. },
 volume={177},
 number={2},
 date={2013},
 pages={171-185},
 issn={}
 }


\bib{Odaka14}{article}{
 author={Odaka, Yuji}
 title={On the moduli of K\"ahler-Einstein Fano manifolds},
 journal={ Proceeding of Kinosaki algebraic geometry symposium 2013, arXiv:1211.4833. },
 volume={},
 number={},
 date={2013},
 pages={},
 issn={}
 }

\bib{Odaka14a}{article}{
 author={Odaka, Yuji}
 title={Compact moduli space of K\"ahler-Einstein Fano varieties},
    JOURNAL = {Publ. Res. Inst. Math. Sci.},
  FJOURNAL = {Publications of the Research Institute for Mathematical            Sciences},
    VOLUME = {51},
     YEAR = {2015},
    NUMBER = {3},
    PAGES = {549--565},
      ISSN = {0034-5318},
   MRCLASS = {14D20 (14J45 32Q20)},
  MRNUMBER = {3395458},
MRREVIEWER = {Francesco Bottacin},
       DOI = {10.4171/PRIMS/164},
       URL = {http://dx.doi.org/10.4171/PRIMS/164},
 }


\bib{OSS}{article}{
 author={Odaka, Yuji}
  author={Spotti, Cristiano }
   author={Sun, Song}
     TITLE = {Compact moduli spaces of {D}el {P}ezzo surfaces and
              {K}\"ahler-{E}instein metrics},
   JOURNAL = {J. Differential Geom.},
  FJOURNAL = {Journal of Differential Geometry},
    VOLUME = {102},
      YEAR = {2016},
    NUMBER = {1},
     PAGES = {127--172},
      ISSN = {0022-040X},
   MRCLASS = {53C55 (14J10 32Q25 53C25 58D27)},
  MRNUMBER = {3447088},
       URL = {http://projecteuclid.org/euclid.jdg/1452002879},
}
	

\bib{OSY}{article}{
    AUTHOR = {Ono, Hajime}
    author={Sano, Yuji}
    author={Yotsutani, Naoto},
     TITLE = {An example of an asymptotically {C}how unstable manifold with
              constant scalar curvature},
   JOURNAL = {Ann. Inst. Fourier (Grenoble)},
    VOLUME = {62},
      YEAR = {2012},
    NUMBER = {4},
     PAGES = {1265--1287},
   }


\bib{Pa2012}{article}{
 author={Paul, Sean}
 title={CM stability of projective varieties},
 journal={arXiv:1206.4923 },
 volume={},
 number={},
 date={2012},
 }
	

\bib{PaTi06}{article}{
 author={Paul, Sean}
 author={Tian, Gang}
 title={CM stability and the generalized Futaki invariant I.},
 journal={arXiv:math/0605278 },
 volume={},
 number={},
 date={2006},
 issn={0003-486X}
 }


\bib{PS2010}{book}{
author={Phong, D. H.}
author={Sturm, Jacob}
title={Lectures on stability and constant scalar curvature }
volume={14}
series= {Handbook of  geometric analysis, Adv. Lect. Math. (ALM)}
publisher={Int. Press}
 year= {2010},
}

\bib{Sj1995}{article}{
author={Sjamaar, Reyer}
title={Holomorphic Slices, Symplectic Reduction and Multiplicities of Representations}
volume={131},
number={1},
journal={Ann. of Math. (2)}
year={1995}
pages={87-129},
}



\bib{Sp2012}{article}{
author={Spotti, Cristiano}
title={Degenerations of K\"ahler-Einstein Fano Manifolds}
journal={Ph.D. Thesis: arXiv:1211.5334}
year={2012}
}



\bib{SSY}{article}{
author={Sun, Song},
author={Spotti, Cristiano},
author={Yao, Chengjian},
title={Existence and deformations of Kahler-Einstein metrics on smoothable $\QQ$-Fano varieties},
journal={Duke Math. J.},
year={2018},
volume={165}
number={16}
pages={3043-3083}
}


\bib{SW2012}{article}{
 author={Song, Jian}
 author={Wang, Xiaowei}
 title={The greatest Ricci lower bound, conical Einstein metrics and the Chern number inequality}
 journal={Geom. Topol.},
year={2016},
volume={20}
number={1}
pages={49--102 }
 }


\bib{Sz2010}{article}{
author={Sz\'ekelyhidi, G\'abor}
title={The K\"ahler-Ricci flow and K-polystability}
volume={132}
number={4}
journal={Amer. J. Math. }
year={2010}
pages={1077-1090}
}

 \bib{Th2006}{article}{
author={Thomas, Richard},
title={Notes on GIT and symplectic reduction for bundles and varieties},
journal={Surveys in Differential Geometry, Int. Press. },
volume={10},
pages={221-273},
year={2006}
}

\bib{Tian1990}{article}{
author={Tian, Gang}
title={On Calabi's conjecture for complex surfaces with positive first Chern class.}
journal={Invent. Math.}
volume={101}
year={1990}
number={1}
pages={101-172}
}

\bib{Tian1997}{article}{
author={Tian, Gang},
title={K\"{a}hler-Einstein metrics with positive scalar curvature},
journal={Invent. Math.},
volume={130},
year={1997},
pages={1-39},
}

\bib{Tian2012}{article}{
author={Tian, Gang},
title={Existence of Einstein metrics on Fano manifolds},
journal={Metric and Differential Geometry,}
volume={297},
year={2012},
pages={119-159},
}

\bib{Tian13}{article}{
    AUTHOR = {Tian, Gang},
     TITLE = {Partial {$C^0$}-estimate for {K}\"ahler-{E}instein metrics},
   JOURNAL = {Commun. Math. Stat.},
  FJOURNAL = {Communications in Mathematics and Statistics},
    VOLUME = {1},
      YEAR = {2013},
    NUMBER = {2},
     PAGES = {105--113},}
	

\bib{Tian2014}{article}{
author={Tian, Gang},
title={K-stability and K\"{a}hler-Einstein metrics},
     TITLE = {K-stability and {K}\"ahler-{E}instein metrics},
   JOURNAL = {Comm. Pure Appl. Math.},
  FJOURNAL = {Communications on Pure and Applied Mathematics},
    VOLUME = {68},
      YEAR = {2015},
    NUMBER = {7},
     PAGES = {1085--1156},
      ISSN = {0010-3640},
   MRCLASS = {53C55 (53C25)},
  MRNUMBER = {3352459},
MRREVIEWER = {Matthew B. Stenzel},
       URL = {http://dx.doi.org/10.1002/cpa.21578},
}




\bib{Woo10}{article}{
author={Woodward, Christopher},
title={Moment maps and geometric invariant theory-Corrected version (October 2011)},
journal={ Les cours du CIRM},
volume={1},
number={ 1}
issue={1},
pages={121-166},
year={2010}
note={Available at \url{http://ccirm.cedram.org/cedram-bin/article/CCIRM_2010__1_1_121_0.pdf}}}


\bib{WX}{article}{
author={Wang, Xiaowei},
author={Xu, Chenyang}
title={Nonexistence of aymptotic GIT compactification},
journal={ Duke Math. J. },
volume={163},
issue={12},
pages={2217-2241},
year={2014}}


 \bib{Vie83}{book}{
  AUTHOR = {Viehweg, Eckart},
     TITLE = {Quasi-projective moduli for polarized manifolds},
    SERIES = {Ergebnisse der Mathematik und ihrer Grenzgebiete (3) [Results
              in Mathematics and Related Areas (3)]},
    VOLUME = {30},
 PUBLISHER = {Springer-Verlag, Berlin},
      YEAR = {1995},
 }


\bib{Yau78}{article}{
author={Yau, Shing-Tung},
title={On the Ricci curvature of a compact Kähler manifold and the complex Monge-Ampère equation. I.},
journal={ Comm. Pure Appl. Math. },
volume={31},
number={3},
pages={339-441},
year={1978}}






\end{biblist}
\end{bibdiv}
\end{document}